\documentclass[11pt]{book}
\usepackage[english]{babel}
\usepackage{amsmath}
\usepackage{amsthm}
\usepackage{amssymb}
\usepackage[dvips]{color}
\newtheorem{theorem}{Theorem}[section]
\newtheorem{definition}[theorem]{Definition}
\newtheorem{lemma}[theorem]{Lemma}
\newtheorem{proposition}[theorem]{Proposition}
\newtheorem{corollary}[theorem]{Corollary}
\newtheorem{remark}[theorem]{Remark}
\newtheorem{example}[theorem]{Example}

\newcommand{\rr}{\mathbb{R}}
\newcommand{\cc}{\mathbb{C}}
\newcommand{\hh}{\mathbb{H}}

\newcommand{\s}{\mathbb{S}}

\newcommand{\pp}{\partial}

\newcommand{\ext}{{\rm ext  }}

\title{\bf  Entire slice regular functions}

\author{F. Colombo, I. Sabadini, D.C. Struppa}

\makeindex

\begin{document}

\pagenumbering{arabic}
\maketitle


\tableofcontents

\newpage

{\bf Abstract}.  Entire functions in one complex variable are extremely
relevant in several areas ranging from the study of convolution equations to special functions. An analog of
entire functions in the quaternionic setting can be defined in the slice
regular setting, a framework which includes polynomials and power series
of the quaternionic variable. In the first chapters of this work we introduce and discuss the
algebra and the analysis of slice regular functions. In addition to offering a self-contained introduction to the theory of slice-regular functions, these chapters also contain a few new results (for
example we complete the discussion on lower bounds for slice regular
functions initiated with the Ehrenpreis-Malgrange, by
adding a brand new Cartan-type theorem).
    The core of the work is Chapter 5, where we study the growth of
entire slice regular functions, and we show how such growth is related
to the coefficients of the power series expansions that these functions
have. It should be noted that the proofs we offer are not simple
reconstructions of the holomorphic case. Indeed, the non-commutative
setting creates a series of non-trivial problems.
Also the counting of the zeros is not trivial because of the presence of
spherical zeros which  have infinite cardinality. We prove the analog of
Jensen and Carath\'eodory theorems in this setting.

\newpage

\chapter{Introduction}

The theory of holomorphic functions in one complex variable assumes a particular flavor when one considers functions that are holomorphic on the entire plane, namely the entire functions. The reasons for the richness of results that are possible in that context are many, but certainly include the fact that for entire functions we can study growth phenomena in a cleaner way than what would be possible if one were to consider the issues introduced by the boundary of the domain of holomorphy associated to an individual function. While the beautiful issues connected with analytic continuation do not arise in the context of entire functions, the theory acquires much strength from the ability of connecting the growth of the functions to the coefficients that appear in their Taylor series.

The study of entire functions, in addition, has great relevance for the study of  convolution equations, where different spaces of entire functions arise naturally from the application of the Paley Wiener theorem (in its various forms), and from the topological vector space approach that was so instrumental in the work of Schwarz, Malgrange, Ehrenpreis, Palamodov, H\"ormander.

Further arguments in support of the study of entire functions are supplied by the consideration that not only the most important elementary functions are entire (polynomials, exponentials, trigonometric and hyperbolic functions), but also that many of the great special functions of analysis (the Jacobi theta function, the Weierstrass sigma function, the Weierstrass doubly periodic functions, among others) are entire. Finally, one knows that the solutions of linear differential equations with polynomial coefficients and with constant highest derivative coefficient are entire as well (this includes the elementary functions, as well as the Airy functions, for example).

While many texts exist, that are devoted exclusively to the study of entire functions (see for example  \cite{boas, gil,levin}), one of the best ways to understand the importance and the beauty of entire functions is the article by de Branges \cite{debranges}.

One might therefore ask whether an analogous analysis can be done for functions defined on the space $\hh$ of quaternions. As it is well known, there are several different definitions of holomorphicity (or regularity, as it is often referred to) for functions defined on quaternions. The most well known is probably the one due to Fueter, who expanded on the work of Moisil, to define regular functions those which satisfy a first order system of linear differential equations that generalizes the one of Cauchy-Riemann. These functions are often referred to as Fueter regular, and their theory is very well developed, see e.g. \cite{csss,ghs}. It is therefore possible to consider those functions which are Fueter regular on the entire quaternionic space $\hh$, and see whether there are important properties that can be deduced. However, the mathematician interested in this generalization would encounter immediately the problem that these functions do not admit a natural power series expansion. While this comment would need to be better qualified, it is the basic reason why most of what is known for holomorphic entire functions cannot be extended to the case of Fueter regular entire functions.

In the last ten years, however, many mathematicians have devoted significant attention to a different notion of regularity, known as slice regularity or slice hyperholomorphy.
This theory  began as a theory of functions from the space of quaternions $\mathbb H$ to itself whose restriction to any complex plane contained in $\mathbb H$ was in the kernel of the corresponding Cauchy-Riemann operator, see \cite{MR2227751, GS}.
It immediately appeared that this class of functions, at least on balls centered at the origin, coincides with the class of polynomials and convergent power series of the quaternionic variable, previously introduced in \cite{deleo}.
In particular, there are very natural generalizations of the exponential and trigonometric functions, that happen to be entire slice regular functions.
 Further studies showed that on suitable open sets called axially symmetric slice domains, this class of function coincides with the class of functions of the form $f(q)=f(x+Iy)=\alpha(x,y)+I\beta(x,y)$ when the quaternion $q$ is written in the form $x+Iy$ ($I$ being a suitable quaternion such that $I^2=-1$) and the pair $(\alpha, \beta)$ satisfies the Cauchy-Riemann system and the conditions $\alpha(x,-y)=\alpha(x,y)$, $\beta(x,-y)=-\beta(x,y)$. This class of functions when $\alpha$ and $\beta$ are real quaternionic or, more in general, Clifford algebra valued is well known: they are the so-called holomorphic functions of a paravector variable, see \cite{ghs, sommen}, which were later studied in the setting of real alternative algebras in \cite{GP1}.\\

But there are deeper reasons why these functions are a relevant subject of study. Maybe the most important point is to notice that slice regular functions, and their Clifford valued companions, the slice monogenic functions (see \cite{MR2520116, CSSd}), have surprising applications in operator theory. In particular, one can use these functions (collectively referred to as slice hyperholomorphic functions) to define a completely new and very powerful slice hyperholomorphic functional calculus (which extends the Riesz-Dunford functional calculus to quaternionic operators and to $n$-tuples of non-commuting operators).
 These possible applications in operator theory gave great impulse to the theory of functions.

In particular, their Cauchy formulas, with slice hyperholomorphic kernels, are the basic tools to extend the Riesz-Dunford functional calculus to quaternionic operators and to $n$-tuples of not necessarily commuting operators.  They also lead to the notion of $S$-spectrum which turned out to be
the correct notion of spectrum for all applications in different areas of quaternionic functional analysis and in
particular to the quaternionic spectral theorem.
\\
The function theory of slice regular and slice monogenic functions was developed in a number of papers, see the list of references (note that some of the references are not explicitly quoted in the text) and the comments below. It has also been extended to vector-valued and more in general operator-valued functions, see \cite{ACSBOOK} and the references therein.
The monographs \cite{ACSBOOK,CSS,GSS} are the main sources for slice hyperholomorphic functions and their applications.

It should be pointed out that the theory of slice hyperholomorphic functions is different from the more classical theory of functions in the kernel of the Dirac operator \cite{bds,csss,DeSoS}, the so-called monogenic functions. While the latter is a refinement of harmonic analysis is several variables, the former has many applications among which one of the most important is in quaternionic quantum mechanics, see \cite{adler,BvN,12,14,21}.
Slice hyperholomorphic functions and monogenic functions can be related using the Fueter theorem, see \cite{CGONZSAB,CLSSo, COSASO} and the Radon transform, see \cite{So1} and the more recent \cite{bls}.
\\
\\
To give a flavor of the several results available in the literature in the context of slice regular functions, with no claim of completeness,  we gather the references in the various areas of research.
\\
\\
 {\it Function theory}. The theory of slice regular functions was delevoped in the papers
\cite{CGeSA,GSTO1,  GSTO3, GStruPRAG, MR2227751, GS,  MR2869150, MR2836832}, in particular, the zeros were treated in \cite{GSTO0,zeri,mjm}
while further properties  can be found in
\cite{BISISTOPP,CGJ2, DRGS, ggs, GGONS2, GSARF,RW,RZ,STOPP3, STOPP1,STOPP2,vlac}.
Slice monogenic functions with values in a Clifford algebra and their main properties were studied in
\cite{MR2742644,MR2520116,CSSd,duality,MR2785869,sheaves,CDB,GP7,YQ}.
Approximation of slice hyperholomorphic functions are collected in the works \cite{runge,GALSAB1,GALSAB2,GALSAB3,GALSAB4,GALSAB5,GalSa1,galsa3, sarfatti1}.
The case of several variables was treated in
\cite{AGALCOSA,CSSsev}, but a lot of work has still to be done in this
direction since the theory is at the beginning.
The generalization of slice regularity to real alternative algebra was developed in the papers \cite{alt,GP1,gp,GP3,GP4} and, finally,
while some results associated with the global operator associated to slice
 hyperholomorphicity are in the papers \cite{CGOZSAB,CoSo,CDEBR,GP5}.
\\
\\
{\it Function spaces}. Several function spaces have been studied in this framework. In particular, the quaternionic Hardy spaces $H^2(\Omega)$, where $\Omega$ is the quaternionic
unit ball $\mathbb B$ or the half space
$\mathbb H^+$ of quaternions with positive real part, together with the Blaschke products are in  \cite{acls_milan,acs1, MR3127378} and further properties are in \cite{arc, arc2,arc3}.
 The Hardy spaces $H^p(\mathbb B)$, $p>2$, are considered in \cite{sarfatti}.
 The Bergman spaces can be found in \cite{cglss, CGS, CGS3} and  the Fock space in \cite{Fock}.
Weighted Bergman spaces, Bloch, Besov, and Dirichlet spaces on the unit ball $\mathbb B$ are considered in
\cite{MR3311947}.
Inner product spaces and Krein spaces in the quaternionic setting, are in \cite{ALLJECK5}.
\\
\\
{\it Groups and semigroups of operators}.
The theory of groups and semigroups of quaternionic operators has been   developed and studied in the papers
\cite{perturbation, MR2803786,GR}.
\\
\\
{\it Functional calculi}.
 There exists at least five functional calculi associated to slice hyperholomorphicity.
For each of it we have the quaternionic version and the version for $n$-tuples of operators.
The $S$-functional calculus, see \cite{acgs,cgssann,cgss,cgss1,JGA,CLOSED,Cauchy,JFA},
 is the analogue of the Riesz-Dunford functional calculus in the quaternionic setting.
Further developments are in  \cite{FRCPOW, Taylor}.
The $SC$-functional calculus, see \cite{SCFCALC}, is the commutative version of the $S$-functional calculus.
For the functional calculus for groups of quaternionic operators based in the Laplace-Stieltjes transform, see \cite{FUNCGEN}.
The $H^\infty$ functional calculus based on the $S$-spectrum, see \cite{HINFTY},
 is the analogue, in this setting, of the calculus introduced by A. McIntosh, see \cite{McI1}.
The F-functional calculus, see \cite{ACSMATMET, Fconsequences,Boundedpert,  FUNBOUNDED, FFRANK},
which is based on the Fueter mapping theorem in integral form,  is a
monogenic functional calculus in the spirit of the one developed in
 \cite{jefferies,jmc,jmcpw,mcp,LIATAO}, but it is associated to slice hyperholomorphicity.
Finally the W-functional calculus, see \cite{WCALCOLO}, is a monogenic plane wave calculus based on slice hyperholomorphic functions.
\\
\\
{\it Spectral theory}.
The spectral theorem based on the $S$-spectrum for bounded and for unbounded quaternionic normal operators on a Hilbert space  was developed in \cite{ack,acks2,spectcomp}.
The case of quaternionic normal matrices was proved in \cite{fp} and it is  based on the right spectrum,
but the right spectrum is equal to the $S$-spectrum in the finite dimensional case.
The Continuous slice functional calculus in quaternionic Hilbert spaces is studied in \cite{GMP}.
\\
\\
{\it Schur Analysis}. This is a very wide field that has been developed in the last five years in the slice hyperholomorphic setting.
Schur analysis originates with the works of Schur, Herglotz, and others
 and can be seen as a collection of
topics pertaining to Schur functions and their  generalizations; for a quick introduction in the classical case see for example \cite{MR2002b:47144}.
For the slice hyperholomorphic case see
\cite{AACKS, AGALCOSA,abcs,ALLJECK4,ALLJECK3, ACKS1,acls_milan, acs1,MR3127378,acs3,ALLJECK1,ALLJECK2} and also the forthcoming monograph \cite{ACSBOOK}.
\\
\\
Since the literature in the field is so vast, it became quite natural to ask whether the deep theory of holomorphic entire functions can be reconstructed for slice regular functions. As this book will demonstrate, the answer is positive, and our contribution here consists in showing the way to a complete theory of entire slice regular functions. It is probably safe to assert that this monograph is only the first step, and in fact we only have chosen a fairly limited subset of the general theory of homomorphic entire functions, to demonstrate the feasibility of our project. We expect to return to other important topics in a subsequent volume.
\\

This monographs contains four chapters, beside this introduction.
 In Chapter 2 and 3 we introduce and discuss the algebra and the analysis of slice regular functions. While most of the results in those chapters are well known, and can be found in the literature, see e.g. \cite{ACSBOOK,CSS,GSS}, we repeated them to make the monograph self-contained. However  there are a few new observations (e.g. in section 1.4 where we tackle the composition of slice regular functions and also the Riemann mapping theorem) that do not appear in the aforementioned monographs. There are also a few new results (for example we complete the discussion on lower bounds for slice regular functions initiated with the Ehrenpreis-Malgrange in section 3.4, by adding a brand new Cartan-type theorem in section 3.5).

   Chapter 4 deals with infinite products of slice regular functions. The results in this chapter are known, but at least the Weierstrass theorem receives here a treatment that is different from the one originally given in \cite{MR2836832}, see also \cite{GSS}. This treatment leads also to the definition of genus of a a canonical product. The core of the work, however, is Chapter 5, where we study the growth of entire slice regular functions, and we show how such growth is related to the coefficients of the power series expansions that these functions have.  This chapter contains new results, the only exception is  section 5.5 which is taken from \cite{GalSa1}. It should be noted that the proofs we offer are not simple reconstructions (or translations) of the holomorphic case. Indeed, the non-commutative setting creates a series of non-trivial problems, that force us to define composition and multiplication in ways that are not conducive to a simple repetition of the complex case.
Also the counting of the zeros is not trivial because of the presence of spherical zeros which  have infinite cardinality.

We believe that much work still needs to be done in this direction, and we hope that our monograph will inspire others to turn their attention to this nascent, and already so rich, new field of noncommutative analysis.

\vskip 1truecm
{\bf Acknowledgments}. The first two authors are grateful to Chapman University for the kind hospitality and the support during the periods in which this work has been written.  The second author thanks the project FIRB 2012 "Geometria Differenziale e Teoria Geometrica delle
Funzioni" for partially supporting this research.

%
%
%


\chapter{Slice regular functions: algebra}
\section{Definition and main results}

In this chapter we will present some basic material on slice regular functions, a generalization of holomorphic functions to the quaternions.\\
The skew field of quaternions $\mathbb H$ \index{$\mathbb H$} is defined as
$$\mathbb{H}=\{q=x_0  + x_{1} i +x_{2} j + x_{3} k ; \ \ \ x_0,\ldots, x_3 \in \mathbb{R} \},$$
where the imaginary units $i, j, k$ satisfy
$$i^2=j^2=k^2=-1, \ ij=-ji=k, \ jk=-kj=i, \ ki=-ik=j.$$ It is a noncommutative field and since  $\mathbb{C}$ can be identified (in a non unique way) with a subfield of $\mathbb{H}$, it extends the class of complex numbers. On $\mathbb{H}$ we define the Euclidean norm
$|q|=\sqrt{x_0^2 +x_1^2+x_2^3+x_3^2}$.

The symbol $\mathbb{S}$ denotes the unit sphere of purely imaginary quaternion, i.e.
$$\mathbb{S}=\{q = i x_{1} + j x_{2} + k x_{3}, \mbox{ such that } x_{1}^{2}+x_{2}^{2}+x_{3}^{3}=1\}.$$
Note that if $I\in \mathbb{S}$, then $I^{2}=-1$. For this reason the elements of $\mathbb{S}$ are also
called imaginary units. For any fixed $I\in\mathbb{S}$ we define $\mathbb{C}_I:=\{x+Iy; \ |\ x,y\in\mathbb{R}\}$. It is easy to verify that
$\mathbb{C}_I$ can be identified with a complex plane, moreover $\mathbb{H}=\bigcup_{I\in\mathbb{S}} \mathbb{C}_I$.
The real axis belongs to $\mathbb{C}_I$ for every $I\in\mathbb{S}$ and thus a real quaternion can be associated with any imaginary unit $I$.
Any non real quaternion $q$ is uniquely associated to the element $I_q\in\mathbb{S}$
defined by $I_q:=( i x_{1} + j x_{2} + k x_{3})/|  i x_{1} + j x_{2} + k x_{3}|$. It is obvious that $q$ belongs to the complex plane $\mathbb{C}_{I_q}$.
\\
\begin{definition}
 Let $U$ be an open set in $\mathbb{H}$ and $f:\, U\to \mathbb{H}$ be real differentiable. The function $f$ is said to be (left) slice regular \index{function!slice regular} \index{slice regular!function} or (left) slice hyperholomorphic
if for every $I\in \mathbb{S}$, its restriction $f_{I}$ to the complex plane ${\mathbb{C}}_{I}=\mathbb{R}+ I \mathbb{R}$ passing through origin
and containing $I$ and $1$ satisfies
$$\overline{\partial}_{I}f(x+I y):=\frac{1}{2}\left (\frac{\partial}{\partial x}+I \frac{\partial}{\partial y}\right )f_{I}(x+I y)=0,$$
on $U\cap \mathbb{C}_{I}$.
The class of (left) slice regular functions on $U$ will be denoted by $\mathcal{R}(U)$. \index{$\mathcal{R}(U)$}
\\
Analogously, a function is said to be right slice regular in $U$ if
$$(f_{I}{\overline{\partial}}_{I})(x+I y):=\frac{1}{2}\left (\frac{\partial}{\partial x}f_{I}(x +I y)+\frac{\partial}{\partial y}f_{I}(x+I y) I\right )=0,$$
on $U\cap \mathbb{C}_{I}$.
\end{definition}
It is immediate to verify that:
\begin{proposition} Let $U$ be an open set in $\mathbb H$. Then $\mathcal R(U)$ is a right linear space on $\mathbb H$.
\end{proposition}
Let $f\in\mathcal{R}(U)$. The so called left (slice) $I$-derivative of $f$ at a point $q=x+I y$ is given by
$$\partial_{I}f_{I}(x+ I y):=\frac{1}{2}\left (\frac{\partial}{\partial x}f_{I}(x+ I y) - I\frac{\partial}{\partial y}f_{I}(x+ I y)\right ).$$
In this case, the right $I$-derivative of $f$ at $q=x+I y$ is given by
$$\partial_{I}f_{I}(x+ I y):=\frac{1}{2}\left (\frac{\partial}{\partial x}f_{I}(x+ I y) - \frac{\partial}{\partial y}f_{I}(x+ I y) I\right ).$$
Let us now introduce a suitable notion of  derivative:
\begin{definition}
Let $U$ be an open set in $\mathbb{H}$, and let $f:U \to \mathbb{H}$
be a slice regular function. The slice derivative $\partial_s f$ \index{slice derivative} of $f$,
 is defined by:
\begin{displaymath}
\partial_s(f)(q) = \left\{ \begin{array}{ll}
\partial_I(f)(q) & \textrm{ if $q=x+Iy$, \ $y\neq 0$},\\ \\
\displaystyle\frac{\partial f}{\partial x} (x) & \textrm{ if\  $q=x\in\mathbb{R}$.}
\end{array} \right.
\end{displaymath}
\end{definition}
The definition of slice derivative is well posed because it is applied
only to slice regular functions, thus
$$
\frac{\partial}{\partial x}f(x+Iy)=
-I\frac{\partial}{\partial y}f(x+Iy)\qquad \forall I\in\mathbb{S}.
$$
Similarly to what happens in the complex case, we have
$$ \partial_s(f)(x+Iy) =
\partial_I(f)(x+Iy)=\partial_x(f)(x+Iy).
$$
We will often  write $f'(q)$ instead of $\partial_s f(q)$.\\
It is important to note that if $f(q)$ is a slice regular function then also $f'(q)$ is a slice regular function.\\
Let
 $I,J\in\mathbb{S}$ be such that $I$ and $J$ are orthogonal, so that $I,J,IJ$ is an orthogonal
basis of $\mathbb H$ and write the restriction  $f_I(x+Iy)=f(x+Iy)$ of $f$ to the complex plane $\mathbb C_I$  as
$f=f_0+If_1+Jf_2+Kf_3.$ It can also be written as $f=F+GJ$ where $f_0+If_1=F$, and $f_2+If_3=G$.
This observation immediately gives the following result:
\begin{lemma}[Splitting Lemma]\label{lemma spezzamento} \index{Splitting Lemma}
If $f$ is a slice regular function
on $U$, then for\index{Splitting Lemma}
every $I \in \mathbb{S}$, and every $J\in\mathbb{S}$,
perpendicular to $I$, there are two holomorphic functions
$F,G:U\cap \mathbb{C}_I \to \mathbb{C}_I$ such that for any $z=x+Iy$, it is
$$f_I(z)=F(z)+G(z)J.$$
\end{lemma}
\begin{proof}
 Since $f$ is slice regular, we know that
$$
(\frac{\partial}{\partial x} +I\frac{\partial}{\partial
y})f_I(x+Iy)=0.
$$
Therefore by decomposing the values of $f_I$ into complex components
$$
f_I(x+Iy)=F(x+Iy)+G(x+Iy)J,
$$
the statement immediately follows.
\end{proof}
In this work we will be interested in the case in which slice regular functions are considered on open balls $B(0;r)$ centered at the origin with radius $r>0$.
We have:
\begin{theorem}
 A function $f:\ {B}(0;{r}) \to \mathbb{H}$ slice regular on ${B}(0;{r})$  has a series representation of the form
\begin{equation}\label{powerseries}
f(q)=\sum_{n=0}^{\infty}q^{n}\frac{1}{n !}\cdot \frac{\partial^{n} f}{\partial x^{n}}(0)=\sum_{n=0}^{\infty}q^{n} a_n,
\end{equation}
uniformly convergent on ${B}(0;{r})$. Moreover, $f\in\mathcal C^\infty (B(0;R))$.
\end{theorem}
\begin{proof}
Let us use the Splitting Lemma to write $f_I(z)=F(z)+G(z)J$, where $z=x+Iy$. Then
 $$
 f_I^{(n)}(z)=\partial_I^{(n)} f(z)=\frac{\partial^n}{\partial z^n} F(z)+ \frac{\partial^n}{\partial z^n} G(z)J.
 $$
 Now we use the fact that $F$, and $G$ admits converging power series expansions which converge uniformly and absolutely on any compact set in $B(0;r)\cap\mathbb C_I$:
\[
\begin{split}
f_I(z)&=\sum_{n\geq 0} z^n \frac{1}{n!}
\frac{\partial^nF}{\partial z^n}(0) + \sum_{n\geq 0} z^n
\frac{1}{n!} \frac{\partial^nG}{\partial z^n}(0)J
\\
&= \sum_{n\geq 0}
z^n \frac{1}{n!} \left( \frac{\partial^n(F+GJ)}{\partial
z^n}(0)\right)= \sum_{n\geq 0} z^n \frac{1}{n!} \left(
\frac{\partial^n f}{\partial z^n}(0)\right)\\
&=\sum_{n\geq 0} z^n
\frac{1}{n!} \left(\frac{1}{2} \left( \frac{\partial}{\partial x} -I
\frac{\partial}{\partial y}\right)\right)^n f(0)=\sum_{n\geq 0} z^n
\frac{1}{n!} f^{(n)} (0).
\end{split}
\]
The function $f$ is infinitely differentiable in $B(0;r)$ by the uniform convergence on the compact subsets.
\end{proof}
The proof of the following result is the same as the proof in the complex case.
\begin{theorem}
Let $\{a_n\}$, $n\in\mathbb N$ be a sequence of quaternions and let
$$
r = \frac{1}{\overline{\lim} |a_n|^{1/n} }.
$$
If $r > 0$ then the power series $\sum_{N=0}^\infty q^n a_n$
converges absolutely and uniformly on compact subsets of $B(0;r)$. Its sum defines a slice regular function on $B(0;r)$.
\end{theorem}
\begin{definition}
A function slice regular on $\mathbb H$ will be called entire slice regular or entire slice hyperholomorphic.\index{function!entire slice regular}
\end{definition}
Every entire regular function admits power series expansion of the form \eqref{powerseries} which converges everywhere in $\mathbb H$ and uniformly on the compact subsets of $\mathbb H$.
\\
A simple computation shows that since the radius of convergence is infinite we have
$$
\lim_{n\to\infty} |a_n|^{\frac 1n}=0,
$$
or, equivalently:
$$
\lim_{n\to\infty} \frac{\log |a_n|}{n}=-\infty
$$

\section{The Representation Formula}
Slice regular functions possess good properties on specific open sets that we will call axially symmetric slice domains. On these domains, slice regular functions satisfy the so-called Representation Formula which allows to reconstruct the values of the function once that we know its values on a complex plane $\mathbb C_I$. As we shall see, this will allow also to define a suitable notion of multiplication between two slice regular functions.
\begin{definition}
Let $U \subseteq \mathbb H$. We say that $U$ is
\textnormal{axially symmetric} \index{axially symmetric} if, for every $x+Iy \in U$,  all the elements $x+\mathbb{S}y=\{x+Jy\ | \ J\in\mathbb{S}\}$ are contained in $U$.
We say that $U$ is a {\em slice domain}  \index{slice domain} (or s-domain \index{s-domain} for short) if it is a connected set whose intersection with every complex plane $\mathbb C_I$ is connected.
\end{definition}
\begin{definition} Given a quaternion $q=x+Iy$, the set of all the elements of the form $x+Jy$ where $J$ varies in the sphere $\mathbb{S}$ is a $2$-dimensional sphere denoted by $[x+Iy]$.
\end{definition}
The Splitting Lemma allows to prove:

\begin{theorem}[Identity Principle] \label{identity principle} \index{Identity Principle}  Let $f:U\to\mathbb{H}$ be a slice regular function
on an s-domain $U$. Denote by $Z_f=\{q\in U : f(q)=0\}$
the zero set of $f$. If there exists $I \in \mathbb{S}$ such that
$\mathbb{C}_I \cap Z_f$ has an accumulation point, then $f\equiv 0$ on
$U$.
\end{theorem}
\begin{proof}
The restriction $f_I=F+GJ$ of $f$ to $U\cap \mathbb C_I$ is such that $F,G:\, U\cap \mathbb C_I \to\mathbb C_I$ are holomorphic functions. Under the hypotheses,
$F$ and $G$  vanish on a set $\mathbb{C}_I \cap Z_f$ which has an accumulation point so $F$ and $G$ are both identically zero. So $f_I$ vanishes on $U\cap \mathbb R$ and, in particular, $f$ vanishes on $U\cap \mathbb R$.
Thus the restriction of $f$ to any other complex plane $\mathbb C_L$ vanishes on a set with an accumulation point and so $f_L\equiv 0$. Since
$$
U=\bigcup_{I\in\mathbb C_I} U\cap\mathbb C_I
$$
we have that $f$ vanishes on $U$.
\end{proof}

The following result shows that the values of a slice regular function defined on an axially symmetric s-domain can be computed if the values of the restriction to a complex plane are known:
\begin{theorem}[Representation Formula]\label{Repr_formula} \index{Representation Formula} Let
$f$ be a slice regular function defined an axially symmetric s-domain $U\subseteq  \mathbb{H}$. Let
$J\in \mathbb{S}$ and let $x\pm Jy\in U\cap\mathbb C_J$.  Then the following equality holds for all $q=x+Iy \in U$:
\begin{equation}\label{distribution_mon}
\begin{split}
f(x+Iy) &=\frac{1}{2}\Big[   f(x+Jy)+f(x-Jy)\Big]
+I\frac{1}{2}\Big[ J[f(x-Jy)-f(x+Jy)]\Big]\\
&= \frac{1}{2}(1-IJ) f(x+Jy)+\frac{1}{2}(1+IJ) f(x-Jy).
\end{split}
\end{equation}
Moreover, for all  $x+Ky \subseteq U$, $K\in\mathbb S$, there exist two functions $\alpha$, $\beta$, independent of $I$, such for any $K \in \mathbb{S}$ we have
\begin{equation}\label{cappa}
\frac{1}{2}\Big[   f(x+yK)+f(x-yK)\Big]=\alpha(x,y),
\quad \frac{1}{2}\Big[ K[f(x-yK)-f(x+yK)]\Big]=\beta(x,y).
\end{equation}

\end{theorem}
\begin{proof} If ${\rm Im}(q)=0$, then $q$ is real, the proof is immediate. Otherwise
let us define the function $\psi:U \to \mathbb{H}$ as follows
$$
\psi(q)=\frac{1}{2}\Big[   f({\rm Re}(q) +|{\rm Im}(q)|J)+f({\rm Re}(q) - |{\rm
Im}(q)|J)
$$
$$
+\frac{{\rm Im}(q)}{|{\rm Im}(q)|}J[f({\rm Re}(q) - |{\rm Im}(q)|J)-f({\rm Re}(q) + |{\rm
Im}(q)|J)]\Big].
$$
Using the fact that $q=x+yI$, $x,y\in\mathbb{R}$, $y\geq 0$ and
$I=\displaystyle\frac{{\rm Im}(q)}{|{\rm Im}(q)|}$ we obtain
$$
\psi(x+yI)=\frac{1}{2}\Big[   f(x+yJ)+f(x-yJ) +I
J[f(x-yJ)-f(x+yJ)]\Big].
$$
Observe that for $I=J$ we have
$$
\psi_J(q)=\psi(x+yJ)=f(x+yJ)=f_J(q).
$$
Thus if we prove that $\psi$ is regular on $U$,
the first part of the assertion follows from the Identity Principle.
Since $f$ is regular on $U$, for any $I\in \mathbb{S}$ we have,
on $U\cap \mathbb C_{I}$
$$
\frac{\partial}{\partial x}2\psi(x+yI)
=\frac{\partial}{\partial x} \Big[f(x+yJ)+f(x-yJ) +I J[f(x-yJ)-f(x+yJ)]\Big]
$$
$$
=\frac{\partial}{\partial x} f(x+yJ)+\frac{\partial}{\partial x} f(x-yJ) +IJ[\frac{\partial}{\partial x} f(x-yJ)-\frac{\partial}{\partial x} f(x+yJ)]
$$
$$
=-J\frac{\partial}{\partial y} f(x+yJ)+J\frac{\partial}{\partial y} f(x-yJ) +I J[J\frac{\partial}{\partial y} f(x-yJ)+J\frac{\partial}{\partial y} f(x+yJ)]
$$
$$
=-J\frac{\partial}{\partial y} f(x+yJ)+J\frac{\partial}{\partial y} f(x-yJ) -I [\frac{\partial}{\partial y} f(x-yJ)+\frac{\partial}{\partial y} f(x+yJ)]
$$
$$
=-I\frac{\partial}{\partial y} \Big[f(x+yJ)+f(x-yJ) +I J[f(x-yJ)-f(x+yJ)]\Big]=-I\frac{\partial}{\partial y}2\psi(x+yI)
$$
i.e.
\begin{equation}\label{reg}
\frac{1}{2}(\frac{\partial}{\partial x}+I\frac{\partial}{\partial
y})\psi(x+yI)=0.
\end{equation}
To prove (\ref{cappa}) we take any $K\in \mathbb{S}$ and use
equation (\ref{distribution_mon}) to show that
\[
\begin{split}
&\frac{1}{2}\Big[   f(x+yK)+f(x-yK)\Big]\\
&=\frac{1}{2}\Big\{\frac{1}{2}\Big[   f(x+yJ)+f(x-yJ)\Big] +K\frac{1}{2}\Big[ J[f(x-yJ)-f(x+yJ)]\Big]\\
&+\frac{1}{2}\Big[   f(x+yJ)+f(x-yJ)\Big] -K\frac{1}{2}\Big[ J[f(x-yJ)-f(x+yJ)]\Big]\Big\} \\
&=\frac{1}{2}\Big[   f(x+yJ)+f(x-yJ)\Big]=\alpha(x,y)
\end{split}
\]
and that
\[
\begin{split}
&\frac{1}{2}\Big[ K[f(x-yK)-f(x+yK)]\Big]\\
&=\frac{1}{2}K\Big\{\frac{1}{2}\Big[   f(x+yJ)+f(x-yJ)\Big] -K\frac{1}{2}\Big[ J[f(x-yJ)-f(x+yJ)]\Big]\\
&-\frac{1}{2}\Big[   f(x+yJ)+f(x-yJ)\Big] -K\frac{1}{2}\Big[ J[f(x-yJ)-f(x+yJ)]\Big]\Big\}\\
&=\frac{1}{2}K\Big[-K\Big[ J[f(x-yJ)-f(x+yJ)]\Big]\\
&=\frac{1}{2}\Big[ J[f(x-yJ)-f(x+yJ)]\Big]=\beta(x,y),
\end{split}
\]
and the proof is complete.
\end{proof}

Some immediate consequences are the following corollaries:
\begin{corollary}\label{alfabeta_qua}
Let $U\subseteq \mathbb{H}$  be an axially symmetric s-domain, let $D\subseteq\rr^2$ be such that $x+Iy\in U$ whenever $(x,y)\in D$ and let $f : U \to \mathbb{H}$. The function $f$ is slice regular if and only if
there exist two differentiable functions $\alpha, \beta: D\subseteq\rr^2 \to \mathbb{H}$ satisfying $\alpha(x,y)=\alpha(x,-y)$, $\beta(x,y)=-\beta(x,-y)$and the Cauchy-Riemann system
 \begin{equation}\label{CRSIST_qua}
\left\{
\begin{array}{c}
\pp_x \alpha-\pp_y\beta=0\\
\pp_x \beta+\pp_y\alpha=0\\
\end{array}
\right.
\end{equation}
such that
\begin{equation}\label{alphabeta}
f(x+Iy)=\alpha(x,y)+I\beta(x,y).
\end{equation}
\end{corollary}
\begin{proof}
A function of the form \eqref{alphabeta} where $\alpha$ and $\beta$ satisfy the hypothesis in the statement is clearly slice regular. Conversely, a slice regular function on an axially symmetric s-domain satisfies the Representation formula and thus it is of the form \eqref{alphabeta}, where $\alpha$ and $\beta$ satisfy the Cauchy-Riemann system. The conditions $\alpha(x,y)=\alpha(x,-y)$, $\beta(x,y)=-\beta(x,-y)$ can be easily verified from the definition of $\alpha$ and $\beta$ given in the Representation Formula.
\end{proof}

\begin{remark}{\rm
Since $\alpha$, $\beta$ are quaternion-valued functions for any given $I\in\mathbb S$ we select $J\in\mathbb S$ orthogonal to $I$ and we can write (omitting, for simplicity, the arguments of the functions):
$$
f_I=\alpha +I\beta = (a_0+Ia_1+Ja_2+IJa_3)+I(b_0+Ib_1+Jb_2+IJb_3)
$$
$$
=(a_0-b_1)+I(a_1+b_0)+J(a_2-b_3)+IJ(a_3+b_2)
$$
$$
=c_0+Ic_1+Jc_2+IJc_3=(c_0+Ic_1)+(c_2+Ic_3)J,
$$
where the functions $a_\ell$, $b_\ell$, $c_\ell$, $\ell=0,\ldots, 3$ are real-valued.\\
Imposing $(\pp_x+I\pp_y)f_I=0$ we obtain the equations
\begin{equation}\label{sys1}
 \left\{
\begin{array}{lc}
&\pp_xc_0-\pp_yc_1=0
\\
&\pp_yc_0+\pp_xc_1=0
\\
&\pp_xc_2-\pp_yc_3=0
\\
&\pp_yc_2+\pp_xc_3=0.
\end{array}
\right.
\end{equation}
Requiring that the pair $\alpha, \beta$ satisfies the Cauchy-Riemann system, we obtain eight real equations:
\begin{equation}\label{sys2}
\pp_xa_i -\pp_y b_i=0\qquad \pp_x b_i+\pp_y a_i=0\qquad i=0,\ldots,3.
\end{equation}
System \eqref{sys1}
 is equivalent to the request that the functions $F=c_0+Ic_1$ and $G=c_2+Ic_3$  prescribed by the Splitting Lemma are holomorphic.
As it can be easily verified, by setting $c_0=a_0-b_1$, $c_1=a_1+b_0$, $c_2=a_2-b_3$, $c_3=a_3+b_2$ the solutions to system (\ref{sys2}) give solutions to system (\ref{sys1}).
}
\end{remark}
The previous remark implies a stronger version of the Splitting lemma  for slice regular functions.
To prove the result we need to recall that complex functions defined on open sets $G\subset \mathbb{C}$ symmetric with respect to the real axis and such that $\overline{f(\bar z)}=f(z)$ are called in the literature intrinsic, see \cite{rinehart}.\index{function!intrinsic}
\begin{proposition} [Refined Splitting Lemma] \label{SL2} \index{Refined Splitting Lemma} \index{Splitting Lemma!Refined}
Let $U$ be an open set in $\hh$ and let $f\in\mathcal{R}(U)$. For any $I\in\mathbb{S}$ there exist four holomorphic intrinsic functions $h_\ell:\ U\cap\cc_I\to\cc_I$, $\ell=0,\ldots, 3$ such that
$$
f_I(x+Iy)=h_0(x+Iy)+ h_1(x+Iy) I+ h_2(x+Iy)J + h_3(x+Iy) K.
$$
\end{proposition}
\begin{proof}
Let us write (omitting the argument $x+Iy$ of the various functions):
$$
f_I=a_0+Ia_1+Ja_2+Ka_3+I(b_0+Ib_1+Jb_2+Kb_3)
$$
$$
=(a_0+Ib_0)+(a_1+Ib_1)I+(a_2+Ib_2)J+(a_3+Ib_3)K$$
$$
=h_0+h_1 I+h_2J+h_3K.$$

It follows from  system (\ref{sys2}) that $h_\ell$, $\ell=0,\ldots, 3$ satisfy the Cauchy-Riemann system.
 The fact that the functions $\alpha(x,y)$ and $\beta(x,y)$ defined in the Splitting Lemma are even and odd, respectively, in the variable $y$ implies that $a_\ell(x,y)$ and $b_\ell (x,y)$ are even and odd, respectively, in the variable $y$.
Thus
$$
h_\ell(x-Iy)=a_\ell(x,-y)+Ib_\ell(x,-y)=a_\ell(x,y)-Ib_\ell(x,y)=\overline{h_\ell(x+Iy)}
$$
and so the functions $h_\ell$, $\ell=0,1,2,3$ are complex intrinsic.
\end{proof}
\begin{corollary}\label{hartogs}
A slice regular function $f : U\to\mathbb{H}$ on an axially symmetric
s-domain is infinitely differentiable on $U$. It is also real analytic on $U$.
\end{corollary}
\begin{corollary}
\label{values general}
Let $U\subseteq \mathbb{H}$  be an axially symmetric s-domain and let $f : U \to \mathbb{H}$
be a slice regular function.
For all $x_0,y_0 \in
\mathbb{R}$ such that $x_0+Iy_0\in U$ there exist $a, b \in \mathbb{H}$ such that
\begin{equation}\label{b e c}
f(x_0+Iy_0)=a+Ib
\end{equation}
for all $I\in \mathbb{S}$. In particular, $f$ is
affine in $I\in\mathbb{S}$ on each 2-sphere $[x_0+Iy_0]$ and
the image of the $2-$sphere
$[x_0+Iy_0]$ is the set $[a+Ib]$.
\end{corollary}

\begin{corollary}\label{corollary sphere_qua} Let $U\subseteq \mathbb{H}$  be an axially symmetric s-domain and let $f : U \to \mathbb{H}$
be a slice regular function.
If $f(x+Jy)=f(x+Ky)$ for $I\neq K$ in $\mathbb{S}$, then $f$ is constant on
$[x+Iy]$.
In particular, if $f(x+Jy)=f(x+Ky)=0$ for $I\neq K$ in $\mathbb{S}$, then $f$
vanishes on the whole $2-$sphere $[x+Iy]$.
\end{corollary}

\begin{corollary}
Let $U_J$ be a domain in $\mathbb C_J$ symmetric with respect to the real axis and such that $U_J\cap\mathbb R\not=\emptyset$. Let $U$ be the axially symmetric s-domain defined by
$$
U=\bigcup_{x+Jy\in U_J,\ I\in\mathbb S} \{x+Iy\}.
$$
If $f:U_J\to \mathbb H$ satisfies $\overline{\partial}_J f=0$ then the function
\begin{equation}\label{ext}
{\rm ext}(f)(x+Iy)=\frac{1}{2}\Big[   f(x+Jy)+f(x-Jy)\Big]
+I\frac{1}{2}\Big[ J[f(x-Jy)-f(x+Jy)]\Big]
\end{equation}
is the unique slice regular extension \index{slice regular!extension} of $f$ to $U$.
\end{corollary}

\begin{definition}\label{axial completion}
Let $U_J$ be any open set in $\mathbb C_J$ and let
\begin{equation}
U=\bigcup_{x+Jy\in U_J,\ I\in\mathbb S} \{x+Iy\}.
\end{equation}
We say that $U$ is the axially symmetric completion \index{axially symmetric!completion} of $U_J$ in $\mathbb H$.
\end{definition}
  Corollary \ref{alfabeta_qua} implies that  slice regular functions are a subclass of the following set of functions:
  \begin{definition}\label{slice function}
  Let  $U\subseteq\mathbb H$ be an axially symmetric open set.
  Functions of the form $f(q)=f(x+Iy)=\alpha(x,y)+I\beta(x,y)$, where $\alpha$ $\beta$ are continuous  $\mathbb H$-valued functions such that $\alpha(x,y)=\alpha(x,-y)$, $\beta(x,y)=-\beta(x,-y)$ for all $x+Iy\in U$ are called {\em continuous slice functions}.\index{slice functions}
 \end{definition}
 Let  $U\subseteq\mathbb H$ be an axially symmetric open set.

 \begin{theorem}[General Representation Formula] \label{1.3} \index{Representation Formula!General}
Let $U \subseteq\mathbb{H}$ be an axially symmetric s-domain and $f:U\to \mathbb{H}$ be a left slice regular function.
The following equality holds for all $q=x+I y \in U$, $J,K\in\mathbb S$:
\begin{eqnarray*}\label{strutturaquat2}
f(x+I y) =(J-K)^{-1}[ J  f(x+Jy)- Kf(x+Ky)]\\
+I(J-K)^{-1}[ f(x+Jy)-f(x-Ky)].
\end{eqnarray*}
\end{theorem}
\begin{proof} If $q$ is real the proof is immediate. Otherwise, for all $q=x+yI$, we define the function
\begin{eqnarray*}\label{cappaJK1}
\phi(x+yI) =(J-K)^{-1}\Big[  J f(x+yJ)-Kf(x+yK)\Big]\\
+I(J-K)^{-1}\Big[ f(x+yJ)-f(x+yK)\Big]
\\
= [(J-K)^{-1}J+I(J-K)^{-1}] f(x+yJ) \\
- [(J-K)^{-1}K + I(J-K)^{-1}]f(x+yK).
\end{eqnarray*}
As we said, for all  $q=x \in U\cap \mathbb{R}$  we have
$$
\phi(x)=f(x).
$$
Therefore if we prove that $\phi$ is slice regular on $U$,
the first part of the assertion will follow from the identity principle
for slice regular functions.
Indeed, since $f$ is slice regular on $U$, for any $L\in \mathbb{S}$ we have $\frac{\partial}{\partial x}f(x+Ly)= -L\frac{\partial}{\partial y}f(x+Ly)$
on $U\cap \mathbb{C}_{L}$; hence
\[
\begin{split}
\frac{\partial \phi}{\partial x}(x+yI)&=-[(J-K)^{-1}J+I(J-K)^{-1}]J\frac{\partial f}{\partial y} (x+yJ)\\
 &+ [(J-K)^{-1}K + I(J-K)^{-1}]K\frac{\partial f}{\partial y}(x+yK)
 \end{split}
\]
and also
\[
\begin{split}
I\frac{\partial \phi}{\partial y}(x+yI)& = I [(J-K)^{-1}J+I(J-K)^{-1}] \frac{\partial f}{\partial y}(x+yJ)\\
 &- I[(J-K)^{-1}K + I(J-K)^{-1}]\frac{\partial f}{\partial y}(x+yK).
\end{split}
\]
It is at this point immediate that
$$
\Big(\frac{\partial}{\partial x}+I\frac{\partial}{\partial
y}\Big)\phi(x+yI)=0
$$
and equality (\ref{strutturaquat2}) is proved.
\end{proof}

\section{Multiplication of slice regular functions}
In the case of slice regular functions defined on axially symmetric s-domains we can define a suitable product, called $\star$-product, which preserves slice regularity. This product extends the very well know product for polynomials and series with coefficients in a ring, see e.g. \cite{fliess} and \cite{lam}.
The inverse of a function with respect to the $\star$-product requires to introduce the so-called conjugate and symmetrization of a slice regular function. These two notions, as we shall see, will be important also for other purposes.
\\
Let $U\subseteq\mathbb{H}$ be an axially symmetric s-domain and let
 $f,g:\ U\to\mathbb{H}$ be slice regular functions. For any $I,J\in\mathbb{S}$, with
 $I\perp J$, the Splitting Lemma guarantees the existence of four holomorphic functions
 $F,G,H,K: \ U\cap\mathbb{C}_I\to \mathbb{C}_I$ such
 that for all $z=x+Iy\in   U\cap\mathbb{C}_I$
 \begin{equation}\label{SLfg}
 f_I(z)=F(z)+G(z)J, \qquad g_I(z)=H(z)+K(z)J.
 \end{equation}
We define the function $f_I\star g_I:\ U\cap\mathbb{C}_I\to \mathbb{H}$ as
\begin{equation}\label{f*g}
f_I\star g_I(z)=[F(z)H(z)-G(z)\overline{K(\bar z)}]+[F(z)K(z)+G(z)\overline{H(\bar z)}]J.
\end{equation}
Then  $f_I\star g_I(z)$ is obviously a holomorphic map and hence
its unique slice regular extension to $U$ defined, according to the extension
formula \eqref{ext} by
$$
{\rm ext}(f_I\star g_I)(q),
$$
 is slice regular on $U$.
 \begin{definition}
Let $U\subseteq\mathbb{H}$ be an axially symmetric s-domain and let
 $f,g:\ U\to\mathbb{H}$  be slice regular. The function \index{slice regular!product}
$$(f\star g)(q)={\rm ext}(f_I\star g_I)(q)$$ defined as the extension of (\ref{f*g})
is called the slice regular product of $f$
and $g$. This product is called $\star $-product or slice regular product.
 \end{definition}
\begin{remark}{\rm
It is immediate to verify that the $\star $-product is associative, distributive
but, in general, not commutative.
}
\end{remark}
\begin{remark}\label{J*H}
{\rm Let $H(z)$ be a holomorphic function in the variable $z\in \mathbb{C}_I$ and let $J\in\mathbb{S}$
be orthogonal to $I$.
Then by the definition of $\star $-product we obtain $J\star H(z)=\overline{H(\bar z)}J$.
}
\end{remark}
Using the notations above we have:
\begin{definition}
Let $f\in\mathcal{R}(U)$  and let $f_I(z)={F(z)}G(z)J$. We define the function $f_I^c:\  U\cap\mathbb{C}_I \to \mathbb{H}$ as
\begin{equation}\label{f^c}
f_I^c(z)=\overline{F(\bar z)}-G(z)J.
\end{equation}
Then  $f_I^c(z)$ is a holomorphic map and we define the conjugate function \index{conjugate function} $f^c$ of $f$ as \index{$f^c$}
$$
f^c(q)={\rm ext}(f_I^c)(q).
$$
\end{definition}
Note that, by construction, $f^c\in\mathcal{R}(U)$. Note that if we write the function $f$ in the form
$f(x+Iy)=\alpha(x,y)+I\beta(x,y)$, then it is possible to show that
$$
f^c(x+Iy)=\overline{\alpha(x,y)}+I\overline{\beta(x,y)}.
$$
\begin{remark} {\rm Some lengthy but easy computations show that if $c$ is a fixed quaternion and $I\in\mathbb S$ there exists $J\in\mathbb S$ such that $Ic=cJ$. Thus we have
\begin{equation}\label{fcab}
\begin{split}
|f^c(x+Iy)|&=|\overline{\alpha(x,y)}+I\overline{\beta(x,y)}|=|\overline{\overline{\alpha(x,y)}+I\overline{\beta(x,y)}}|
\\
&=|{\alpha(x,y)}-{\beta(x,y)}I|= |{\alpha(x,y)}+J{\beta(x,y)}|\\
&=|f(x+Jy)|
\end{split}
\end{equation}
for a suitable $J\in\mathbb S$.}
\end{remark}
Using the notion of $\star $-multiplication of slice regular functions,
it is possible to associate to any slice regular function $f$ its ''symmetrization'' \index{symmetrization}
also called ''normal form'',\index{normal form}
denoted by $f^s$. We will show that all the zeros of $f^s$ are spheres of type $[x+Iy]$
(real points, in particular) and that,
if $x+Iy$ is a zero of $f$ (isolated or not) then $[x+Iy]$ is a zero of $f^s$.

Let $U\subseteq\mathbb{H}$ be an axially symmetric s-domain and let
 $f:\ U\to\mathbb{H}$ be a
  slice regular function. Using the notation in \eqref{fcab}, we consider the function $f^s_I:\ U\cap\mathbb{C}_I\to \mathbb{C}_I$ defined by
\begin{equation}\label{f^s_qua}
f^s_I=f_I\star f^c_I=(F(z)+G(z)J)\star (\overline{F(\bar z)}-G(z)J)
\end{equation}
$$
=[F(z)\overline{F(\bar z)}+G(z)\overline{G(\bar z)}]+[-F(z)G(z)+G(z)F(z)]J
$$
$$
=F(z)\overline{F(\bar z)}+G(z)\overline{G(\bar z)}=f^c_I\star f_I.
$$
Then  $f_I^s$ is holomorphic. We give the following definition:
\begin{definition}\label{fsymmetr}
Let $U\subseteq\mathbb{H}$ be an axially symmetric s-domain and let
 $f:\ U\to\mathbb{H}$  be slice regular. The function \index{$f^s$}
$$f^s(q)={\rm ext}(f_I^s)(q)$$ defined by the extension of (\ref{f^s_qua}) is called the symmetrization (or normal form) \index{symmetrization}\index{normal form} of $f$.
\end{definition}
\begin{remark}\label{symmLI_qua}
{\rm Notice that formula (\ref{f^s_qua}) yields that, for all $I\in\mathbb{S}$,
$f^s( U\cap\mathbb{C}_I)\subseteq \mathbb{C}_I$.}
\end{remark}
 We now show how the conjugation and the symmetrization of a slice regular function behave with respect to the $\star $-product:
\begin{proposition}\label{propmoltiplicative}
Let $U\subseteq\mathbb{H}$ be an axially symmetric s-domain and let
 $f,g:\ U\to\mathbb{H}$ be slice regular functions.
Then $$(f\star g)^c =
g^c\star f^c$$ and
\begin{equation}
(f\star g)^s = f^s g^s = g^s f^s.
\end{equation}
\end{proposition}
\begin{proof}
It is sufficient to show that $(f\star g)^c =
g^c\star f^c$.
As customary, we can use the Splitting Lemma to write on $ U\cap\mathbb{C}_I$
that $f_I(z)=F(z)+G(z)J$ and $g_I(z)=H(z)+K(z)J$.
We have
$$
f_I\star g_I(z)=[F(z)H(z)-G(z)\overline{K(\bar z)}]+[F(z)K(z)+G(z)\overline{H(\bar z)}]J
$$
and hence
$$
(f_I\star g_I)^c(z)=[\overline{F(\bar z)}\ \overline{H(\bar z)}-\overline{G(\bar z)}{K(z)}]
-[F(z)K(z)+G(z)\overline{H(\bar z)}]J.
$$
We now compute
$$
g_I^c(z)\star f_I^c(z)=(\overline{H(\bar z)}-K(z)J)\star (\overline{F(\bar z)}-G(z)J)
$$
$$
=\overline{H(\bar z)}\star \overline{F(\bar z)}-\overline{H(\bar z)}\star G(z)J
-K(z)J\star \overline{F(\bar z)}+K(z)J\star G(z)J
$$
and conclude by Remark \ref{J*H}.
 \end{proof}
\begin{proposition}\label{invefs}
Let $U\subseteq \mathbb{H}$  be an axially symmetric s-domain and let $f : U \to \mathbb{H}$
be a slice regular function.
The function
$(f^s(q))^{-1}$
is slice regular on $U\setminus \{ q\in \mathbb{H}\ | \ f^s(q)=0\}$.
\end{proposition}
\begin{proof}
The function $f^s$ is such that $f^s( U\cap\mathbb{C}_I)\subseteq \mathbb{C}_I$ for all $I\in\mathbb{S}$
by Remark \ref{symmLI_qua}. Thus, for any given $I\in\mathbb{S}$ the Splitting Lemma implies
the existence
of a holomorphic function $F:\  U\cap\mathbb{C}_I\to \mathbb{C}_I$ such that $f^s_I(z)=F(z)$ for all
$z\in  U\cap\mathbb{C}_I$. The inverse of the function $F$ is holomorphic on $ U\cap\mathbb{C}_I$
outside the zero set of $F$. The conclusion follows by the equality $(f_I^s)^{-1}=F^{-1}$.
 \end{proof}

The  $\star$-product can be related to the pointwise product as described in
the following result:
\begin{theorem}\label{pointwise}
Let $U \subseteq \mathbb{H}$ be an axially symmetric s-domain,
 $f, g : U \to \mathbb{H}$ be slice hyperholomorphic functions. Then
\begin{equation}
\label{productfg}
(f \star g)(p) = f(p) g(f(p)^{-1}pf(p)),
\end{equation}
for all $p\in U$, $f(p)\not=0$, while $(f \star g)(p) = 0$ when $p\in U$, $f(p)=0$.
\end{theorem}
An immediate consequence is the following:
\begin{corollary}\label{zeroes}
If $(f\star g)(p)=0$ then either $f(p)=0$ or $f(p)\not=0$ and $g(f(p)^{-1}pf(p))$ $=0$.
\end{corollary}

\section{Quaternionic intrinsic functions}
An important subclass of the class of slice regular functions on an open set $U$, denoted by $\mathcal{N}(U)$,  is defined as follows:
$$
 \mathcal{N}(U)=\{ f \ {\rm slice\ regular\ in}\ U\ :  \ f(U\cap \mathbb{C}_I)\subseteq  \mathbb{C}_I,\ \  \forall I\in \mathbb{S}\}.
$$
\begin{remark}{\rm
If $U$ is axially symmetric and if we denote by $\overline{q}$ the conjugate of a quaternion $q$, it can be shown that a function $f$ belongs to $\mathcal{N}(U)$
if and only if it satisfies $f(q)=\overline{f(\bar q)}$.
In analogy with the complex case, we say that the functions in  the class $\mathcal N$ are  {\em quaternionic intrinsic}.\index{function!quaternionic intrinsic}
}
\end{remark}

If one considers a ball $B(0,R)$ with center at the origin, it is immediate that a function slice regular on the ball belongs to $\mathcal{N}(B(0,R))$ if and only if its
power series expansion has real coefficients. Such functions are also said to be {\em real}. More in general, if $U$ is an axially symmetric s-domain, then  $f\in\mathcal{N}(U)$ if and only if $f(q)=f(x+Iy)=\alpha(x,y)+I\beta(x,y)$
with $\alpha$, $\beta$ real valued, in fact we have:
\begin{proposition}
Let $U\subseteq\mathbb H$ be an axially symmetric open set. Then
$f\in\mathcal{N}(U)$ if and only if
\begin{equation}\label{dec}
f(x)=\alpha(x,y)+I\beta(x,y)
\end{equation}
with $\alpha$, $\beta$ real valued, $\alpha(x,-y)=\alpha(x,y)$, $\beta(x,-y)=-\beta(x,y)$ and satisfy the Cauchy--Riemann system.
\end{proposition}
 \begin{proof}
 If $\alpha$, $\beta$ are as above, trivially $f(q)$ defined by (\ref{dec}) belongs to $\mathcal{N}(U)$. Conversely, assume that $f(q)\in\mathcal{N}(U)$. Let $I,J\in\mathbb{S}$, $I$ and $J$ orthogonal and let $f_I(z)=F(z)+G(z)J$ with $F, G:\ U\cap\mathbb{C}_I\to\mathbb{C}_I$ holomorphic. Since $f$ takes $\mathbb{C}_I$ to itself, then $G\equiv 0$. Then the function $f_I(z)=F(z)=\alpha(x,y)+I\beta(x,y)$ where $\alpha$, $\beta$ are real valued and satisfy the Cauchy-Riemann system. Thus, by the Identity Principle, the function $f(q)$ can be written as $f(q)=\alpha(x,y)+I_q\beta(x,y)$. The equalities $\alpha(x,-y)=\alpha(x,y)$, $\beta(x,-y)=-\beta(x,y)$ follow from the Representation Formula.
\end{proof}
\begin{remark}\label{transcendent}{\rm
The class  $\mathcal{N}(\mathbb H)$
 includes all elementary transcendental \index{function!transcendental} functions, in particular
 $$
\exp(q)=e^q=\sum^{\infty}_{n=0}\frac{q^n}{n!},
$$
$$
\sin(q)=\sum^{\infty}_{n=0}(-1)^{n}\frac{q^{2n+1}}{(2n+1)!},
$$
$$
\cos(q)=\sum^{\infty}_{n=0}(-1)^n\frac{q^{2n}}{(2n)!}.
$$
Note that these functions coincide with the analogous complex functions on any complex plane $\mathbb C_I$.
}
\end{remark}
\index{quaternionic!exponential}\index{quaternionic!sine}\index{quaternionic!cosine}
Also the quaternionic logarithm is a quaternionic intrinsic function. Let us recall the definition of quaternionic logarithm, see \cite{GSS}, \cite{ghs}. It is the function inverse of the exponential function $\exp(q)$ in $\mathbb{H}$.

\begin{definition}\label{deflogaritmo1}
Let $U \subseteq \mathbb{H}$ be a connected open  set. We define a {branch of the quaternionic logarithm} \index{quaternionic!logarithm}(or simply a \emph{logarithm}) on  $U $ a function $f:U \to \mathbb{H}$  such that for every $q\in U $  $$e^{f(q)}=q.$$
\end{definition}
Since $\exp(q)$ never vanishes, we suppose that $0\notin U $. Recalling that
$$  I_q= \left\{ \begin{array}{ll}
 {\rm Im}(q)/|{\rm Im}(q)| & \textrm{if $q\in \mathbb{H}\setminus\mathbb{R}$}\\
\textrm{any element of $\mathbb{S}$} & \textrm{otherwise}
  \end{array} \right. $$
 we have that for every $q \in \mathbb{H}\setminus \{0\}$ there exists a unique $\theta \in [0,\pi]$ such that $q=|q|e^{\theta I_q}$. Moreover we have  $\theta =\arccos({\rm Re}(q)/|q|)$.

 \begin{definition}
 The function $\arccos({\rm Re}(q)/|q|)$ will be called the {principal quaternionic argument} of $q$ and it will be denoted  by $\arg_{\mathbb{H}}(q)$ for every $q\in\mathbb{H}\setminus\{0\}$.
 \end{definition}

Below we define the principal quaternionic logarithm.

\begin{definition}\label{defilog} Let $\log (x)$ be the natural real logarithm of $x\in\mathbb R^+$. For every $q\in\mathbb{H}\setminus (-\infty,0],$ we define the {principal quaternionic logarithm} \index{quaternionic!principal logarithm} or, in short, {principal logarithm}) of $q$ as
 $${\rm Log}(q)=\ln|q|+\arccos\left(\frac{{\rm Re}(q)}{|q|}\right)I_q .$$
\end{definition}
\begin{remark}\label{RmkLog}{\rm
It is easy to directly verify that the principal quaternionic logarithm coincides with the principal complex  logarithm on the complex plane $\mathbb C_I$, for any $I\in \mathbb S$.
}
\end{remark}
We now go back to the properties of intrinsic slice regular functions.
The following result has an elementary proof which is left to the reader.
\begin{proposition}[Algebraic properties]\label{teofNU}
Let $U$, $U'$ be two open sets  in $\mathbb H$.
\begin{enumerate}
\item[(1)]
Let $f$ and $g\in \mathcal{N}(U)$, then $f+g\in \mathcal{N}(U)$ and $fg=gf\in \mathcal{N}(U)$.
\item[(2)]
Let $f$ and $g\in \mathcal{N}(U)$ such that $g(q)\not=0$ for all $q\in U$, then $g^{-1}f=fg^{-1}\in \mathcal{N}(U)$.
\item[(3)]
Let $f\in \mathcal{N}(U')$,
$g\in \mathcal{N}(U)$ with $g(U)\subseteq U'$.
Then
$f(g(q))$ is slice regular for $q\in U$.
\end{enumerate}
\end{proposition}

\begin{proposition}\label{pro20}
Let $U\subseteq\mathbb H$ be an axially symmetric s-domain and let $\{1,I,J,IJ\}$ be a basis of $\mathbb{H}$. Then
$$\mathcal{R}(U)=\mathcal{N}(U)  \oplus \mathcal{N}(U) I \oplus  \mathcal{N}(U)  J \oplus \mathcal{N}(U) IJ.
$$
\end{proposition}
\begin{proof} By the Refined Splitting Lemma, there exist four functions $h_0, h_1, h_2,h_3$ holomorphic intrinsic on $U\cap\mathbb C_I$ such that
$$f_I(z)=h_0(z)+h_1(z) I  +  h_2(z) J  +h_3(z) IJ,$$
and $f={\rm ext}(f_I)=\ext(h_0)+\ext(h_1) I  + \ext( h_2) J  + \ext( h_3 ) IJ$.
\\
Since the extension of a complex intrinsic function
 is a  quaternionic intrinsic function, we obtain
$\mathcal{R}(U)=\mathcal{N}(U)  + \mathcal{N}(U) I   + \mathcal{N}(U) J+ \mathcal{N}(U) IJ$.

To show that the sum is a direct sum,
suppose that   $f\in \mathcal{N}(U)\cap  \mathcal{N}(U) I $. Then there exists $g\in  \mathcal{N}(U)$, such that   $f=gI$.  So there exist $h_1,h_2$ complex intrinsic and holomorphic such that
$f={\rm ext}(h_1)$, and $g=\ext (h_2)$. Then $h_1 =h_2I$, and for any $q\in U_{I}$ one has
$$h_2(q)I=h_1(q)=\overline{h_1(\bar q)} = \overline{h_2(\bar q)I} =-  h_2(q)I , $$
then $h_2=h_1=0$, and $\mathcal{N}(U)\cap  \mathcal{N}(U) I =\{0\}$.

Similarly one can see that all the other intersections between
$\mathcal{N}(U)$,  $\mathcal{N}(U)I$, $\mathcal{N}(U)J$, $\mathcal{N}(U)IJ$, are $\{0\}$ and the statement follows.
\end{proof}

\section{Composition of power series}
In this section we first introduce and study a notion of composition of slice regular functions, see also \cite{ggs}.
As it is well know,  the composition  $f\circ g$ of two slice regular functions is not, in general, slice regular, unless $g$ belongs to the subclass of quaternionic intrinsic functions. Our choice of the notion of composition is based on the fact that slice regularity is not preserved by the pointwise product, but is preserved by the  $\star$-product.
Thus the power of a function is slice regular only if it is computed with respect to the $\star $-product and we will write $(w(q))^{\star n}$ to denote that we are taking the $n$-th power with respect to this product.
To introduce the notion of composition, we first treat the case of formal power series.
\begin{definition}\label{bullet}
Denoting $g(q)=\sum_{n=0}^{\infty} q^{n}a_{n}$ and $w(q)=\sum_{n=1}^{\infty}q^{n} b_{n}$.  We define
$$(g \bullet w)(q)=\sum_{n=0}^\infty (w(q))^{\star n}a_n.$$
\end{definition}

\begin{remark}{\rm When $w \in {\mathcal{N}}(B(0;1))$, then $g\bullet w=g\circ w$ where $\circ$ represents the standard composition of two functions.
Moreover, when $w \in {\mathcal{N}}(B(0;1))$  then $(w(q))^{\star n}=(w(q))^n$ so, in particular, $q^{\star n}=q^n$.
}
\end{remark}

\begin{remark}\label{cartanrem}{\rm The {\em order} of a series $f(q)=\sum_{n=0}^{\infty} q^{n}a_{n}$ can be defined as in \cite{cartan} and we denote it by $\omega(f)$. It is the lowest integer $n$ such that $a_n\not=0$ (with the convention that  the order of the series identically equal to zero is $+\infty$). Assume to have a family $\{f_i\}_{i\in \mathcal I}$ of power series where $\mathcal I$ is a set of indices. This family is said to be {\em summable} if for any $k\in\mathbb N$, $\omega(f_i)\geq k$ for all except a finite number of indices $i$. By definition, the sum
of $\{f_i\}$ where $f_i(q)=\sum_{n=0}^{\infty} q^{n}a_{i,n}$ is
$$
f(q)=\sum_{n=0}^{\infty} q^{n}a_{n},
$$
where $a_n=\sum_{i\in\mathcal I}a_{i,n}$.  The definition of $a_{n}$ is well posed since our hypothesis guarantees that for any $n$ just a finite number of $a_{i,n}$ are nonzero.
}
\end{remark}

\begin{remark}{\rm
 In Definition \ref{bullet} we require the hypothesis $b_0=0$. This is necessary in order to guarantee that the minimum power of $q$  in the term $(w(q))^{\star n}$ is at least $q^n$ or, in other words, that $\omega(w(q)^{\star n})\geq n$ for all indices. With this hypothesis, the series $\sum_{n=0}^\infty (w(q))^{\star n}a_n$ is summable according to Remark \ref{cartanrem}, and we can regroup the powers of $q$.}
\end{remark}

Let us consider the following example:
let $f(q)=q^2 c$, $g(q)=qa$ and $w(q)=q^2b$. We have
$$((f\bullet g) \bullet w)(q)=q^4 b^2a^2c$$ while $$(f\bullet (g \bullet w))(q)=q^4babac$$
so
$$
(f\bullet g) \bullet w\not=
f\bullet (g \bullet w).
$$
Thus we conclude that:
\begin{proposition}
The composition $\bullet$  is, in general, not associative.
\end{proposition}
 However, we will prove that the composition is associative in some cases and to this end we need a preliminary Lemma.

\begin{lemma}\label{propertybullet}
Let $f_1(q)=\sum_{n=0}^{\infty} q^{n}a_{n}$, $f_2(q)=\sum_{n=0}^{\infty} q^{n}b_{n}$, and $g(q)=\sum_{n=1}^{\infty} q^{n}c_{n}$. Then:
\begin{enumerate}
\item[(1)] $(f_1+f_2)\bullet g= f_1\bullet g+f_2\bullet g$;
\item[(2)] if $g$ has real coefficients $(f_1\star f_2)\bullet g=(f_1 \bullet g)\star (f_2 \bullet g)$;
\item[(3)] if $g$ has real coefficients $f^{\star n}\bullet g=(f \bullet g)^{\star n}$.
\item[(4)] if $\{f_i\}_{i\in\mathcal I}$ is a summable family of power series then $\{f_i\bullet g\}_{i\in\mathcal I}$ is summable and
 $(\sum_{i} f_i) \bullet g=\sum_i (f_i\bullet g).$
 \end{enumerate}
\end{lemma}
\begin{proof}
Point $(1)$ follows from
\[
\begin{split}
(f_1+f_2)\bullet g&=
\sum_{n=0}^\infty g^{\star n}(a_{n}+b_n)\\
&=\sum_{n=0}^\infty g^{\star n}a_{n}+\sum_{n=0}^\infty g^{\star n}b_n\\
&= f_1\bullet g+f_2\bullet g.
\end{split}
\]
 To prove $(2)$, recall that $(f_1\star f_2)(q)= \sum_{n=0}^\infty q^{n}(\sum_{r=0}^n a_rb_{n-r})$ and  taking into account that the coefficients of $g$ are real we have:
$$
((f_1\star f_2)\bullet g)(q)=\sum_{n=0}^\infty (g(q))^{n}(\sum_{r=0}^n a_rb_{n-r}).
$$
The statement follows since
\[
\begin{split}
(f_1 \bullet g)(q)\star (f_2 \bullet g)(q)&=\left(\sum_{n=0}^\infty g(q)^{n} a_n\right)\star \left(\sum_{m=0}^\infty g(q)^{m} b_m\right)\\
&= \sum_{n=0}^\infty g(q)^{n}  (\sum_{r=0}^n a_rb_{n-r}).
\end{split}
\]
 To prove point $(3)$ we use induction. The statement is true for $n=2$ since it follows from $(2)$. Assume that the assertion is true for the $n$-th power. Let us show that it holds for $n+1$-th power.
Let us compute
\[
(f^{\star (n+1)}\bullet g)(q)=( (f^{\star (n)}\star f)\bullet g)(q)\overset{(2)}=(f \bullet g)^{\star n}\star (f\bullet g)=(f \bullet g)^{\star (n+1)},
\]
and the statement follows.\\
To show $4$ we follow \cite[p. 13]{cartan}. Let $f_i(q)=\sum_{n=0}^\infty q^n a_{i,n}$ so that, by definition
$$
\sum_{i\in\mathcal I} f_i(q)= \sum_{n=0}^\infty q^n (\sum_{i\in\mathcal I}a_{i,n}).
$$
Thus we obtain
\begin{equation}\label{e1}
(\sum_{i\in\mathcal I} f_i(q))\bullet g = \sum_{n=0}^\infty g(q)^{\star n} (\sum_{i\in\mathcal I}a_{i,n})
\end{equation}
and
\begin{equation}\label{e2}
\sum_{i\in\mathcal I} (f_i\bullet g)(q)= \sum_{i\in\mathcal I} (\sum_{n=0}^\infty  g(q)^{\star n} a_{i,n}).
\end{equation}
By hypothesis on the summability of $\{f_i\}$, each power of $q$ involves just a finite number of the coefficients $a_{i,n}$ so we can apply the associativity of the addition in $\mathbb H$ and
so (\ref{e1}) and (\ref{e2}) are equal.
\end{proof}
\begin{proposition}\label{assoc}
If  $f(q)=\sum_{n=0}^{\infty} q^{n}a_{n}$,
$g(q)=\sum_{n=1}^{\infty} q^{n}b_{n}$, $w(q)=\sum_{n=1}^{\infty}q^{n} c_{n}$ and $w$ has real coefficients, then
$
(f\bullet g) \bullet w = f\bullet (g \bullet w).
$
\end{proposition}

\begin{proof}
We follow the proof of Proposition 4.1 in \cite{cartan}.
We first prove the assertion in the special case in which $f(q)=q^n a_n$.
We have:
$$
((f\bullet g) \bullet w )(q)=(g^{\star n}\bullet w)a_n
$$
and
$$
(f\bullet (g \bullet w))(q)= (g \bullet w)^{\star n}(q) a_n.
$$
Lemma \ref{propertybullet}, point $(3)$ shows that $g^{\star n}\bullet w=(g\bullet w)^{\star n}$ and the equality follows.
The general case follows by considering $f$ as the sum of the summable family $\{q^na_n\}$ and using the first part of the proof:
\[
(f\bullet g)\bullet w =\sum_{n=0}^\infty (g^{\star n}\bullet w) a_n =\sum_{n=0}^\infty (g\bullet w)^{\star n} a_n= f\bullet (g\bullet w).
\]
 \end{proof}
The discussion, so far, was on formal power series without specifying the set of convergence. We now consider this aspect, by proving the following result which is classical for power series with coefficients in a commutative ring, see \cite{cartan}.

 \begin{proposition} Let $f(q)= \sum_{n=0}^\infty q^n a_n$ and $g(q)= \sum_{n=1}^\infty q^n b_n$ be convergent in the balls of nonzero radius $R$ and $\rho$, respectively, and let $h(q)=(f\bullet g)(q)$. Then the radius of convergence of $h$ is nonzero and it is equal to
$\sup\{r> 0 ; \sum_{n=1}^{\infty} r^{n}||b_{n}||<R \}$.
\end{proposition}
\begin{proof}
Let us compute
\begin{equation}\label{recursion}
\begin{split}
\left| \left(\sum_{m=1}^\infty q^m b_m\right)\star \left(\sum_{m=1}^\infty q^m c_m\right) \right| &=| \sum_{m=1}^\infty q^{m}(\sum_{r=0}^m b_r c_{m-r}) |\\
&\leq \sum_{m=1}^\infty |q|^{m}(\sum_{r=0}^m |b_r| | c_{m-r}|) \\
&=(\sum_{m=1}^\infty | q|^m |b_m|)(\sum_{m=1}^\infty | q|^m |c_m|),
\end{split}
\end{equation}
and thus the formula
$$
\left| \left(\sum_{m=1}^\infty q^m b_m\right) ^{\star n}\right|\leq  \left(\sum_{m=1}^\infty | q|^m | b_m |\right)^n
$$
is true for $n=2$. The general formula follows recursively by using (\ref{recursion}).
So we have
\begin{equation}\label{norms}
\begin{split}
\left| \sum_{n=0}^\infty \left(\sum_{m=1}^\infty q^m b_m\right)^{\star n} a_n\right| &\leq \sum_{n=0}^\infty  \left| \left(\sum_{m=1}^\infty q^m b_m\right) ^{\star n}\right|  |a_n| \\
&\leq \sum_{n=0}^\infty   \left(\sum_{m=1}^\infty |q|^m |b_m|\right)^{n}  |a_n|.
\end{split}
\end{equation}
 Since the series expressing $g$ is converging on a ball of finite radius, there exists a positive number $r$, sufficiently small, such that $\sum_{n=1}^\infty r^n |b_n| $ is finite.
Moreover, $\sum_{n=1}^\infty r^n |b_n|=r\sum_{n=1}^\infty r^{n-1} |b_n|\to 0$ for $r\to 0$ and so there exists $r$ such that $\sum_{n=1}^\infty r^n |b_n| <R$.
Thus, from (\ref{norms}), we deduce
$$
\sum_{n=0}^\infty   \left(\sum_{m=1}^\infty  r^m |b_m|\right)^{n}  |a_n| = \sum_{m=1}^\infty  r^m \gamma_m <\infty .
$$
Thus we have that $(g\bullet f)(q)=\sum_{m=0}^\infty q^m d_m$ and $| d_m|\leq \gamma_m$ and the radius of convergence of $g\bullet f$ is at least equal to $r$.
\end{proof}
Using this notion of composition, it is also possible to define, under suitable conditions, left and right inverses of a power series:
\begin{proposition} Let $g:\, B(0;R)\to \mathbb{H}$, $R>0$, be a function slice regular of the form $g(q)=\sum_{n=0}^\infty q^n a_n$.
\begin{enumerate}
\item[(1)] There exists a power series
$g_r^{-\bullet}(q)=\sum_{n=0}^\infty q^n b_n$  convergent in a disc with positive radius, such that $(g\bullet g_r^{-\bullet})(q)=q$ and $g_r^{-\bullet}(0)=0$ if and only if $g(0)=0$ and $g'(0)\not=0$.
\item[(2)] 
There exists a power series $g_l^{-\bullet}(q)=\sum_{n=0}^\infty q^n b_n$ convergent in a disc with positive radius,
such that $(g_l^{-\bullet}\bullet g)(q)=q$ and $g_l^{-\bullet}(0)=0$ if and only if $g(0)=0$ and $g'(0)\not=0$.
\end{enumerate}
\end{proposition}

\begin{proof} To prove $(1)$, assume that $g_r^{-\bullet}$ exists. Then $\sum_{n=0}^\infty\left(\sum_{m=1}^\infty q^m b_m\right) ^{\star n} a_n =q$. By explicitly writing the terms of the equality
we see that we have
\begin{equation}\label{composition series}
a_0+ \left(\sum_{m=1}^\infty q^m b_m\right)a_1 + \left(\sum_{m=1}^\infty q^m b_m\right) ^{\star 2}a_2+\ldots +\left(\sum_{m=1}^\infty q^m b_m\right) ^{\star n}a_n+\ldots =q
\end{equation}
and so to have equality it is necessary that $a_0=0$, i.e. $g(0)=0$, and $b_1a_1=1$ and so $a_1\not=0$, i. e. $g'(0)\not=0$.
To prove that the condition is sufficient, we observe that for $n\geq 2$, the coefficient of $q^n$ is zero on the right hand side of (\ref{composition series}) while on the left hand side it is
given by
\begin{equation}\label{coefficients}
b_n a_1 + P_n(b_1,\ldots , b_{n-1},a_2,\ldots ,a_n),
\end{equation}
thus we have $$b_n a_1 + P_n(b_1,\ldots , b_{n-1},a_2,\ldots ,a_n)=0,$$ where the polynomials $P_n$ are linear in the $a_i$'s and they contain all the possible monomials $b_{j_1}\ldots b_{j_r}$ with $j_1+\ldots +j_r=n$ and thus also $b_1^n$.   In particular we have:
$b_1a_1=1$ and so $ b_1=a_1^{-1}$, then
$$b_2a_1 +b_1^2a_2=0$$ and so $$b_2=-a_1^{-2}a_2a_1^{-1}.$$ Using induction, we can compute $b_n$ if we have computed $b_1,\ldots ,b_{n-1}$  by putting (\ref{coefficients}) equal to $0$ and using the fact that $a_1$ is invertible. This concludes the proof since the function $g_r^{-\bullet}$ is a right inverse of $g$ by construction.\\
We now show that $g_r^{-\bullet}$ converges in a disc with positive radius following the proof of  \cite[Proposition 9.1]{cartan}. Construct a power series with real coefficients $A_n$ which is a majorant of $g$ as follows: set $$\bar g(q)=q A_1-\sum_{n=2}^\infty q^n A_n$$ with $A_1=|a_1|$ and $A_n\geq |a_n|$, for all $n\geq 2$. It is possible to compute the inverse  of $\bar g$ with respect to the (standard) composition to get the series $\bar g^{-1}(q)=\sum_{n=1}^\infty q^n B_n$. The coefficients $B_n$ can be computed with the formula
$$
B_n A_1 + P_n(B_1,\ldots , B_{n-1},A_2,\ldots ,A_n)=0,
$$
analog of (\ref{coefficients}). Then we have
\[
\begin{split}
&B_1=A_1^{-1}=|a_1|^{-1},\\
 &B_2=A_1^{-2}(-A_2)A_1^{-1}\geq |a_1|^{-2}|a_2|\cdot |a_1|^{-1}=|b_2|
 \end{split}
 \]
 and, inductively
$$B_n=Q_n(A_1,\ldots ,A_n)\geq Q_n(|a_1|,\ldots ,|a_n|) =|b_n|.$$
We conclude that the radius of convergence of $g_r^{-\bullet}$ is greater than or equal to the radius of convergence of $\bar g^{-1}$ which is positive, see \cite[p. 27]{cartan}.

Point $(2)$ can be proved with similar computations and the function $g_l^{-\bullet}$ so obtained is a left inverse of $g$.
\end{proof}
\begin{remark}\label{funzioni*}{\rm
As we have seen in Remark \ref{transcendent}, the transcendental functions cosine, sine, exponential are entire slice regular functions. Let $f(q)$ be another entire function, for example a polynomial. It is then possible to define the composed functions
$$
\exp_\star(f(q))=e_{\star}^{f(q)}=\sum^{\infty}_{n=0}\frac{(f(q))^{\star n}}{n!},
$$
$$
\sin_{\star} (f(q))=\sum^{\infty}_{n=0}(-1)^{n}\frac{(f(q))^{\star 2n+1}}{(2n+1)!},
$$
$$
\cos_{\star} (f(q))=\sum^{\infty}_{n=0}(-1)^n\frac{(f(q))^{\star 2n}}{(2n)!}.
$$
}
\end{remark}

\vskip 1 truecm
\noindent{\bf Comments to Chapter 2}. The material in this chapter comes from several papers. The theory of slice regular functions started in \cite{MR2227751} and \cite{GS} for functions defined at ball centered at the origin and then evolved in a series of papers, among which we mention \cite{MR2742644},
\cite{MR2555912}, \cite{CGeSA}, into a theory on axially symmetric domains. The theory was developed also for functions with values in a Clifford algebra, see \cite{MR2520116}, \cite{CSSd}, for functions with values in a real alternative algebra, see \cite{gp}. The composition of slice regular functions which appears in this chapter is taken from \cite{ggs}.

\chapter{Slice regular functions: analysis}
\section{Some integral formulas}
In this section we collect the generalizations of the Cauchy and Schwarz integral formulas to the slice regular setting. We begin by stating a result which can be proved with standard techniques, see \cite{bds}.
\begin{lemma}\label{lemma1_qua}
Let $f$, $g$ be quaternion valued, continuously (real) differentiable functions on an open
set $U_I$ of the plane $\mathbb{C}_I$.  Then for every open set $W_I\subset
U_I$ whose boundary is
a finite union of
continuously differentiable Jordan curves,  we have
$$
\int_{\partial W_I} g ds_I f=2\int_{W_I} ((g\overline{\partial}_I)f+g(\overline{\partial}_I f)) d\sigma
$$
where $s=x+Iy$ is the variable on $\mathbb{C}_I$, $ds_I=-I ds$ and
$d\sigma=dx\wedge dy$.
\end{lemma}
An immediate consequence of this lemma is the following:
\begin{corollary}\label{int_nullo_qua}
Let $f$ and $g$ be  a left slice regular and a right slice regular function, respectively, on an open set $U\in \mathbb{H}$.
For any $I\in\mathbb{S}$ let $U_I=U\cap\mathbb C_I$. For
every open $W_I\subset
U\cap \mathbb C_I$ whose boundary is
a finite union of
continuously differentiable Jordan curves,  we have:
$$
\int_{\partial W_I} g ds_I f=0.
$$
\end{corollary}
We are now ready to state the Cauchy formula (see \cite{CSS} for its proof):
\begin{theorem}\label{Cauchynuovo}\index{Cauchy!integral formula}
Let $U \subseteq \mathbb{H}$ be an axially symmetric slice domain  such that
$\partial (U \cap \mathbb{C}_I)$ is union of a finite number of
continuously differentiable Jordan curves, for every $I\in\mathbb{S}$. Let $f$ be
a slice regular function on an open set containing $\overline{U}$ and, for any $I\in \mathbb{S}$,  set  $ds_I=-Ids$.
Then for every $q=x+Iy_q\in U$ we have:
\begin{equation}\label{integral}
 f(q)=\frac{1}{2 \pi}\int_{\partial (U \cap \mathbb{C}_I)} -(q^2-2{\rm Re}(s)q+|s|^2)^{-1}
 (q-\overline{s})ds_I f(s).
\end{equation}
Moreover
the value of the integral depends neither on $U$ nor on the  imaginary unit
$I\in\mathbb{S}$.
\end{theorem}
The function
$$
S_L^{-1}(s,q)=-(q^2-2{\rm Re}(s)q+|s|^2)^{-1}
 (q-\overline{s})
$$
is called the Cauchy kernel for left slice regular functions. \index{Cauchy kernel}
It is slice regular in $q$ and right slice regular in $s$ for $q,s$ such that
$q^2-2{\rm Re}(s)q+|s|^2\not=0$.\\
In the case of right slice regular functions, the Cauchy formula is written in terms of the Cauchy kernel
$$
S_R^{-1}(s,q)=-(q-\overline{s})(q^2-2{\rm Re}(s)q+|s|^2)^{-1}.
$$
An immediate consequence of the Cauchy formula is the following result:
\begin{theorem}{(Derivatives using the slice regular Cauchy kernel) }
Let $U\subset\mathbb{H}$ be an axially symmetric slice domain.
Suppose
$\partial (U\cap \mathbb{C}_I)$ is a finite union of
continuously differentiable Jordan curves  for every $I\in\mathbb{S}$.
Let $f$ be a slice regular function on $U$ and set  $ds_I=ds/ I$.
 Let $q$,
$s$.
Then
 $$
f^{(n)}(q)
=\frac{n!}{2 \pi}
\int_{\partial (U\cap \mathbb{C}_I)}  (q^2-2{\rm Re}(s) q+|s|^2)^{-n-1} (q-\overline{s})^{(n+1)\star } ds_I f(s)
$$
\begin{equation}\label{quattordici_qua}
=\frac{n!}{2 \pi}
\int_{\partial (U\cap \mathbb{C}_I)}  [S^{-1}(s,q)(q-\overline{s})^{-1}]^{n+1}
(q-\overline{s})^{(n+1)\star } ds_I f(s)
\end{equation}
where
\begin{equation}\label{stellina_qua}
(q-\overline{s})^{n\star }=\sum_{k=0}^{n}\frac{n!}{(n-k)!k!} q^{n-k}\overline{s}^k,
\end{equation}
is the $n$-th power with respect to the $\star $-product. Moreover
the value of the integral depends neither on $U$ nor on the  imaginary unit
$I\in\mathbb{S}$.
\end{theorem}
\begin{proposition}\label{Cauchyestimates} (Cauchy estimates)\index{Cauchy estimates} Let $f : U \to\mathbb H$ be a slice regular function
and let $q \in U\cap \mathbb C_I$.
 For all discs $B_I (q,R) = B(q,R)\cap\mathbb C_I$, $R > 0$ such that $\overline{B_I (q,R)}\subset
U\cap\mathbb C_I$ the following formula holds:
$$
|f^{(n)}(q)|
\leq \frac{n!}{R^n} \max_{s\in\partial B_I(q,R)} |f(s)| .
$$
\end{proposition}
\begin{proof} Let $\gamma (t) = q + R e^{2\pi It}$, $t\in[0,2\pi)$, then
$$
|f^{(n)}(q)|\leq n ! \left| \frac{1}{2\pi}
\int_{\gamma}
ds_I (s - q)^{-(n+1)}
\right|
\max_{s\in\partial B_I(q,R)} |f(s)| \leq \frac{n!}{R^n} \max_{s\in\partial B_I(q,R)} |f(s)|.
$$
\end{proof}
If $f$ is an entire regular function the Cauchy estimates yield the Liouville Theorem.
\begin{theorem} (Liouville) \index{Liouville!theorem} Let $f$ be a quaternionic entire function. If $f$ is
bounded then it is constant.
\end{theorem}
\begin{proof}
 Let us consider $q = 0$ and let $R > 0$. We use the Cauchy estimates to show that if $R \to\infty$ then all the derivatives $f^{(n)}(0)$ must vanish for $n>0$. We conclude that $f(q) =f(0)$ for all $q\in\mathbb H$.
\end{proof}
In the sequel we will need the Schwarz formula. We first prove the result for slice regular intrinsic functions and then in the general setting, in both cases on a slice, namely on a complex plane $\mathbb C_I$.
As customary, given an open set $U\subseteq \mathbb H$, we will denote by $U_I$ the set $U\cap\mathbb C_I$. To say that a function is harmonic in $U_I$ means that the functions is harmonic in the two variables $x,y$ if we denote by $x+Iy$ the variable in $\mathbb C_I$. Note that in the sequel, we will denote the variable of the complex plane $\mathbb C_I$ in polar coordinates as $re^{I\theta}$ and thus we will write $f(re^{I\theta})=\alpha(r\cos\theta, r\sin\theta)+I \beta(r\cos\theta, r\sin\theta)$.

\begin{proposition}\label{kgftv}
Let  $f\in \mathcal{N}(U) $,   {let $q \in U$}, and   let  $\alpha,\beta\in C^2(U_{I_{ {q} }},\mathbb R) $ be harmonic functions on $U_{I_{ {q} } }$ such that $f_{I_{ {q} }} = \alpha+{I_{ {q} }}\beta$. Assume that for a suitable $\delta>0$ the disk  $|z-q| \leq \delta$ is contained in $U_{I_{ {q} }}$. Then  there exists a real number $b$  such that, for any $z\in U_{I_{q}}$ with  $|z-q| <r< \delta$, the following formula holds
 $$\begin{array}{l}
\displaystyle f (z) = {I_{q}} b  +  \frac{1}{2\pi} \int_{0}^{2\pi}   \frac{ r e^{I_{q} \varphi}  + (z-q) }{ r e^{{ I_{q}} \varphi} -(z-q)  } \alpha(  r\cos\varphi, r\sin\varphi) d\varphi,
\end{array} $$
for any $z\in U_{I_q}$ such that $|z-q|<r<\delta$
\end{proposition}
\begin{proof}
    We use the classical Schwarz formula for holomorphic functions on the complex plane $\mathbb{C}_{I_{q}}$, assuming that the pair of points $z,q\in U_{I_{q}}$ satisfy the hypothesis. We obtain
$$\begin{array}{l}
\displaystyle f (z) =   f_{I_{q}}(z)= {I_{q}} b  +  \frac{1}{2\pi} \int_{0}^{2\pi}  \frac{ r e^{I_{q} \varphi}  + (z-q) }{ r e^{{ I_{q}} \varphi} -(z-q)  }  \alpha(  r, \varphi) d\varphi.
\end{array} $$
\end{proof}

\begin{theorem}[The Schwarz formula on a slice]
Let $U $ be an open set in $\mathbb{H}$, $f\in \mathcal {R}(U )$  and let $q \in U $.  Assume that for a suitable $\delta>0$  the disk  $|z-q| \leq \delta$ is contained in $U _{I_ { {q}}}$. Then  there exist $b\in\mathbb{H}$  and  a harmonic $\mathbb{H}$-valued function $\alpha$, such that for any $z\in U _{I_{q}}$ with  $|z-q| <r< \delta$
\begin{equation}\label{PoissonRn_quat}
\displaystyle f(z)= {I_{q}} b
+  \frac{1}{2\pi} \int_{0}^{2\pi}  \frac{ r e^{I_{q} \varphi}  + (z-q) }{ r e^{{ I_{q}} \varphi} -(z -q)  }      \alpha( r\cos\varphi, r\sin\varphi ) d\varphi.
\end{equation}
\end{theorem}
\begin{proof}
Let $f$ be a slice regular function on $U$ and let us set $I_q=I$ for the sake of simplicity. We use the Refined Splitting Lemma to write $f(z)=f_I(z)=F_0(z)+F_1(z)I
+F_2(z)J + F_3(z)IJ$ where $F_{\ell}$ are holomorphic intrinsic functions, i.e. $F_\ell (\bar z)=\overline{F_\ell (z)}$ for $z\in\ U\cap \mathbb C_I$ and $J\in\mathbb S$ is orthogonal to $I$. Let $\alpha_\ell,\beta_\ell \in C^2(U_{I},\mathbb R)$ be harmonic functions  such that $F_\ell = \alpha_\ell+ I \beta_\ell$, $\ell=1\dots, 3$.
By the complex Schwarz formula, there exists $b_\ell \in\mathbb R$ such that
$$\displaystyle F_\ell(z)= {I} b_\ell  +  \frac{1}{2\pi} \int_{0}^{2\pi}  \frac{ r e^{I \varphi}  + (z-q) }{ r e^{{ I} \varphi} -(z-q)  }  \alpha_\ell ( r\cos\varphi, r\sin\varphi ) d\varphi.
$$
Then, by setting $I_0=1$, $I_1=I$, $I_2=J$, $I_3=IJ$, we have
\[\begin{split}
\displaystyle f(z)&= f_I(z)=\sum_{\ell=0}^3 F_\ell (z) I_\ell\\
&={I} (\sum_{\ell=0}^3 b_\ell I_\ell)
+  \frac{1}{2\pi} \int_{0}^{2\pi}  \frac{ r e^{I_{q} \varphi}  + (z-q) }{ r e^{{ I_{q}} \varphi} -(z-q)  }  \left(  \sum_{\ell=0}^3    \alpha_\ell ( r\cos\varphi, r\sin\varphi) {I}_\ell   \right) d\varphi\\
&\displaystyle= {I} b
+  \frac{1}{2\pi} \int_{0}^{2\pi}  \frac{ r e^{I \varphi}  + (z-q) }{ r e^{I \varphi} -(z-q)  }      \alpha(  r\cos\varphi, r\sin\varphi ) d\varphi,
\end{split}
\]
where
 $$b= \sum_{\ell=0}^3    b_\ell  {I}_\ell, \qquad \alpha ( r\cos\varphi, r\sin\varphi)= \sum_{\ell=0}^3    \alpha_\ell ( r\cos\varphi, r\sin\varphi)  {I}_\ell    .$$
\end{proof}

In order to prove a Schwarz formula in a more general form, we need two lemmas. To state them we set $$\mathbb{D}_I:=\{ z=u+Iv\in \mathbb{C}_I : |z|<1\},$$ $$\overline{\mathbb D}_I:=\{ z=u+Iv\in \mathbb{C}_I : |z|\leq 1\}.$$

\begin{lemma}\label{general schwarz} Let $f:\ \mathbb{D}_I\to \mathbb{C}_I$ be a holomorphic function and let $\alpha$ be its real part, so that
 \begin{equation}\label{spennsac}
f(z)=\frac{1}{2\pi}\int_{0}^{2\pi}\frac{e^{I\, \varphi}+z}{e^{I\, \varphi}-z}\alpha(e^{I\,\varphi})d\varphi .
\end{equation}
Then its slice regular extension to the ball $\mathbb{B}$, still denoted by $f$, is given by
\begin{equation}\label{poisfenne}
f(q)
=\frac{1}{2\pi}\int^{2\pi}_0
(e^{I\, \varphi}-q)^{-\star }\star (e^{I\, \varphi}+q)\,
\alpha(e^{I\, \varphi})d\varphi.
\end{equation}
\end{lemma}

\begin{proof}
Using the extension operator induced by the Representation Formula we have that
$$
f(q)={\rm ext}(f)(q)=\frac 12\left[f(\overline{z})+f(z)+I_qI(
f(\overline{z})-f(z))\right]
$$
$$=\frac{1}{2\pi}\int_0^{2\pi}
\frac 12 \Big[
\frac{e^{I\, \varphi}+z}{e^{I\, \varphi}-z}+ \frac{e^{I\, \varphi}+\overline{z}}{e^{I\, \varphi}-\overline{z}}
+I_q I\left( \frac{e^{I\, \varphi}+\overline{z}}{e^{I\, \varphi}-\overline{z}}-\frac{e^{I\, \varphi}+z}{e^{I\, \varphi}-z}
\right)
\Big]\alpha(e^{I\, \varphi})d\varphi
$$
$$
=\frac{1}{2\pi}\int^{2\pi}_0
(e^{I\, \varphi}-q)^{-\star }\star (e^{I\, \varphi}+q)\,
\alpha(e^{I\, \varphi})d\varphi.
$$
\end{proof}
To prove our next result we need to introduce the Poisson kernel in this framework:
\begin{definition}[Poisson kernel]\index{Poisson kernel}
{\rm
Let $I\in \mathbb{S}$, $0\leq r<1$ and $\theta\in \mathbb{R}$. We call Poisson kernel of the unit ball the function
$$
\mathcal{P}(r,\theta):=\sum_{n\in \mathbb{Z}} r^{|n|}e^{            { n I \theta }      }.
$$
}
\end{definition}
Note that it would have seemed more appropriate to write $\mathcal{P}(re^{I\, \theta})$ but next result shows that the Poisoon kernel does not depend on $I\in\mathbb S$:
\begin{lemma}
Let $I\in \mathbb{S}$, $0\leq r<1$ and $\theta\in \mathbb{R}$. Then the
Poisson kernel belongs to $\mathcal{N}(U)$  and it can be written in the form
\begin{equation}\label{Pois}
\mathcal{P}(r,\theta )=\frac{1-r^2}{1-2r\cos\theta+r^2}={\rm Re}\Big[\frac{1+re^{I\theta}}{1-re^{I\theta}}\Big]
\end{equation}
\end{lemma}
\begin{proof}
The functions $1\pm re^{I\theta}$ obviously map $\mathbb{C}_I$ into itself.
  Moreover, note that $$(1+re^{I\theta})(1-re^{I\theta})^{-1}=(1-re^{I\theta})^{-1}(1+re^{I\theta})$$ and so the ratio in (\ref{Pois}) is well defined.
Note also that
 $$
 (1+re^{I\theta})(1-re^{I\theta})^{-1}=(1+re^{I\theta})\sum_{n\geq 0}r^ne^{nI\theta}=
 1+2\sum_{n\geq  {1}}r^ne^{nI\theta},
 $$
 and
 $$
 {\rm Re}\Big[\frac{1+re^{I\theta}}{1-re^{I\theta}}\Big]=1+2\sum_{n\geq  {1}}r^n\cos( n\theta)=1+\sum_{n\geq {1}}r^n(
 e^{nI\theta}+e^{-nI\theta})=\mathcal{P}(r,\theta).
 $$
 Since
 $$
 \frac{1+re^{I\theta}}{1-re^{I\theta}}
 =\frac{1+re^{I\theta} -re^{-I\theta}-r^2}{1-2r\cos\theta+r^2}
 $$
 we obtain
 $$
  {\rm Re}\Big[\frac{1+re^{I\theta}}{1-re^{I\theta}}\Big] =\frac{1-r^2}{1-2r\cos\theta+r^2}.
  $$
\end{proof}

\begin{lemma}\label{schwarz_M}
Let $\alpha: \overline{\mathbb D}_I\subseteq\mathbb{C}_I\to \mathbb{H}$ be a continuous function that is harmonic on $\mathbb{D}_I$.
Then
\begin{equation}\label{flower}
\alpha(re^{I\theta})=\frac{1}{2\pi}\int_{0}^{2\pi} \mathcal{P}(r,\theta -\varphi)\alpha(e^{I\, \varphi})d\varphi
\end{equation}
for all $I\in \mathbb{S}$, $0\leq r<1$ and $\theta\in \mathbb{R}$.
Moreover, if
$\alpha(z)=\frac{1}{2}(f(z)+f(\bar z))$ for some holomorphic map  $f: \mathbb{D}_I\to \mathbb{H}$, $I\in \mathbb{S}$, such that $f(0)\in\mathbb R$, then
\begin{equation}\label{spennsac1}
f(z)=\frac{1}{2\pi}\int_{0}^{2\pi}\frac{e^{I\, \varphi}+z}{e^{I\, \varphi}-z}\alpha(e^{I\, \varphi})d\varphi,
\end{equation}
where $z=re^{I\theta}$.
\end{lemma}
\begin{proof}
Let us write $\alpha$ as $\alpha=\sum_{\ell =0}^{3} \alpha_\ell I_\ell$ where $\alpha_\ell$ are real valued and
$I_0=1$, $I_1=I$, $I_2=J$, $I_3=IJ$, where $J\in\mathbb S$ is orthogonal to $I$. The functions $\alpha_\ell$ are harmonic since the Laplacian is a real operator. Thus \eqref{flower} follows from the classical Poisson formula applied to $\alpha_\ell$ by linearity. Note also that \eqref{Pois} yields:
\[
\begin{split}
\alpha(re^{I\theta})&=\frac{1}{2\pi}\int_{0}^{2\pi} \mathcal{P}(r,\theta -\varphi)\alpha(e^{I\, \varphi})d\varphi
\\
&=\frac{1}{2\pi}\int_{0}^{2\pi}
\frac{1-r^2}{1-2r\cos(\theta-\varphi)+r^2}\alpha(e^{I\, \varphi})d\varphi
\\
&=\frac{1}{2\pi}\int_{0}^{2\pi} {\rm Re}\Big[\frac{1+re^{I(\theta-\varphi)}}{1-re^{I(\theta-\varphi)}}\Big]\alpha(e^{I\, \varphi})d\varphi,
\end{split}
\]
for all $I\in \mathbb{S}$, $0\leq r<1$ and $\theta\in \mathbb{R}$. Formula (\ref{spennsac}) allows to write the holomorphic functions $f_\ell$. The function $f(z)=\sum_{\ell=0}^3 f_\ell (z) I_\ell$ equals the function in (\ref{spennsac1}) and is in the kernel of the Cauchy-Riemann operator $\partial/\partial \bar z$ by construction.
\end{proof}
The previous results have been proved on the unit disc of the complex plane $\mathbb C_I$ but it is immediate to generalize them to a disc centered at the origin with radius $r>0$.
\\
We can now prove:
\begin{theorem}[Schwarz formula] \label{Schwarz formula}\index{Schwarz formula}
Let  $U $ be an axially symmetric slice domain in $\mathbb{H}$, $f\in \mathcal {R}(U )$  and assume that $0 \in U $. Suppose that $f(0)\in\mathbb{R}$ and  that the ball $B(0;r)$ with center $0$ and radius $r$ is contained in $U$, for a suitable $r$. Let $f(x+I_q y)=\alpha(x,y)+I_q\beta_q(x,y)$ for any $q$ in the ball $B(0;r)\subset\mathbb{H}$, then the following formula holds
\begin{equation}\label{generalq}
f(q)
=\frac{1}{2\pi}\int^{2\pi}_0
(re^{It}-q)^{-\star }\star (re^{It}+q)\,
\alpha(re^{I\, t})dt.
\end{equation}
\end{theorem}

\begin{proof}
Recall that the Representation Formula implies that the function
 $$\alpha(q)=\alpha(x+Iy)=\frac{1}{2}(f(x+Iy)+f(x-Iy))$$ depends on $x,y$ only. By Remark \ref{armoniche} it follows that  $\alpha$ is harmonic. Thus the result follows by extension from formula (\ref{spennsac1}).
\end{proof}

We end this section by proving an analog of the Harnack inequality. We recall that  according to the  classical Harnack inequality, if  $a\in \mathbb C$ and  $ R>0 $ then any  real positive-valued harmonic function $\alpha$ on the disk $|z-a|<R$  satisfies $$\frac{R-r}{R+r}\alpha(a) \leq \alpha(z) \leq \frac{R+r}{R-r} \alpha(a),$$
for any $|z-a|<r<R$. This inequality is a direct consequence of the Poisson Formula.\\
For slice regular functions we have the following:
\begin{proposition}\index{Harnack inequality}\label{Harnack}
Let $f\in \mathcal{N}(U)$, let $q\in U \subseteq\mathbb H$, and let $R>0$ such the  ball  $|p - q|$ is contained in $U$. Assume that, for any $I\in\mathbb S$, there exist two harmonic functions $\alpha, \beta$ on $U_{I}=U\cap\mathbb C_I$, with
real positive values on the disk    $\{ p  \in \mathbb H \ \mid \ | p - q  | < R  \} \cap U_{I}$  such that
$f(u+I_x v) =  \alpha(u,v) + I_x\beta(u,v) $, for all $u+I_x v\in U$. Then for any  $p\in U$ with $|p-q|<r<R$,  one has  $$\frac{R-r}{R+r} | f(q) |\leq |f(p)| \leq \frac{R+r}{R-r}|f(q)|. $$
\end{proposition}
\begin{proof}
Since $\alpha$ and $\beta$ are functions of two variables, we can use  the  classical Harnack inequality and we obtain
$$\left( \frac{R-r}{R+r}\right)^2 \alpha(a,b)^2 \leq \alpha(u,v)^2 \leq \left(\frac{R+r}{R-r}\right)^2 \alpha(a,b)^2 ,$$
$$\left( \frac{R-r}{R+r}\right)^2 \beta(a,b)^2 \leq \beta(u,v)^2 \leq \left(\frac{R+r}{R-r}\right)^2 \beta(a,b)^2 ,$$
where $q=a+I_q b$, $p=u+ I_p v$, and  $|p-q|<r<R$. The result is obtained by adding, respectively, the terms of the previous inequalities and by taking the square root.
\end{proof}

\begin{corollary}
Let $f\in \mathcal{R}(U)$,  $q\in U$, and let $R>0$ be such that the ball $|p- q|$ is contained in $U$. Assume that the elements $I_1, I_2, I_3=I_1I_2 \in\mathbb S$  form a basis of $\mathbb H$ and set $I_0=1$. If  $f = \sum_{\ell=0}^{3} f_\ell I_\ell$ and the functions $f_\ell$ satisfy the hypothesis of Proposition \ref{Harnack}, then for any  $p\in U$ with $|p-q|<r<R$,  one has $$\frac{R-r}{R+r}  \sum_{\ell=0}^{3} |f_\ell(q)|  \leq \frac{R-r}{R+r} \sum_{\ell=0}^{3} |f_\ell(p)|  \leq \frac{R+r}{R-r} \sum_{\ell=0}^{3} |f_\ell(q)| . $$
\end{corollary}

\section{Riemann mapping theorem}
The Riemann mapping theorem, formulated by Riemann back in 1851, is a fundamental result  in the geometric theory of functions of a complex variable. It has a variety of applications not only in the theory of functions of a complex variable, but also in mathematical physics, in the theory of elasticity, and in other frameworks. This theorem cannot be extended in its full generality to the quaternionic setting, see \cite{galsa3}, and in order to understand how it can be generalized, it is useful to recall it in the classical complex case:
\begin{theorem}[Riemann Mapping Theorem]\label{Rthm} \index{Riemann mapping theorem!complex} Let $\Omega \subset \mathbb C$ be a simply connected domain,  $z_0\in \Omega$ and let  $\mathbb{D}=\{z\in \mathbb{C} \, :\,  |z|<1\}$ denote the open unit disk. Then there exists a unique bijective analytic function
$f:\ \Omega \to\mathbb D$ such that $f(z_0)=0$, $f^{\prime}(z_0)>0$.
\end{theorem}
The theorem holds also for simply connected open subsets of the Riemann sphere which both lack at least two points of the sphere.
\\
As we said before, this theorem does not generalize to the case of all simply connected domains in $\mathbb H$. First of all, it is necessary to characterize which open sets can be mapped bijectively onto the unit ball of $\mathbb H$ by a slice regular function $f$ having prescribed value at one point and with prescribed value of the derivative at the same point.
The class of open sets to which the theorem can be extended,
as we shall see, is the class of  axially symmetric  slice domains which are simply connected.
\\
Let us recall that in \cite{duren}, the so-called {\em typically real} functions are functions defined on the open unit disc $\mathbb D$, which take real values on the real line and only there.  These functions have real coefficients when expanded into power series and so they are complex intrinsic  (see  \cite{duren}, p.55). Thus the image of the open unit disc through such mappings is symmetric with respect to the real line.
\\ The following result will be used in the sequel:

\begin{proposition}\label{Tm:Riemann intrinsic} Let $\Omega \subset\mathbb C$ be a simply connected domain such that $\Omega \cap\mathbb R\not=\emptyset$, let
 $x_0\in \Omega \cap \mathbb{R}$ be fixed and let
$f:\ \Omega \to\mathbb D$ with $f(x_0)=0$, $f'(x_0)>0$ be the bijective analytic function as in Theorem \ref{Rthm}. Then $f^{-1}$ is typically real
 if and only if $\Omega$ is symmetric with respect to the real axis.
\end{proposition}
   \begin{proof}
    Let $f:\ \Omega \to\mathbb D$ be the function as in the Riemann mapping theorem. If  $\Omega$ is symmetric with respect to the real axis, then  by the uniqueness of $f$ we have $\overline{f(z)}=f(\bar z)$, see e.g. \cite{ahlfors}, Exercise 1, p. 232. Thus $f$ maps bijectively $\Omega \cap\mathbb R$ onto $\mathbb D\cap\mathbb R$ and
$f^{-1}$ is typically real. Conversely, assume that $f^{-1}:\mathbb D\to \Omega$ is typically real. Then  $\Omega$, being the image of the open unit disc, is symmetric with respect to the real line.
\end{proof}

\begin{corollary}\label{Riemann intrinsic}
Let $\Omega \subset\mathbb C$ be a simply connected domain such that $\Omega \cap\mathbb R\not=\emptyset$,
let $x_0\in \Omega \cap \mathbb{R}$ be fixed and let
$f:\ \Omega \to\mathbb D$ with $f(x_0)=0$, $f'(x_0)>0$ be the bijective analytic function as in Theorem \ref{Rthm}. Then $f^{-1}$ is typically real
 if and only if  $f$ is complex intrinsic.
\end{corollary}
\begin{proof}
If $f^{-1}$ is typically real then $\Omega$ is symmetric with respect to the real line and $\overline{f^{-1}(w)}=f^{-1}(\bar w)$. By setting $w=f(z)$, we have $z=f^{-1}(w)$, and so
$$
f(\bar z)=f(\overline{f^{-1}(w)})=f(f^{-1}(\overline{w}))=\overline{w}=\overline{f(z)}.
$$
 Thus $f$ is complex intrinsic. Conversely, let $f$ be complex intrinsic: then $f$ is defined on set $\Omega$ symmetric with respect to the real line and $f(\bar z)=\overline{f(z)}$. So, by Proposition \ref{Tm:Riemann intrinsic},  $f^{-1}$ is complex intrinsic.
\end{proof}
To summarize the results, we state the following:
\begin{corollary}\label{Riemann intrinsic}
Let $\Omega \subset\mathbb C$ be a simply connected domain such that $\Omega \cap\mathbb R\not=\emptyset$,
let $x_0\in \Omega \cap \mathbb{R}$ be fixed and let
$f:\ \Omega \to\mathbb D$ with $f(x_0)=0$, $f'(x_0)>0$ be the bijective analytic function as in Theorem \ref{Rthm}. Then the following statements are equivalent:
\begin{enumerate}
\item[(1)] $f^{-1}$ is typically real;
\item[(2)]
   $f$ is complex intrinsic;
   \item[(3)] $\Omega$ is symmetric with respect to the real line.
   \end{enumerate}
\end{corollary}

 \begin{remark}\label{extension}{ Let $\Omega \subset \mathbb{C}$ be symmetric with respect to the real axis and let the function $f$ as in the statement of the Riemann mapping theorem be intrinsic, i.e. $\overline{f(\bar z)}=f(z)$. By identifying $\mathbb C$ with $\mathbb C_J$ for some $J\in\mathbb S$ and using the extension formula \eqref{ext}, we obtain a function ${\rm ext}(f)$ which is quaternionic intrinsic (see the proof of Proposition \ref{pro20}). Thus ${\rm ext}(f)\in\mathcal N (U_\Omega )$ where $U_\Omega \subset \mathbb{H}$ denotes the axially symmetric completion of $U$.  }
\end{remark}
\begin{definition} We will denote by $\mathfrak{R}(\mathbb H)$ \index{$\mathfrak{R}(\mathbb H)$} the class of  axially symmetric open sets $U$ in $\mathbb H$ such that $U \cap\mathbb C_I$ is simply connected for every $I\in\mathbb S$.
\end{definition}
 The set $U \cap\mathbb C_I$ is simply connected for every $I\in\mathbb S$ and thus it is connected, so $U$ is a slice domain.\\
In view of Remark \ref{extension} and of Corollary \ref{Riemann intrinsic} we have the following :
\begin{corollary}[Quaternionic Riemann mapping theorem] \index{Riemann mapping theorem!quaternionic}
Let $U \in \mathfrak{R}(\mathbb H)$,  $\mathbb{B}\subset \mathbb{H}$ be the open unit ball and let  $x_{0}\in U \cap \mathbb{R}$.
Then there exists a unique quaternionic intrinsic slice regular function $f: \ U \to \mathbb B$ which is bijective and such that $f(x_0)=0$, $f' (x_0)>0$.
\end{corollary}
\begin{proof}
Let us consider $U _I=U \cap\mathbb C_I$ where $I\in\mathbb S$. Then $U_I$ is simply connected by hypothesis and symmetric with respect to the real line since $U$ is axially symmetric. Let us set $\mathbb D_I=\mathbb B\cap\mathbb C_I$. By Corollary \ref{Riemann intrinsic} there exists a bijective, analytic intrinsic map $f_I: U_I\to \mathbb D_I$ such that $f(x_0)=0$, $f'(x_0)>0$. By Remark \ref{extension}, $f_I$ extends to $f:\ U \to \mathbb B$ and $f\in\mathcal N(U )$. Note that for every $J\in\mathbb S$ we have $f_{|\mathbb C_J}: U_J\to \mathbb D_J$ since $f$ takes each complex plane to itself.
\end{proof}
\begin{remark}
We now show that the Riemann mapping theorem proved above, holds under the optimal hypotheses, i.e. the class
$\mathfrak{R}(\mathbb H)$ cannot be further enlarged. The class of open sets $\mathfrak{R}(\mathbb H)$ contains all the possible simply connected open sets in $\mathbb H$ intersecting the real line for which a map $f: U \to\mathbb B$ as in the Riemann mapping theorem belongs to the class $\mathcal{N}(U )$, namely to the class of functions for which the composition is allowed. In fact, assume that a simply connected open set $U \subset \mathbb H$ is mapped bijectively onto $\mathbb B$ by a map $f\in\mathcal{N}(U )$, with $f(x_0)=0$, $f'(x_0)>0$, $x_0\in\mathbb R$. Since $f\in\mathcal{N}(U )$,  we have that  $f_{|\mathbb C_I}:\ U \cap \mathbb C_I\to \mathbb B\cap\mathbb C_I=\mathbb D_I$ for all $I\in\mathbb S$ and so $f$ takes $U \cap\mathbb R$ to $\mathbb B\cap\mathbb R$.  By its uniqueness, $f_{|\mathbb C_I}$  is the map prescribed by the complex Riemann mapping theorem, moreover $f_{|\mathbb C_I}$ takes $U_I\cap\mathbb R$ bijectively to $\mathbb D_I\cap\mathbb R$ so $f^{-1}$ is totally real. By Corollary \ref{Riemann intrinsic}, it follows that $f$ is complex intrinsic, thus
$U \cap \mathbb C_I$ is symmetric with respect to the real line.
 Since $I\in\mathbb S$ is arbitrary, $U$ must be also axially symmetric, so it belongs to $\mathfrak{R}(\mathbb H)$.
\end{remark}

\section{Zeros of slice regular functions}
Zeros of polynomials with coefficients on one side (say, on the right in the present setting) have been studied
in several papers over the years.
The case of zeros of polynomials, though is a very special case, is perfect to illustrate the general situation that occurs with slice regular functions.

Let us begin by a very well known fact.
Consider the equation:
$$
(q-a)\star(q-b)=q^2- q(a+b)+ab=0.
$$
Then $q=a$ is a zero while $q=b$, in general, is not.
If $b$ does not belong to $[a]$ then, by \eqref{pointwise}, the second zero is $b'=(b-\bar a)^{-1}b(b-\bar a)$.
If $b\in [a]$ but $b\not=\bar a$ then $q=a$ is the only zero of the equation and it has multiplicity $2$.
If $b=\bar a$, then we obtain
\begin{equation}\label{companion}
(q-a)\star(q-\bar a)=q^2- 2{\rm Re}(a) q+|a|^2=0
\end{equation}
and the zero set coincides with the sphere $[a]$.
Sometimes, we will denote the sphere $[a]$ also as $a_0+\mathbb S a_1$, if $a=a_0+I_a a_1$.
When we have a sphere of zeros of $f$, we shall speak of a {\em spherical zero} of $f$.\index{spherical zero}
\\
It is important to note that the factorization of a quaternionic polynomial is not unique, in fact we have:
\begin{theorem} \label{dfactor}
Let $a,b$ be quaternions belonging to two different spheres. Then
$$
(q-a)\star (q-b)=(q-b')\star (q-a')
$$
if and only if $a'=c^{-1}ac$, $b'=c^{-1}bc$ where $c=b- \bar a$.
\end{theorem}
However, we have cases in which the factorization is unique:
\begin{theorem}
The polynomial
\[
P(q)=(q-q_1)\star \cdots \star (q-q_r)
\]
where $q_\ell\in [q_1]$ for all $\ell=2,\ldots, r$ and $q_{\ell +1}\not=\bar q_\ell$, for $\ell=1,\ldots, r-1$ is the unique factorization of $P(q)$.
\end{theorem}

Any quaternionic polynomial admits a factorization as described below. We do not prove this result here since we will prove a result in the more general framework of slice regular functions. We refer the reader to \cite{mjm} for more details.

\begin{theorem} \label{dueuno} Let $P(q)$ be a slice regular polynomial of
degree $m$. Then there exist
$p,m_{1},\ldots,m_{p}\ \in \mathbb{N},$
and $w_{1},\ldots, w_{p}\in \mathbb{H}$, generators of
the spherical roots of $P$, so that
\begin{equation}\label{prodotto}
P(q)=(q^{2}-2q{\rm Re}(w_{1})+|w_{1}|^{2})^{m_{1}}\cdots
(q^{2}-2q {\rm Re}(w_{p})+|w_{p}|^{2})^{m_{p}}Q(q),
\end{equation}
where $Q$ is a
slice regular polynomial with coefficients in
$\mathbb{H}$ having (at most) only non spherical zeroes. Moreover, if
$n=m-2(m_1+\dots +m_p)$ there exist a constant $c\in \mathbb{H}$, $t$
distinct $2-$spheres
$S_1=x_1+y_1\mathbb{S},\ldots,S_t=x_t+y_t\mathbb{S}$, $t$ integers
$n_1,\ldots,n_t$ with $n_1+\dots +n_t=n$, and (for any
$i=1,\ldots,t$) $n_i$ quaternions $\alpha_{ij}\in S_i$,
$j=1,\ldots,n_i$, such that
\begin{equation}\label{fattorizzazione}
Q(q)=[\underset{i=1}{\overset{\star t}\prod} \underset{j=1}{\overset{\star n_i}\prod}(q-\alpha_{ij})]c.
\end{equation}
\end{theorem}

This factorization result extend to slice regular functions thanks to the following result.

\begin{theorem}\label{R-fattorizzazslice}
Let $f$ be a slice regular function on an axially symmetric  slice domain $U$, suppose $f \not \equiv 0$ and let $[x+Iy] \subset U$. There exist $m \in \mathbb N, n \in \mathbb N$, $p_1,\ldots,p_n \in x+y\s$ (with $p_i \neq \bar p_{i+1}$ for all $i \in \{1,\ldots,n-1\}$) such that
\begin{equation}\label{R-factorizationslicedomain}
f(q) = [(q-x)^2+y^2]^m (q-p_1)\star (q-p_2)\star \ldots\star (q-p_n)\star g(q)
\end{equation}
for some slice regular function $g:U \to \hh$ which does not have zeros in the sphere $[x+Iy]$.
\end{theorem}

\begin{proof}
If $f$ is not identically $0$ on $U$, and $f$ vanishes at $[x+Iy]$, then there exists an $m \in \mathbb N$ such that $$f(q) = [(q-x)^2+y^2]^m h(q)$$ for some $h$ not identically zero on $[x+Iy]$. In fact, suppose that it were possible to find, for all $k \in \mathbb N$, a function $h^{[k]}(q)$ such that $f(q) = [(q-x)^2+y^2]^k h^{[k]}(q)$. Then, choosing an $I \in \mathbb S$, the holomorphic map $f_I$ would have the factorization
$$f_I(z) = [(z-x)^2+y^2]^k h^{[k]}_I(z) = [z-(x+yI)]^k[z-(x-yI)]^k h^{[k]}_I(z)$$
for all $k \in \mathbb N$. This would imply $f_I \equiv 0$ and, by the Identity Principle, $f \equiv 0$.
\\
Now let $h$ be a slice regular function on $U$ which does not vanish identically on $[x+Iy]$. By Proposition \ref{zericoniugata}, $h$ has at most one zero $p_1 \in [x+Iy]$. If this is the case then $$g^{[0]}=h(q) = (q-p_1) \star  g^{[1]}(q)$$  for some function $g^{[1]}$ which does not vanish identically on $[x+Iy]$. If for all $k \in \mathbb N$ there existed a $p_{k+1} \in [x+Iy]$ and a $g^{[k+1]}$ such that $$g^{[k]}(q) = (q-p_{k+1})\star g^{[k+1]}$$ then we would have $$h(q) = (q-p_1)\star \ldots\star (q-p_k)\star g^{[k]}(q)$$for all $k \in \mathbb N$. Thus the symmetrization $h^s$ of $h$ would be such that $$h^s(q) = [(q-x)^2+y^2]^k (g^{[k]})^s(q)$$ for all $k \in \mathbb N$. By the first part of the proof, this would imply $h^s \equiv 0$. Thus Proposition \ref{zericoniugata} implies that $h \equiv 0$, which is a contradiction. Thus there exists an $n \in \mathbb N$ such that $g^{[n]}$ does not have zeros in $[x+Iy]$ and, setting $g = g^{[n]}$, we have the statement.
\end{proof}
We now introduce the notion of multiplicity of a root of a slice regular function in the various cases.
\begin{definition}
We say that a  function $f$ slice regular on an axially symmetric slice domain $U$ has a zero at $[q_0]\subset U$, $q_0=x_0+Iy_0$, $y_0\not=0$ with spherical multiplicity $m$, if $m$
is the largest natural number such that
$$
f(q) = [(q - x_0)^2 + y_0^2]^m \star g(q)
$$
for some  function $g$ slice regular on $U$ which does not vanish at $[q_0]$.
\end{definition}
\begin{definition} If a function $f$ slice regular on an axially symmetric slice domain $U$ has a root $q_1\in [q_0]$ in $U$, $q_0=x_0+Iy_0$, $y_0\not=0$, then we
say that $f$ has isolated multiplicity $r$ at $q_1$, if $r$ is the largest natural
number such that there exist $q_2, \ldots, q_r\in[q_0]$, $\bar q_{\ell+1}\not= q_\ell$ for all $\ell=1,\ldots, r-1$ and a function $g$  slice regular in $U$ which does not have zeros in $[q_0]$ such that
$$
f(q) = (q - q_1) \star (q - q_2) \star \ldots \star (q - q_r) [(q - x_0)^2 + y_0^2]^m \star g(q),
$$
where $m\in\mathbb N$.\\
If $q_0\in\mathbb R$ then we call isolated multiplicity of $f$ at $q_0$ the largest
$r\in\mathbb N$ such that
$$
f(q) = (q - q_0)^r  g(q)
$$
for some function $g$ slice regular in $U$ which does not vanish at $q_0$.
\end{definition}
\begin{remark}
{\rm If $f$ vanishes at the points of a sphere $[q_0]$, all the points of that sphere have the same multiplicity $m$ as zeros of $f$, except possibly for one point which may have higher multiplicity. To see an example, it suffices to consider $f(q)=(q-i)\star (q-j) (q^2+1)^3$. All points of the unit sphere $\mathbb S$ have multiplicity $3$ except for the point $q=i$ which has multiplicity $5$.}
\end{remark}

In particular, Theorem \ref{pointwise} implies that for each zero of $f\star g$ in $[q_0]$ there exists a zero of $f$ or a zero of $g$ in $[q_0]$. However, there are examples of products $f\star g$ whose zeros are not in one-to-one correspondence with the union of the zero sets of $f$ and $g$. It suffices to take $f(q)=q-i$ and $g(q)=q+i$ to see that both $f$ and $g$ have one isolated zero, while $f\star g$ has the sphere $\mathbb S$ as its set of zeros.
\\
We now study the relation between the zeros of $f$ and those of $f^c$ and $f^s$. We need two preliminary steps.

\begin{lemma}\label{simmetriazerireale}
Let $U \subseteq \mathbb{H}$ be an axially symmetric  slice domain and let $f \in\mathcal{N}( U)$. If $f(x_0+I_0y_0)=0$ for some $I_0\in\mathbb{S}$, then $f(x_0+Iy_0)=0$ for all $I\in \mathbb{S}.$
\end{lemma}

\begin{proof}
Since $f \in\mathcal{N}( U)$, then
 $f(U _I)\subseteq \mathbb C_I$ for all $I \in \mathbb{S}$. Thus $f(x)$ is real for all $x\in U \cap \mathbb{R}$. The restriction $f_{I_0}:U _{I_0} \to \mathbb{C}_{I_0}$ is a holomorphic function mapping $U \cap \rr$ to $\rr$. By the (complex) Schwarz Reflection Principle, $f(x+yI_0)=\overline{f(x-yI_0)}$ for all $x+yI_0 \in U _{I_0}$.
Since $f(x_0+y_0I_0)=0$, we conclude that $f(x_0-y_0I_0)=0$ and the Representation Formula allows to deduce the statement.
\end{proof}

\begin{lemma}\label{simmetrizzatareale}
Let $U \subseteq \mathbb{H}$ be an axially symmetric  slice domain, let $f : U \to \mathbb{H}$ be a regular function and let $f^s$ be its symmetrization. Then $f^s(U _I) \subseteq \mathbb C_I$ for all $I \in \s$.
\end{lemma}
\begin{proof}
It follows by direct computation from the definition of $f^s$.
\end{proof}

\begin{proposition}\label{zericoniugata}
Let $U \subseteq\mathbb{H}$ be an axially symmetric  slice domain, let $f\in\mathcal{R}(U)$ and $[x_0+Iy_0]\subset U $.
\begin{enumerate}
\item[(1)] The function $f^s$ vanishes identically on $[x_0+Iy_0]$ if and only if $f^s$ has a zero in $[x_0+Iy_0]$, if and only if $f$ has a zero in $[x_0+Iy_0]$;
   \item[(2)] The zeros of $f$ in $[x_0+Iy_0]$ are in one-to-one correspondence with those of $f^c$;
   \item[(3)] The function $f$ has a zero in $[x_0+Iy_0]$ if and only if $f^c$ has a zero in $[x_0+Iy_0]$.
     \end{enumerate}
\end{proposition}

\begin{proof} We prove assertion (1). First of all, we note that by Lemma \ref{simmetriazerireale} and Lemma \ref{simmetrizzatareale}, if $f^s(x_0+I_=y_0) = 0$ for some $I_0\in\mathbb S$ then $f^s(x_0+Iy_0) = 0$ for all $I \in \s$. The converse is trivial.\\
If $q_0=x_0+I_0y_0$ is a zero of $f$ then, by Theorem \ref{pointwise}, also $f^s = f\star f^c$ vanishes at $q_0$.
Conversely, assume that $q_0$ is a zero of $f^s$. By formula \eqref{pointwise} either $f(q_0)=0$ or $f^c$ vanishes at the point
$$
f(q_0)^{-1}  q_0 f( q_0) = x_0 + y_0 [f( q_0)^{-1} I_0 f(q_0)] \in [q_0].
$$
In the first case, we have concluded, in the second case, we recall that $f^s$ vanishes not only at $q_0$ but on the whole sphere $[q_0]$ and so  $f^s(\bar q_0) = 0$. This fact implies that either $f(\bar q_0)=0$ or $$f^c(f(\bar q_0)^{-1} \bar q_0 f(\bar q_0))=0.$$ In the first case, the Representation Formula yields that  $f$ vanishes identically on $[q_0]$, which implies that, by its definition, also $f^c$ vanishes on $[q_0]$. In the second case, $f^c$ vanishes at the point
$$
f(\bar q_0)^{-1} \bar q_0 f(\bar q_0) = x_0 - y_0 [f(\bar q_0)^{-1} I_0 f(\bar q_0)] \in [q_0],
$$
thus $f^c$ vanishes at $[q_0]$ and so also $f$ vanishes at $[q_0]$ and the statement follows.\\
Thus, we have proven that $f$ has a zero in $[q_0]$ if and only if $f^s$ has a zero in $[q_0]$, which leads to the vanishing of $f^s$ on the whole $[q_0]$, which implies the existence of a zero of $f^c$ in $[q_0]$. Since $(f^c)^c=f$, exchanging the roles of $f$ and $f^c$ gives the rest of the statement.
\end{proof}


We now study the distribution of the zeros of regular functions on axially symmetric  slice domains. In order to obtain a full characterization of the zero set of a regular function, we first deal with a special case that will be crucial in the proof of the main result.

\begin{lemma}\label{zeriseriereale}
Let $U \subseteq \mathbb{H}$ be an axially symmetric  slice domain and let $f : U \to \mathbb{H}$ be a slice regular intrinsic function. If $f\not\equiv 0$, the zero set of $f$ is either empty or it is the union of isolated points (belonging to $\mathbb{R}$) and/or isolated 2-spheres.
\end{lemma}

\begin{proof}
We know from Lemma \ref{simmetriazerireale} that the zero set of such an $f$ consists of  2-spheres of the type $x+y\mathbb{S}$, possibly reduced to real points. Now choose $I$ in $\mathbb{S}$ and notice that the intersection of $\mathbb C_I$ with the zero set of $f$ consists of all the real zeros of $f$ and of exactly two zeros for each sphere $x+y\mathbb{S}$ on which $f$ vanishes (namely, $x+yI$ and $x-yI$). If $f\not \equiv 0$ then, by the Identity Principle, the zeros of $f$ in $\mathbb C_I$ must be isolated. Hence the zero set of $f$ consists of isolated real points and/or isolated 2-spheres.
\end{proof}

We now state and prove the result on the topological structure of the zero set of regular functions.

\begin{theorem}[Structure of the Zero Set]\label{structurethm}
Let $U \subseteq\mathbb{H}$ be an axially symmetric  slice domain and let  $f:U \to\mathbb{H}$ be a regular function. If $f$ does not vanish identically, then the zero set of $f$ consists of isolated points or isolated 2-spheres.
\end{theorem}

\begin{proof}
Consider the symmetrization $f^s$ of $f$: by Lemma \ref{simmetrizzatareale}, $f^s$ fulfills the hypotheses of Lemma \ref{zeriseriereale}. Hence the zero set of $f^s$ consists of isolated real points or isolated 2-spheres. According to Proposition \ref{zericoniugata}, the real zeros of $f$ and $f^s$ are exactly the same. Furthermore, each 2-sphere in the zero set of $f^s$ corresponds either to a 2-sphere of zeros, or to a single zero of $f$. This concludes the proof.
\end{proof}
An immediate consequence is:
\begin{corollary}[Strong Identity Principle] \index{Identity principle!strong}
Let $f$ be a function slice regular in an axially symmetric slice domain $U$. If there exists a sphere $[q_0]$ such that the zeros of $f$ in $U\setminus [q_0]$ accumulate to a point of $[q_0]$ then $f$ vanishes on $U$.
\end{corollary}

\section{Modulus of a slice regular function and Eh\-ren\-preis-Malgrange lemma}
We start by proving two results extending the Maximum and Minimum Modulus Principle to  the present setting.
\begin{theorem}[Maximum Modulus Principle]\index{Maximum modulus principle}
Let $U$ be a slice domain and let $f : U \to \hh$ be slice regular. If $|f|$ has a relative maximum at $p \in U$, then $f$ is constant.
\end{theorem}
\begin{proof}
If $p$ is a zero of $f$ then $|f|$ has zero as its maximum value, so $f$ is identically $0$. So we assume $f(p) \neq 0$.  It is not reductive to assume that $f(p) \in \rr, f(p)>0$, by possibly multiplying $f$ by $\overline{f(p)}$ on the right side. Let $I,J \in \mathbb S$ be such that $p$ belongs to the complex plane $\mathbb C_I$ and $I \perp J$. We use the Splitting Lemma to write $f_I = F+GJ$ on $U_I=U\cap\mathbb C_I$. Then, for all $z$ in a neighborhood $V_I=U\cap\mathbb C_I$  of $p$ in $U_I$ we have, since $f(p)$ is real:
$$|F(p)|^2 = |f_I(p)|^2 \geq |f_I(z)|^2 = |F(z)|^2 + |G(z)|^2 \geq |F(z)|^2.$$
Hence $|F|$ has a relative maximum at $p$ and the Maximum Modulus Principle for holomorphic functions of one complex variable allows us to conclude that $F$ is constant and so $F \equiv f(p)$. As a consequence, we have
$$|G(z)|^2 = |f_I(z)|^2-|F(z)|^2 = |f_I(z)|^2-|f_I(p)|^2 \leq  |f_I(p)|^2-|f_I(p)|^2 = 0$$
for all $z \in V_I$, and so $f_I = F \equiv f(p)$ in $V_I$. From the Identity Principle, we deduce that $f \equiv f(p)$ in $U$.
\end{proof}

\begin{theorem}[Minimum Modulus Principle]\index{Minimum modulus principle} \label{minimum}
Let $U$ be an axially symmetric  slice domain and let $f : U \to \hh$ be a slice regular function. If $|f|$ has a local minimum point $p\in U$ then either $f(p)=0$ or $f$ is constant.
\end{theorem}

\begin{proof}
Consider a slice regular function $f : U \to \hh$ whose modulus has a minimum point $p \in U $ with $f(p) \neq 0$. Such an $f$ does not vanish on the sphere $S$ defined by $p$. Indeed, if $f$ vanished at a point $p' \in S$ then $|f_{|_S}|$ would have a global minimum at  $p'$, a global maximum and no other extremal point, which contradicts the hypothesis on $p$. But if $f$ does not have zeroes in $S$, neither $f^s$ does. Hence the domain $U ' = U \setminus {Z}_{f^s}$ of $f^{-\star }$ includes $S$ (where ${Z}_{f^s}$ denotes the set of zeros of $f^s$). Thanks to Theorem \ref{pointwise},
$$|f^{-\star }(q)| = \frac{1}{|f(\tilde q)|}, \qquad q \in U ' $$ for a suitable $\tilde q$ belonging to the sphere of $q$.
 If $|f|$ has a minimum at $p \in x+y\s \subseteq U '$ then $|f |$ has a minimum at $\tilde p \in U '$. As a consequence, $|f^{-\star }|$ has a maximum at $\tilde p$. By the Maximum Modulus Principle, $f^{-\star }$ is constant on $U '$. This implies that $f$ is constant in $U '$ and, thanks to the Identity Principle, it is constant in  $U$.
\end{proof}

We now prove an analog of the Ehrenpreis-Malgrange lemma, giving lower bounds for the moduli of polynomials away from their zeros, for polynomials with quaternionic coefficients.
We begin with a simple result, which deals with the case in which we are interested in finding a lower bound on a sphere centered at the origin. In general, the bounds will be assigned on a $3-$dimensional toroidal hypersurface.
\begin{theorem}\label{EMsphere}
Let $P(q)$ be a slice regular polynomial of degree $m,$ with leading coefficient $a_m.$ Let $p$ be the number of  distinct
spherical zeroes of $P(q)$, let
$t$ be the number of  distinct isolated zeroes and
let $M=p+t.$ Given any $R>0$, we can find a sphere $\Gamma$ centered at the origin
and of radius $r<R$ on which
$$|P(q)|\geq |a_m|\left(\frac{R}{2(M+1)}\right)^m.$$
\end{theorem}
\begin{proof}
By using Theorem \ref{dueuno}, we can decompose $P(q)$ as
 $$P(q)=S(q)Q(q)a_m$$ where $S$ is a product of factors of the form
 $$
 S(q)=(q^{2}-2qRe(w_{1})+|w_{1}|^{2})^{m_{1}}\cdots
(q^{2}-2qRe(w_{p})+|w_{p}|^{2})^{m_{p}}
 $$
and $Q$ as in (\ref{fattorizzazione}). The cardinality of the set $V=\{|q|  : q \in \mathbb{H}, P(q)=0\}$ is at most $M$, in fact some isolated zeros may belong to some spheres of zeros. In any case, there exists a subinterval $[a,b]$ of $[0, R]$ of length at least $\frac{R}{M+1}$ which does not contain any element of $V$. Let $\Gamma$ be the $3-$sphere centered in the origin and with radius $\frac{a+b}{2}.$

We now estimate from below the absolute value of $P(q)$ on a generic point on $\Gamma.$ Since $P(q)=S(q)Q(q)a_m$, we will estimate the absolute values of  both $S(q)$ and $Q(q)$ on $\Gamma$.
To estimate $S(q)$, it is useful to recall that for any pair of quaternions $q$ and $\alpha$, we have

\[
\begin{split}
q^2-2{\rm Re}(\alpha)q+|\alpha|^2&=q^2-\alpha q-\overline \alpha q +\overline \alpha \alpha\\
&=
(q-\alpha)\left(q-(q-\alpha)^{-1}\overline \alpha (q-\alpha)\right);
\end{split}
\]
and so

 \[
 \begin{split}
 |q^2-2{\rm Re}(\alpha)q+|\alpha|^2|&=|(q-\alpha)||q-(q-\alpha)^{-1}\overline \alpha (q-\alpha)|\\
 &\geq ||q|-|\alpha||\cdot||q|-|(q-\alpha)^{-1}\overline \alpha (q-\alpha)||
 \\
 &=||q|-|\alpha||\cdot||q|-|\overline{\alpha}||\\
 &=||q|-|\alpha||^2.
 \end{split}
 \]
Thus we have

\[
\begin{split}
|S(q)|&=|(q^{2}-2q{\rm Re}(w_{1})+|w_{1}|^{2})^{m_{1}}\cdots
(q^{2}-2q{\rm Re}(w_{p})+|w_{p}|^{2})^{m_{p}}|\\
&=|q^{2}-2q{\rm Re}(w_{1})+|w_{1}|^{2}|^{m_{1}}\cdots
|q^{2}-2q{\rm Re}(w_{p})+|w_{p}|^{2}|^{m_{p}}\\
&\geq ||q|-|w_1||^{2m_1}\ldots ||q|-|w_p||^{2m_p}.
\end{split}
\]

To estimate $Q(q)$, we first note that, for suitable quaternions $\alpha_1,\ldots,\alpha_N,$
we can split $Q(q)$ into linear factors as $$Q(q)=(q-\alpha_1)\star \cdots \star (q-\alpha_N),$$
and the estimate for $Q(q)$ can be obtained recursively as follows.
\\
It is immediate that

$$|q-\alpha_N|\geq ||q|-|\alpha_N||.$$

Let $\ell \leq N-1$ be an integer and set
$$h_{\ell +1}(q):=(q-\alpha_{\ell +1})\star \cdots \star (q-\alpha_N).$$
We now assume that for  $\ell \leq N-1$ we have established
\begin{equation}\label{induzione}
|h_{\ell +1}(q)|=|(q-\alpha_{\ell +1})\star \cdots \star (q-\alpha_N)|\geq ||q|-|\alpha_{\ell +1}||\cdots||q|-|\alpha_N||,
\end{equation}
and we proceed to the estimate for

$$|h_{\ell}(q)|=|(q-\alpha_{\ell})\star \cdots \star (q-\alpha_N)|.$$
The associativity of the $\star$-product and Theorem \ref{pointwise} imply that

\[
\begin{split}
|(q-\alpha_{\ell})&\star (q-\alpha_{\ell+1})\star \cdots \star (q-\alpha_N)|\\
&=|(q-\alpha_{\ell})\star ((q-\alpha_{\ell+1})\star \cdots \star (q-\alpha_N))|\\
&=|(q-\alpha_{\ell})\star (h_{\ell +1}(q)|\\
&=|(q-\alpha_{\ell})\cdot h_{\ell +1}((q-\alpha_{\ell})^{-1}q(q-\alpha_{\ell}))|\\
&=|(q-\alpha_{\ell})|\, |h_{\ell +1}((q-\alpha_{\ell})^{-1}q(q-\alpha_{\ell}))|.
\end{split}
\]
Using
$$|(q-\alpha_{\ell})^{-1}q(q-\alpha_{\ell})|=|q|,$$
and (\ref{induzione}), we obtain
$$
|(q-\alpha_{\ell})\star \cdots \star (q-\alpha_N)|\geq ||q|-|\alpha_{\ell}||\cdots||q|-|\alpha_N||.
$$
In conclusion we have:
$$|Q(q)|\geq ||q|-|\alpha_{1}||\cdots||q|-|\alpha_N||.$$
Since each factor in the decomposition of $P(q)$ is bounded below by $\dfrac{R}{2(M+1)}$, the statement follows.
\end{proof}

\begin{remark}{\rm
 The estimate proved in Theorem \ref{EMsphere} holds also if the sphere $\Gamma$ is centered at any real point $q_0.$}
\end{remark}

The next result explains what happens if one attempts to estimate $|P(q)|$ from below, on spheres centered on points $q_0$ which are not real.

\begin{theorem}\label{sferico}
Let $P(q)$ be a slice regular polynomial of degree $m$ with only spherical zeroes, i.e. a polynomial of the form
$$P(q)=(q^{2}-2q{\rm Re}(w_{1})+|w_{1}|^{2})^{m_{1}}\cdots
(q^{2}-2q{\rm Re}(w_{p})+|w_{p}|^{2})^{m_{p}}a_m,$$
with $w_1,\ldots,w_p,a_m \in \mathbb{H}.$ For any $q_0=u+vI \in \mathbb{H}$ and for any $R>0$, there exist $r<R$ and a $3-$dimensional compact hypersurface
 $$\Gamma=\Gamma(q_0,r)=\{x+yI: (x-u)^2+(y-v)^2=r^2 \ \ \hbox{and}\  \ I\in \mathbb{S}\},$$
smooth if $r<v$, such that for every $q\in \Gamma$ it is
$$|P(q)|\geq |a_m|\left(\frac{R}{2(m+1)}\right)^{m}.$$
\end{theorem}
\begin{proof}
Without loss of generality, we can assume $a_m=1.$ The restriction of $P(q)$ to the complex plane $\mathbb C_I$ is a complex polynomial with up to $m$  distinct zeroes. Consider the set $$V=\{|q-q_0|  : q \in \mathbb C_I, P(q)=0\};$$ the cardinality of $V$ is at most $m$. So we can find a subinterval $[a,b]$ of $[0, R]$ of length at least $\frac{R}{m+1}$ which does not contain any element of $V$. Let $\Gamma$ be the circle in $\mathbb C_I$ centered in $q_0$ and with radius $r=\frac{a+b}{2}.$ Then, on $\mathbb C_I$, one has
\[
\begin{split}
|q^{2}-2q{\rm Re}(w)+|w|^{2}|&=|(q-w)(q-\overline{w})|=|q-w||q-\overline{w}|\\
&=|(q-q_0)-(w-q_0)||(q-q_0)-(\overline{w}-q_0)|\\
&\geq ||q-q_0|-|w-q_0||\cdot ||q-q_0|-|\overline{w}-q_0||.
\end{split}
\]
Since $w$ and $\overline w$ are roots of $P(q)$, we have
$$|q^{2}-2q{\rm Re}(w)+|w|^{2}|\geq \left( \frac{R}{2(m+1)}\right)^2$$
and so
$$|P(q)|\geq\left(\frac{R}{2(m+1)}\right)^m.$$
Since $P$ has real coefficients, the estimate is independent of $I$, and the bound we have proved holds on $\Gamma.$
\end{proof}

We now prove the following simple lemma:

\begin{lemma} \label{lagrange}
Let $q_0= u+I_0v$ be a given point in $\mathbb{H}$ and let $w=a+Ib$ be the generic point on the sphere $[w]$. The distance between $w$ and $q_0$ achieves its extremal points at
$w^0=a+I_0b$ and $\overline{w}^{0} = a-I_0b$.
\end{lemma}
\begin{proof}
 Up to a rotation, we may assume that $I_0=i$ so that $q_0=u+iv$. An element $I\in \mathbb{S}$ can be written as $I=\alpha i+\beta j+\gamma k$ with $\alpha^2+\beta^2+\gamma^2=1$ so that $$w=a+b(\alpha i+\beta j+\gamma k).$$ The square of the distance between $w$ and $q_0$ is therefore given by
$$d^2(\alpha, \beta, \gamma)=a^2+u^2-2au+b^2\alpha^2+v^2-2b\alpha v+b^2\beta^2+b^2\gamma^2.$$
Since $\mathbb{S}$ is a compact set, the function $d^2(\alpha, \beta, \gamma)$ admits at least a maximum and a minimum in $[w]$, which can be computed with standard techniques. A quick computation shows that the maximum and minimum are achieved for $\alpha=\pm 1, \beta=\gamma=0$.
\end{proof}

\begin{theorem}\label{isolato}
Let $P(q)$ be a slice regular polynomial of degree $m$ with only isolated zeroes, i.e. a polynomial of the form
$$P(q)=(q-\alpha_1)\star \cdots \star (q-\alpha_m)a_m,$$
with $\alpha_1,\ldots,\alpha_m,a_m \in \mathbb{H}.$ For any $q_0=u+vI_0 \in \mathbb{H}$ and for any $R>0$, there exist $r<R$, and a $3-$dimensional compact hypersurface
$\Gamma=\Gamma(q_0,r)$, smooth if $r<v$, such that for every $q\in \Gamma$ it is
$$|P(q)|\geq |a_m|\left(\frac{R}{2(2m+1)}\right)^{m}.$$
\end{theorem}

\begin{proof} As before, it is not reductive to assume that $a_m=1$.
For $t=1, \ldots m$, define $\alpha^0_t={\rm Re}(\alpha_t)+I_0 {\rm Im}(\alpha_t)$ and consider the following subset of the real numbers
$$V=\{ |\alpha^0_t-q_0|, \ |\overline{\alpha}^0_t-q_0|\ \  \hbox{for}\ \  t=1, \ldots, m \}.$$
Given any $R>0$ there is at least one subinterval $[c, d]$ of $[0, R]$ which does not contain any element of $V$ and whose length is at least $\frac{R}{2m+1}$. Set $r= \frac{c+d}{2}$ and define
$$\Gamma=\{x+Iy\ :\ (x-u)^2+(y-v)^2=r^2 \ \ \hbox{and}\  \ I\in \mathbb{S}\}.$$
We now estimate, for $q\in \Gamma$, the modulus of $P(p)=(q-\alpha_1)\star \cdots \star (q-\alpha_m)$. By Theorem \ref{pointwise}, there exist $$\alpha'_t \in S_{\alpha_t}= {\rm Re}(\alpha_t)+{\rm Im}(\alpha_t)\mathbb{S},$$ for $t=1,\ldots,m$, such that
$$|P(p)|=|(q-\alpha'_1)\cdots (q-\alpha'_m)|=|q-\alpha'_1|\cdots |q-\alpha'_m|.$$
Let us now estimate each factor:
$$|q-\alpha'_t|=|(q-q_0)-(\alpha'_t-q_0)|\geq ||q-q_0|-|\alpha'_t-q_0||.$$
 Note that $||q-q_0|-|\alpha'_t-q_0||$ is either $|q-q_0|-|\alpha'_t-q_0|$ or $|\alpha'_t-q_0|-|q-q_0|$. It is evident that these two cases can be treated in the same way, thus we consider
$$
||q-q_0|-|\alpha'_t-q_0||=|q-q_0|-|\alpha'_t-q_0|.
$$
In order to find a lower bound for this expression, we need a lower bound for $|q-q_0|$ and an upper bound for $|\alpha'_t-q_0|$. By Lemma \ref{lagrange}, and since $\alpha'_t \in S_{\alpha_t}$ we have
$$
|q-q_0|-|\alpha'_t-q_0|\geq |q^0-q_0|-|\widetilde{\alpha}^0_t-q_0|
$$
where $q^0={\rm Re}(q)\pm {\rm Im}(q)I_0$ and $\widetilde{\alpha}^0_t$ is either $\alpha^0_t$ or $\overline{\alpha}^0_t$. By the definition of the set $V$, we obtain that
$$
|q-\alpha'_t|\geq \frac{R}{2(2m+1)},
$$
and
$$|P(q)|\geq |a_m|\left(\frac{R}{2(2m+1)}\right)^{m}.$$
\end{proof}

\noindent As a consequence we can prove a lower bound for general slice regular polynomials:

\begin{theorem}\label{generale}
Let $P(q)$ be a slice regular polynomial of degree $m,$ with leading coefficient $a_m.$ Let $p$ be the number of  distinct
spherical zeroes of $P(q)$, and let
$t$ be the number of its  distinct isolated zeroes.
Given any $q_0=u+vI_0 \in \mathbb{H}$ and any $R>0$, we can find $r<R$ and a compact $3-$dimensional hypersurface
$$\Gamma= \{x+yI: (x-u)^2+(y-v)^2=r^2 \ \ \hbox{and}\  \ I\in \mathbb{S}\},$$
smooth if $R<v$, on which
$$|P(q)|\geq |a_m|\left(\frac{R}{2(p+2t+1)}\right)^m.$$
\end{theorem}
\begin{proof}
The assertion follows from Theorem \ref{sferico} and Theorem \ref{isolato}.
\end{proof}

\noindent Theorem \ref{generale} can be reformulated in a way that does not require the knowledge of the nature of the zeros. See the result below. The estimate that one obtains is naturally not as sharp, but it is the exact analog of the corresponding result in the complex case:

\begin{theorem}\label{2m}
Let $P(q)$ be a slice regular polynomial of degree $m,$ with leading coefficient $a_m.$
Given any $q_0=u+vI_0 \in \mathbb{H}$ and any $R>0$, we can find a compact $3-$dimensional hypersurface
$$\Gamma= \{x+yI: (x-u)^2+(y-v)^2=r^2 \ \ \hbox{and}\  \ I\in \mathbb{S}\},$$ with $r<R$,
smooth if $R<v$, on which
$$|P(q)|\geq |a_m|\left(\frac{R}{2(2m+1)}\right)^m.$$
\end{theorem}

\noindent The next result is a consequence:

\begin{theorem}\label{esterno}
Let $P(q)$ be a slice regular polynomial of degree $m,$ with leading coefficient $a_m.$ For any $R>0$ there exist a natural number
$n\leq m$, $n$ quaternions $q_1=u_1+v_1I_1, \ldots , q_n=u_n+v_nI_n$, $n$ stricly positive radii $r_1 < R, \ldots, r_n<R$, and $n$ corresponding compact sets
$$D(q_\ell, r_\ell) = \{x+yI: (x-u_\ell)^2+(y-v_\ell)^2\leq r_\ell^2 \ \ \hbox{and}\  \ I\in \mathbb{S}\}$$
($\ell=1,\ldots,n$) bounded, respectively, by the $3-$dimensional hypersurfaces $\partial D(q_\ell, r_\ell)$, smooth if $r_\ell <v_\ell$,
such that
$$|P(q)|\geq |a_m|\left(\frac{R}{2(2m+1)}\right)^m$$
for $q$ outside
$$
D=\bigcup _{\ell=1}^n D(q_\ell, r_\ell).
$$
\end{theorem}
\begin{proof}
Given $R>0$, it is immediate to find $n\leq m$ quaternions $q_\ell=u_\ell+v_\ell I_\ell$, $\ell=1,\ldots, n$ such that $A=\bigcup _{\ell=1}^n D(q_\ell, R)$ contains all the roots of $P$. We then apply Theorem $\ref{2m}$ and deduce that, for each $q_\ell$ there exists $r_\ell$ such that
\begin{equation}\label{ipersuperficie}
|P(q)|\geq |a_m|\left(\frac{R}{2(2m+1)}\right)^m
\end{equation}
on the $3$-hypersurface $\partial D(q_\ell, r_\ell)$. Thus
inequality (\ref{ipersuperficie}) holds on
$$
B = \bigcup_{\ell=1}^n \partial D(q_\ell, r_\ell).
$$
The set $B$ contains the boundary $\partial D$ of $D=\bigcup_{\ell=1}^n D(q_\ell, r_\ell)$, that coincides with the boundary $\partial (\mathbb{H}\setminus D)$ of $\mathbb{H}\setminus D$. The slice regular polynomial $P$ has no zeros in $\mathbb{H}\setminus D$ and since  $\lim_{q\to +\infty} |P(q)| =+ \infty$, we get that inequality (\ref{ipersuperficie}) holds on each (open) connected component of $\mathbb{H}\setminus D$. In fact, if this were not the case, $|P|$ would have a local minimum at some point  $q \in \mathbb{H}\setminus D$ with $P(q)\neq0$, and by the Minimum Modulus Principle applied to $P$ on $\mathbb{H}$, $P$ would be constant. The statement follows.
\end{proof}

\section{Cartan theorem}
In this section we prove an analog of Cartan theorem, providing an estimate from below for the modulus of a quaternionic polynomial. Note that in the statement the roots may be repeated.
\begin{theorem}\label{cartan}
    Let $P(q)$ be a polynomial having isolated roots $\alpha_1, \ldots \alpha_t$ and spherical zeros $[\beta_1], \ldots, [\beta_p]$. Let $n=\deg P$, i.e. $n=t+2p$ and let $H$ be any positive real number. Then there are balls in $\mathbb H$ with the sum of their radii equal to $2H$ such that for each point $q$ lying outside of these balls the following inequality is satisfied:
\begin{equation}
|P(q)|> \left(\frac He\right)^n.
\end{equation}
\end{theorem}
\begin{proof}
We divide the proof in steps.
\\
Step 1. If there is a ball of radius $H$ containing all the zeros of the polynomial $P$, we can consider a ball $B$ with the same center and radius $2H$. Then, for any $q\in\mathbb H\setminus B$ we have that the distance from $q$ to any isolated zero of $P$ is at least $H$. By Lemma \ref{lagrange}, also the distance from $q$ to any spherical zero is at least $H$.
Consider a decomposition of $P$ into factors (see Theorem \ref{dueuno}) and write the factors associated to a spherical zero $[\beta]$ as $(q-\beta)\star(q-\overline\beta)$. From these considerations, it follows that
\[
\begin{split}
&\left|(q-\alpha_1)\star\ldots\star (q-\alpha_t)\left[(q-\beta_1)\star(q-\overline\beta_1)\ldots (q-\beta_p)\star(q-\overline\beta_p)\right]\right|\\
&=|(q-\alpha_1)\star\ldots \star(q-\alpha_t)| \, |(q-\beta_1)\star(q-\overline\beta_1)\ldots (q-\beta_p)\star(q-\overline\beta_p)|.\\
\end{split}
\]
Using Theorem \ref{pointwise}, we can rewrite the products $(q-\alpha_1)\star\ldots \star(q-\alpha_t)$ and
$(q-\beta_1)\star\ldots \star(q-\beta_p)$ as pointwise products f the form $(q-\alpha_1)\star\ldots \star(\tilde q-\alpha_t)$ and
$(q-\beta_1)\star\ldots \star(\hat q- \beta_p)$ where $\tilde q,\ldots, \hat q\in [q]$ and so their distance from the zeros of $P$ is at least $H$. So we have
 \[
\begin{split}
&|(q-\alpha_1)\star\ldots \star(q-\alpha_t)| \, |(q-\beta_1)\star(q-\overline\beta_1)\ldots (q-\beta_p)\star(q-\overline\beta_p)|\\
&=|(q-\alpha_1)\ldots (\tilde q-\alpha_t)| \, |(q-\beta_1)\ldots (\hat q-\beta_p)|\\
&> H^n > \left(\frac He \right)^n,
\end{split}
\]
\noindent
Step 2. If there is no ball as described in Step 1, let us consider the balls with radius $\lambda H/n$ containing exactly $\lambda$ zeros of $P$, where each spherical zero is described in a decomposition of $P$ into linear factors $(q-\beta)$, $(q- \overline\beta)$ where $\beta$, $\overline\beta$ is any pair of points belonging to the sphere (and thus a sphere counts as two zeros).
Let $\lambda_1$ be the largest integer such that a ball $B_1$ of radius $\lambda_1 H/n$ contains exactly $\lambda_1$ zeros of $P$. No ball of radius $\eta H/n$ greater than or equal to $\lambda_1 H/n$ can contain more than $\eta$ zeros. In fact, assume that there is a ball of radius $\eta H/n$, $\eta\geq\lambda_1$ containing $\eta'>\eta$ zeros of $P$. Then the concentric ball of radius $\eta' H/n$ contains either $\eta'$ zeros or $\eta''>\eta'$ zeros. The first case is not possible as $\lambda_1$ was the largest integer. In the second case, we consider a concentric ball of radius $\eta'' H/n$ and we repeat the procedure. Since the number of zeros (isolated or spherical) is finite, we will find a $\lambda$ such that the ball of radius $\lambda H/n >\lambda_1 H/n$ contains $\lambda$ zeros. This is absurd by our choice of $\lambda_1$.
\\
The zeros of $P$ contained in $B_1$ will be said of rank $\lambda_1$.
\\
Step 3. By removing the zeros in $B_1$, we have $n-\lambda_1$ zeros and repeating the above procedure, we construct a ball $B_2$ of radius $\lambda_2H/n$ containing $\lambda_2$ zeros. We have that $\lambda_2\leq\lambda_1$. In fact, if $\lambda_2>\lambda_1$, the ball $B_2$ would have radius larger than $\lambda_1H/n$ and it would contain $\lambda_2$ points. This contradicts Step 2.
 \\ We then remove the $\lambda_2$ points contained in $B_2$ and we iterate the procedure until we obtain a finite sequence of balls $B_1,\ldots, B_r$ of radii $\lambda_1H/n\geq\ldots\geq \lambda_rH/n$. Each ball $B_i$ contains $\lambda_i$ zeros of $P$ and $$\lambda_1\geq\lambda_2\geq\ldots\geq \lambda_r.$$ Note that, by construction,
$$
(\lambda_1+\cdots +\lambda_r)\frac Hn=H.
$$
\\
The zeros of $P$ contained in $B_i$ will be said of rank $\lambda_i$.
\\
Step 4. Consider the balls $\Sigma_1,\ldots, \Sigma_r$ with the same center as $B_1,\ldots, B_r$ and with the radius of $\Sigma_i$ which is twice the radius of $B_i$, $i=1,\ldots, r$. Let $q\in\mathbb H\setminus(\Sigma_1\cup\ldots\cup\Sigma_r)$. Consider the closed ball centered at $q$ and with radius $\lambda_0 H/n$ for some $\lambda_0\in\mathbb N$.
\\
By the construction above we have that the ball $B_0$ intersects the balls $B_i$ with radius at least equal to $\lambda_0 H/n$. So this ball can contain only zeros of rank less than $\lambda_0$. If we remove all the zeros of rank greater than or equal to $\lambda_0$, no ball of radius $\lambda H/n$ with $\lambda >\lambda_0$ can contain $\lambda$ of the remaining zeros. We conclude that $B_0$ contains at most $\lambda_0$ zeros of $P$.
\\
Let us label the zeros of $P$ in order of increasing distance from $q$ and let us write them as $\gamma_1,\ldots, \gamma_n$, where the $\gamma_\ell$ are either the isolated zeros $\alpha_i$ or one of the representatives of a spherical zero $[\beta_i]$, i.e. $\beta_i$ or $\overline\beta_i$. We have
$$
|q-\gamma_\ell|>\ell \frac{H}{n}
$$
so that
\[
\begin{split}
&|P(q)|=\left|(q-\alpha_1)\star\ldots\star (q-\alpha_t)\left[(q-\beta_1)\star(q-\overline\beta_1)\ldots (q-\beta_p)\star(q-\overline\beta_p)\right]\right|\\
&>\left(\frac Hn\right)^n n!> \left(\frac He\right)^n.
\end{split}
\]
\end{proof}
Cartan theorem will be used to prove a lower bound for the modulus of a function slice regular in a ball centered at the origin, see Theorem \ref{cartansr}.

\vskip 1 truecm
\noindent{\bf Comments to Chapter 3}.
The integral formulas in this chapter come from \cite{CGeSA} and \cite{moscow}, and the Riemann mapping theorem was originally proved in \cite{galsa3}. The results on the zeros come from \cite{zeri}, \cite{GSS}, \cite{mjm} while the Ehrenpreis-Malgrange lemma is taken from \cite{MR2869150}. Finally, Cartan theorem appears here for the first time.

%
%
%


\chapter{Slice regular infinite products}

\section{Infinite products of quaternions}

In the sequel we will deal with infinite products of quaternions. As in the classical complex case,
we will say that an infinite product of quaternions
$$
\prod_{n=1}^\infty u_n
$$
converges if the sequence $p_N=\prod_{n=1}^N u_n$ converges to a nonzero limit. It is crucial to note that the exclusion of $0$ as a limit is due to the fact that, by allowing it, one has that any sequence $\{u_n\}$ containing at least one vanishing term would converge, independently on the sequence.  For practical purposes, it is less reductive to assume that the
infinite product $\prod_{n=1}^\infty u_n$ converges if and only if the sequence $\{u_n\}$ contains a finite number of zero elements and the partial products $p_N$ obtained by multiplying the nonzero factors converges to a nonzero limit.  Since in a convergent infinite product the general term $u_n$ tends to 1 as $n\to\infty$, we will write $u_n=1+a_n$ where $a_n\to 0$ as $n\to\infty$. In complex analysis, $\prod_{n=1}^\infty (1+a_n)$ converges to a nonzero limit simultaneosly with the series $\sum_{n=1}^\infty \log(1+a_n)$. The convergence is not simultaneous if the infinite product tends to zero. \\
 The following proposition extends to the quaternionic setting the corresponding result in the complex case:
 \begin{proposition}\label{15.3} Let $a_1,\ldots ,a_N\in\mathbb H$ and let
 $$
 p_N=\prod_{n=1}^N (1+a_n), \qquad p_N^*=\prod_{n=1}^N (1+|a_n|).
 $$
 Then
 \begin{equation}\label{in1}
 p_N^*\leq \exp(|a_1|+\ldots +|a_N|)
 \end{equation}
 and
 \begin{equation}\label{in2}
 |p_N-1|\leq p_N^*-1.
 \end{equation}
\end{proposition}
\begin{proof}
The proof is based on the properties of the modulus of a quaternion, thus the proof in the complex case applies also here. We repeat it for the sake of completeness.
Since for any real number the inequality $1+x\leq e^x$ holds, we immediately have $p_N^*\leq \exp (|a_1|+\ldots +|a_N|)$.
To show \eqref{in2} we use induction. It is clear that \eqref{in2} holds for $N=1$, so we assume that it holds up to $k$. We have
\begin{equation}\label{eqin3}
p_{k+1}-1=p_k(1+a_{k+1})-1=(p_k-1)(1+a_{k+1}) +a_{k+1}.
\end{equation}
From \eqref{eqin3} we deduce
$$
|p_{k+1}-1|\leq |p_k-1|(1+|a_{k+1}|) +|a_{k+1}|=p_{k+1}^*-1.
$$
\end{proof}
\begin{proposition}
Suppose that the sequence $\{a_n\}$ is such that $0 \leq a_n < 1$. Then
$$
\prod_{n=1}^\infty (1-a_n) >0\ {\it if\ and\ only\ if\ }\sum_{n=1}^\infty a_n<\infty.
$$
\end{proposition}
\begin{proof}
 Let $p_N=\prod_{n=1}^N (1-a_n)$ then, by construction,
$p_1\geq p_2 \geq \ldots p_N>0$ and so the limit $p$ of the sequence $\{p_N\}$ exists. Assume that $\sum_{n=1}^\infty a_n<\infty$, then Theorem \ref{15.4} implies $p>0$. To show the converse, note that $$p\leq p_N\leq \exp(-a_1-\ldots -a_N)$$ and the right hand side expression tends to $0$ when $N\to\infty$ if $\sum_{n=1}^\infty a_n$ diverges to $+\infty$.
\end{proof}

In the complex case it is well known that the argument of a product  is equal to the sum of the arguments of the factors (up to an integer multiple of 2$\pi$). In the quaternionic case this equality does not hold in general, since the exponents may belong to different complex planes.
\begin{lemma}\label{lemma1}
Let $n\in \mathbb{N}$ and let $\theta_1,..., \theta_n$ $\in [0,\pi)$ be such that $\sum_{i=1}^{n}\theta_i <\pi$. Then for every $\{I_1,\ldots, I_n\}\in \mathbb{S}$ the inequality
$${\rm arg}_{\hh}(e^{\theta_1I_1}\cdots e^{\theta_nI_n})\leq\sum_{i=1}^{n}\theta_i$$
holds.
\end{lemma}
\begin{proof}
 For $n=1$ equality holds and thus the statement is true. Let us proceed by induction and choose $\theta_1, \theta_2,..., \theta_n$ $\in [0,\pi)$ with $\sum_{i=1}^{n}\theta_i<\pi$. Let $\phi \in [0,\pi)$
  and $J\in \mathbb{S}$ be such that  $e^{\theta_1I_1}\cdots e^{\theta_{n-1}I_{n-1}}=e^{\phi J}$ and take the product $e^{\phi J}e^{\theta_nI_n}$:
$$e^{\phi J}e^{\theta_nI_n}=\cos\phi \cos\theta_n+\cos\phi \sin\theta_nI_n+\sin\phi \cos\theta_nJ+ \sin\phi \sin\theta_nJI_n.$$
    Since $JI_n=-\langle J,I_n\rangle+J\times I_n$ we obtain that
    $$
    \cos({\rm arg}_{\hh}(e^{\phi J}e^{\theta_nI_n}))={\rm Re}\left( e^{\phi J}e^{\theta_nI_n}\right) = \cos\phi \cos\theta_n-\sin\phi \sin\theta_n\langle J,I_n\rangle.
    $$

However $\langle J,I_n\rangle \leq |J ||I_n |=1$ and $\sin\phi \sin\theta_n\geq 0$ and so  $$\cos({\rm arg}_{\hh}(e^{\phi J}e^{\theta_nI_n}))\geq \cos\phi \cos\theta_n-\sin\phi \sin\theta_n= \cos(\phi + \theta_n).$$
But the function $\cos(x)$ is decreasing in $[0,\pi]$ and so we deduce
$${\rm arg}_{\hh}(e^{\phi J}e^{\theta_nI_n})\leq\phi + \theta_n.$$ Thanks to the induction hypothesis we have the thesis.
\end{proof}
Let $\{a_i\}_{i\in\mathbb{N}}\subseteq \mathbb{H}$ be a sequence
such that the associated series is absolutely convergent. Then the series itself is convergent, namely $\sum_{i=0}^{\infty}|a_i|$ convergent implies that  $\sum_{i=0}^{\infty}a_i$ convergent. This observation is used to prove the following result.

\begin{theorem}\label{log2}Let $\{a_i\}_{i\in\mathbb{N}}\subseteq \mathbb{H}$ be a sequence. If the series $\sum_{i=0}^{\infty}|{\rm Log}(1+a_i)|$ converges, then the product  $\prod^{\infty}_{i=0}(1+a_i)$ converges.
\end{theorem}
\begin{proof}
 If the series $\sum_{i=0}^{\infty}|{\rm Log}(1+a_i)|$ is convergent, then the sequence $\left\lbrace a_i\right\rbrace_{i\in\mathbb{N}}$ tends to zero. Thus we can suppose that $1+a_i\notin \ (-\infty,0]$.
 Then, for every $i\in  \mathbb{N}$, we consider suitable $\theta_i \in [0,\pi)$ and  $I_i\in \mathbb{S}$ such that $$
 1+a_i=|1+a_i|e^{\theta_i I_i}.
 $$
 Consequently,
 $${\rm Log}(1+a_i)=\ln|1+a_i|+\theta_i I_i$$ and so
\begin{equation}\label{norma}\left|\ln|1+a_i|\right|\leq \left|{\rm Log}(1+a_i)\right| \end{equation}
moreover
\begin{equation}\label{norma1}\left|\theta_i\right|=\left|\theta_i I_i\right|\leq \left|{\rm Log}(1+a_i)\right| \end{equation}
for every  $i\in \mathbb{N}$.\\
Our next task is to show that  the sequence of the partial products
$$Q_n=\prod^{n}_{i=0}(1+a_i)$$ has a finite, nonzero, limit.
First of all, note that for every $n\in\mathbb{N}$
$$\prod^{n}_{i=0}(1+a_i)=\prod^{n}_{i=0}|1+a_i|\ \prod^{n}_{i=0} e^{\theta_i I_i}.$$
 Our hypothesis and inequality (\ref{norma}) show that the series  $\sum_{i=0}^{\infty}\ln|1+a_i|$ is convergent and hence  the infinite product $\prod^{\infty}_{i=0}|1+a_i|$ is convergent. Then it is sufficient to prove that the sequence  $R_n=\prod^{n}_{i=0}e^{\theta_i I_i}$ $\subseteq \partial B(0,1)=\left\{ q \in\mathbb{H}: |q|=1 \right\}$ is convergent. To this end, we will show that $\{ R_n \}_{n\in \mathbb{N}}$ is a Cauchy sequence. For  $n>m\in \mathbb{N}$, we have:
\[
\begin{split}|R_n-R_m|&=\left|\prod^{n}_{i=0}e^{\theta_i I_i} - \prod^{m}_{i=0}e^{\theta_i I_i}\right|\\
&=\left|\prod^{m}_{i=0}e^{\theta_i I_i}\prod^{n}_{i=m+1}e^{\theta_i I_i}-\prod^{m}_{i=0}e^{\theta_i I_i}\right|\\
&=\left|\prod^{m}_{i=0}e^{\theta_i I_i}\right|\left|\prod^{n}_{i=m+1}e^{\theta_i I_i} - 1\right|\\
&=\left|\prod^{n}_{i=m+1}e^{\theta_i I_i}-1\right|\leq {\rm arg}_{\hh}\left(\prod^{n}_{i=m+1}e^{\theta_i I_i}\right),
\end{split}
\]
where the last inequality holds since $\left|\prod^{n}_{i=m+1}e^{\theta_i I_i}\right|=1$.
Inequality (\ref{norma1}) implies that the series $\sum_{i=0}^{\infty}\theta_i$ is convergent. Therefore the sequence $S_n=\sum^{n}_{i=0} \theta_i $ of the partial sums is a Cauchy sequence. As a consequence, for all $\epsilon >0$, there exists $m_0\in\mathbb{N}$ such that  for all $n>m>m_0$
$$\sum_{i=m+1}^{n}\theta_i<\epsilon.$$
In particular for $\epsilon <\pi$, by  Lemma \ref{lemma1}, we deduce  $${\rm arg}_{\hh}\left(\prod^{n}_{i=m+1}e^{\theta_i I_i}\right) \leq \sum^{n}_{i=m+1} \theta_i <\epsilon,$$ and this finishes the proof.
\end{proof}

The following proposition, which will be used in the sequel, is an expected extension of the analog property in the complex case:

\begin{proposition} \label{proplog}Let ${\rm Log}$ be the principal quaternionic logarithm. Then
\begin{equation}\label{lim}\lim_{q\rightarrow 0}q^{-1}{\rm Log}(1+q)=1.\end{equation}
\end{proposition}
\begin{proof}
 Consider a sequence $\{ q_n\}_{n \in \mathbb{N}}=\{x_n +I_ny_n\}_{n \in \mathbb{N}}$ such that  $q_n \rightarrow 0 $ for $n\rightarrow\infty$. It is not reductive to assume that $y_n>0$. Let us compute
\begin{equation} |(x_n+I_n y_n)^{-1}{\rm Log} (1+x_n+I_n y_n)-1|.\end{equation}
It can be easily checked that both the real part and the norm of the imaginary part of
$(x_n+I_n y_n)^{-1}{\rm Log}(1+x_n+I_n y_n)$
 do not depend on  $I_n$ and so we can set $I_n=I_0$ for every  $n\in \mathbb{N}$ and so we compute
\begin{equation}\label{101} |(x_n+I_0 y_n)^{-1}{\rm Log}(1+x_n+I_0 y_n)-1|.\end{equation}
  It is immediate that $(x_n+I_0 y_n)^{-1}{\rm Log}(1+x_n+I_0 y_n) \in \mathbb C_{I_0}$ for all $n\in\mathbb{N}$ and in the complex plane $\mathbb C_{I_0}$ we can use the classical arguments. Thus we conclude that the sequence (\ref{101}) tends to zero and the statement follows.
\end{proof}

\begin{corollary}\label{ass}
The series $\sum^{\infty}_{n=0}|{\rm Log}(1+a_n)|$ converges if and only if  the series $\sum^{\infty}_{n=0}|a_n|$ converges.
\end{corollary}
\begin{proof}
 Assume that the series $\sum^{\infty}_{n=0}|{\rm Log}(1+a_n)|$ or  the series $\sum^{\infty}_{n=0}|a_n|$ are convergent. Then we have  $\lim_{n\rightarrow\infty}a_n=0$.
Proposition  \ref{proplog} implies that for any given $\epsilon > 0$
$$|a_n|(1-\epsilon)\le|{\rm Log}(1+a_n)|\le |a_n|(1+\epsilon)$$ for all sufficiently large $n$. Thus the two series in the statement are simultaneously convergent.
\end{proof}

The following result contains a sufficient condition for the convergence of quaternionic infinite products.
\begin{theorem}\label{teo33}
Let  $\{a_n\}_{n\in\mathbb{N}}\subseteq \mathbb{H}$. If $\ \sum^{\infty}_{n=1}|a_n|$ is convergent then
the product $\prod^{\infty}_{n=1}(1+a_n)$ is convergent.
\end{theorem}
\begin{proof}
It is a consequence of Theorem \ref{log2} and Corollary \ref{ass}.
\end{proof}
\section{Infinite products of functions}
In this section we consider infinite products of functions of a quaternionic variable defined in an open set $U\subseteq\mathbb H$. Given a sequence of functions $\{a_n(q)\}$ which do not vanish on $U$, we will say that $\prod_{n=0}^\infty a_n(q)$ converges uniformly on the compact subsets of $U$ if the sequence of partial products $\{\prod_{n=0}^N a_n(q)\}$ converges uniformly on the compact subsets of $U$ to a non vanishing function. Since there may be factors with zeros, we require that for any fixed compact subset $K$ at most a finite number of factors vanish in some points of $K$. We thus give the following definition:
\begin{definition} \label{infprod}
Let $\{a_n(q)\}$ be a sequence of functions defined on an open set $U\subseteq\mathbb H$. We say that $\prod_{n=0}^\infty a_n(q)$ converges compactly in $U$ to a function $f:\ U\to\mathbb H$ if for any compact set $K\subset U$ the following conditions are fulfilled:
\begin{itemize}
\item[-] there exists $N_K\in \mathbb N$ such that $a_n\not=0$ for all $n\geq N_K$;
\item[-] the residual product $\prod_{n=N_K}^\infty a_n(q)$ converges uniformly on $K$ to a never vanishing function $f_{N_K}$;
\item[-] for all $q\in K$
$$
f(q)= \left( \prod_{n=0}^{N_k-1}a_n(q)\right) f_{N_K}(q).
$$
\end{itemize}
\end{definition}
\begin{theorem}\label{15.4}
Suppose $\{a_n\}$ is a sequence of bounded quaternionic-valued functions defined on a
set $U\subseteq\mathbb H$, such that $\sum_{n=1}^\infty | a_n(q)|$ converges uniformly on $U$. Then the product
\begin{equation}
f(q) = \prod_{n=1}^\infty (1+a_n(q))
\end{equation}
converges compactly in $U$, and $f(q_0) = 0$ at some point $q_0\in U$ if and only if $a_n(q_0) =
- 1$ for some $n$.
\end{theorem}
\begin{proof} First of all we observe that the hypothesis on $\sum_{n=1}^\infty | a_n(q)|$ ensures that every $q\in U$ has a neighborhood in which at most finitely many of the factors $(1+a_n(q))$ vanish. Then, since $\sum_{n=1}^\infty | a_n(q)|$ converges uniformly on $U$, we have that it is also bounded on $U$.
 Let us set $$p_N(q)=\prod_{n=1}^N (1+a_n(q));$$ then Proposition \ref{15.3} yields  the existence of a (finite) constant $С$ such that $| p_N(q) | \leq C$ for all $N\in\mathbb N$ and all $q\in U$.
For every $\varepsilon$ such that $0 < \varepsilon < \frac 12$ there exists a $N_0$ such that
\begin{equation}\label{equa3}
\sum_{n=N_0}^\infty | a_n(q)|<\varepsilon, \qquad q\in U.
\end{equation}
Let $M,N\in\mathbb N$ and assume that $M>N$.
  Proposition \ref{15.3} and \eqref{equa3} show that
 \begin{equation}\label{ineqrM}
 |p_M(q)-p_N(q)|\leq |p_N(q)| (e^\varepsilon -1) \leq 2 |p_N(q)| \varepsilon \leq 2 C \varepsilon,
 \end{equation}
 thus the sequence $p_N$ converges
 uniformly to a limit function $f$. Moreover, \eqref{ineqrM} and the triangular inequality give,
 for $M>N$ $$|p_M(q)|\geq (1-2\varepsilon) |p_{N}(q)|$$ on $U$ and so
$$|f(q)|\geq (1-2\varepsilon) |p_{N}(q)|.$$
Thus $f(q)=0$ if and only if $p_{N}(q)=0$.
\end{proof}
We now give the analog of Definition \ref{infprod} in the case of the slice regular product:
\begin{definition} \label{infprodstar}
 Let $\{f_n(q)\}$ be a sequence of functions slice regular on an axially symmetric slice domain $U\subseteq\mathbb H$. We say that the infinite $\star$-product $\prod_{n=0}^{\star\infty} (1+f_n(q))$ converges compactly in $U$ to a function $f:\ U\to\mathbb H$ if for any axially symmetric compact set $K\subset U$ the following conditions are fulfilled:
\begin{itemize}
\item[-] there exists $N_K\in \mathbb N$ such that $f_n\not=0$ for all $n\geq N_K$;
\item[-] the residual product $\prod_{n=N_K}^{\star\infty} (1+f_n(q))$ converges uniformly on $K$ to a never vanishing function $f_{N_K}$;
\item[-] for all $q\in K$
$$
f(q)= \left( \prod_{n=0}^{\star N_k-1}(1+f_n(q))\right) f_{N_K}(q).
$$
\end{itemize}
\end{definition}

In order to relate an infinite product with an infinite slice regular product, we need a technical result:
 \begin{lemma}\label{prop358}
Let $\left\lbrace f_n\right\rbrace_{n\in\mathbb{N}}$ be a sequence of slice regular functions defined on an axially symmetric slice domain $U$.
Let $K\subseteq U$ be an axially symmetric compact set. Assume that there exists an integer $N_K$ such that $1+f_n\neq 0$  on $K$ for all $n\geq N_K$. Let
$$
F^{m}_{N_K}(q)=\overset{\star m}{\underset{n=N_K}\prod}(1+f_n(q)).
$$
Then for all $m\geq N_K$ and  for any $q\in K$
$$
F^{m}_{N_K}(q)= \overset{m}{\underset{i=N_K}\prod}(1+f_i(T_i(q)))\neq 0
$$
where $$T_j(q)=(F_{N_K}^{(j-1)}(q))^{-1}qF_{N_K}^{(j-1)}(q) \ \ {\it for}\ \ j>N_K$$ and $T_j(q)=q$ for $j=N_K$.
\end{lemma}

\begin{proof}
Let $q\in K$.
 The assertion is true for $m=N_K$ by hypothesis since $F^{N_K}_{N_K}(q)=1+f_{N_K}(q)\neq 0$. We now proceed by induction. Assume that the assertion is true for $m=N_K,\cdots, n-1$.
Then
 \begin{equation}\label{f1}F^{n}_{N_K}(q)=F^{(n-1)}_{N_K}(q)\star (1+f_n(q)).\end{equation}
Since $T_j(q)$ is a rotation of $q$, it is immediate that
$$
{\rm Re}\, (T_j(q))={\rm Re}\, (q)\quad {\rm and}\quad  |{\rm Im}\, (T_j(q)|=|{\rm Im}\, (q)|\quad {\rm for\ all} \ j\geq N_K .
$$
Since $K$ is an axially symmetric set, obviously $T_j(q)\in K$ if and only if $q\in K$.
By Theorem \ref{pointwise}, and by the induction hypothesis, we have that formula (\ref{f1}) rewrites as:
$$F^{n}_{N_K}(q)=F^{(n-1)}_{N_K}(q)(1+f_n(T_n(q))).$$
The factor $(1+f_{n})$ does not vanish on  $K$ and hence, since $T_n(q)\in K$, the function  $F^{n}_{N_K}(q)$ does not vanish on $K$.  Using again the induction hypothesis, we have $$F^{n-1}_{N_K}(q)= \overset{n-1}{\underset{i=N_K}\prod}(1+f_i(T_i(q)))$$ and so
\[
\begin{split} F^{n}_{N_K}(q)&=\overset{n-1}{\underset{i=N_K}\prod}(1+f_i(T_i(q)))\ (1+f_n(T_n(q)))\\
&=\overset{n}{\underset{i=N_K}\prod}(1+f_i(T_i(q)))\not=0,
\end{split}
\]
and this concludes the proof.
\end{proof}

\begin{theorem}\label{the6} Let  $\left\lbrace f_n \right\rbrace_{n\in\mathbb N} $   be a sequence of slice regular functions defined on an axially symmetric slice domain $U$. The infinite $\star$-product
 $$\overset{\star\infty}{\underset{n=0}\prod}(1+f_n(q))$$ converges compactly in $U$ if and only if the infinite product $$\overset{\infty}{\underset{n=0}\prod}(1+f_n(q))$$  converges compactly in $U$.
\end{theorem}

\begin{proof} It is not reductive to assume that the compact subsets of $U$ are axially symmetric. Let $K$ be a compact, axially symmetric subset of $U$.
Let $N_K\in\mathbb N$ be such that for every $n\geq N_K$ the factors  $1+f_n(q)$ do not  vanish on $K$. By Lemma \ref{prop358} the infinite $\star$-product $$\overset{\star\infty}{\underset{i=N_K}\prod}(1+f_i(q))$$ converges if and only if  $$\overset{\infty}{\underset{i=N_K}\prod}(1+f_i(T_n(q)))$$
converges.
Since ${\rm Re}(T_j(q))={\rm Re}(q)$ and $|{\rm Im}(T_j(q)|=|{\rm Im}(q)|$ for all $j\geq N_K$, we have that
$T_j(q)\in K$ if and only if $q\in K$. This ends the proof.
\end{proof}

The following result will be useful to establish when an infinite product of slice regular functions is slice regular.
\begin{proposition}\label{sereg}
Let $\left\lbrace f_n \right\rbrace _{n\in\mathbb{N}}$ be a sequence of slice regular functions defined on an axially symmetric slice
domain $U\subseteq\mathbb H$ and converging uniformly to a function $f$ on the compact sets of $U$. Then $f$ is slice regular
on $U$.
\end{proposition}
\begin{proof}
We write the restriction of $f_n$ to $\mathbb C_I$ as $$f_{n|\mathbb C_I}(x+Iy)=F_n(x+Iy)+G_n(x+Iy)J$$ using the Splitting Lemma. The functions $F_n$, $G_n$ are holomorphic for every $n\in\mathbb N$.
The restriction $f_I(x+Iy)=F(x+Iy)+G(x+Iy)J$ to $\mathbb C_I$ of the limit function $f$  is  such that $F$ and $G$ are the limit of $F_n$ and $G_n$ respectively. Since $f_n\to f$ uniformly also $F_n\to F$ and $G_n\to G$ uniformly and so both $F$ and $G$ are holomorphic. It follows that $f_I$ is the kernel of $\partial_x+I\partial_y$ and by the arbitrariness of $I\in\mathbb S$ the statement follows.
\end{proof}

\begin{proposition}\label{pro4.6}
Let $\left\lbrace f_n \right\rbrace _{n\in\mathbb{N}}$ be a sequence of slice regular functions  defined on a symmetric   slice domain  $U$. If the infinite slice regular product
$$\overset{\star\infty}{\underset{n=0}\prod}(1+f_n(q))$$ converges compactly in $U$ to a  function $f$, then $f$ is slice regular on $U.$
\end{proposition}
\begin{proof} By the assumption made at the beginning of this section, for any compact set $K\subseteq \mathbb{H}$ there exists an integer $N_K$  such that $1+f_n\neq 0$ on $K$ if $n\geq N_K$.
The $\star -$product \begin{equation}\overset{\star\infty}{\underset{n=N_K}\prod}(1+f_n(q))\end{equation}
forms a sequence of slice regular functions which converges uniformly on $K$ to a slice regular function $F_{N_K}$ that does not vanish on $K$. Therefore the function $f$  can be written as a  finite product of  slice regular functions  $$f(q)= \left[\overset{\star N_K-1}{\underset{n=0}\prod}(1+f_n(q))\right]\star F_{N_K}(q) $$  and hence it is slice regular on $K$.  Moreover, by Lemma \ref{prop358}, the zero set of $f$ on $K$ coincides with the zero set of the  finite product  $\overset{\star(N_K-1)}{\underset{n=0}\prod}(1+f_n(q))$.
\end{proof}
\begin{proposition}\label{prop2:23}
 Let $\{f_n\}$ be a sequence of functions slice regular on an axially symmetric slice domain $U$ and assume that no $f_n$ is identically zero on $U$. Suppose that
$$
\sum_{n=1}^{\infty} |1-f_n(q)|
$$
converges uniformly on the compact subsets of $U$. Then $\prod_{n=1}^{\star\infty} f_n(q)$ converges compactly in  $U$ to a function $f\in\mathcal R(U)$.
\end{proposition}
\begin{proof}
It is an immediate consequence of Theorem \ref{15.4}, Theorem \ref{the6} and Proposition \ref{pro4.6}.
\end{proof}

\section{Weierstrass theorem}
Weierstrass theorem was original proved in \cite{MR2836832}. Here we provide an alternative statement and we will show how to retrieve the result in \cite{MR2836832}.\\
   Since we need to consider the slice regular composition of the exponential function with a polynomial as introduced in Remark \ref{funzioni*}, for the sake of clarity we repeat the definition below.
\begin{definition}
Let $p(q)=b_0+qb_1+\ldots +q^mb_m$, $b_i\in\mathbb H$.
We define $e_\star^{p(q)}$ as:
\begin{equation}\label{expstar}
e_\star^{p(q)}:=\sum_{n=0}^\infty \frac{1}{n!}(p(q))^{\star n}=\sum_{n=0}^\infty \frac{1}{n!}(b_0+qb_1+\ldots +q^mb_m)^{\star n}.
\end{equation}
\end{definition}
\begin{remark}\label{estarI}{\rm
When the polynomial $p$ has real coefficients, we have $e_\star^{p(q)}=e^{p(q)}$. When $b_i\in\mathbb C_I$ for some $I\in\mathbb C_I$ we have that, denoting by $z$ the complex variable on $\mathbb C_I$, $(e_\star^{p(q)})_{|\mathbb C_I}=e^{p(z)}$, i.e. the restriction of the function $e_\star^{p(q)}$ to $\mathbb C_I$ coincides with the complex valued function $e^{p(z)}$. Therefore $e_\star^{p(q)}$ is the slice regular extension to $\mathbb H$ of $e^{p(z)}$. Note that sometimes we will write $\exp(p(z))$ or $\exp_\star(p(q))$ instead of $e^{p(z)}$, $e_\star^{p(q)}$.}
\end{remark}
\begin{theorem}\label{teo8}
Let $\{a_n\}_{n\in\mathbb{N}}\subseteq \mathbb{H}\setminus\{0\}$ be a diverging sequence. Let $\{p_n\}$ be a sequence of nonnegative integers such that
\begin{equation}\label{condit1}
\sum_{n=1}^\infty \left(\frac{r}{|a_n|}\right)^{p_n+1}<\infty
\end{equation}
for every positive $r$.
Set
$$
e_{p_n}(q, a_n^{-1})=e_\star^{qa_n^{-1}+\frac 12 q^2a_n^{-2}\ldots +\frac{1}{p_n}q^{p_n}a_n^{-p_n}}
$$
for all  $q\in \mathbb H$.
Then the infinite $\star$-product
\begin{equation}\label{p}\overset{\star\infty}{\underset{n=0}\prod}(1-qa_n^{-1})\star e_{p_n}(q, a_n^{-1})\end{equation}
converges compactly in $\mathbb{H}$ to an entire slice regular function.
Furthermore, for every $n\in\mathbb{N}$, the function $e_{p_n}(q, a_n^{-1})$ has no zeros in $\mathbb{H}$.
\end{theorem}

\begin{proof}For all $n\in\mathbb{N}$ let us set
$$
\mathcal{P}_n(q)=(1-qa_n^{-1})\star e_{p_n}(q, a_n^{-1}).
$$
    Theorem \ref{the6} yields that the $\star$-product (\ref{p}) converges compactly in $\mathbb H$ if and only if the pointwise product
\begin{equation}\label{p2}\overset{\infty}{\underset{n=0}\prod}\mathcal{P}_n(q)
\end{equation}
converges  compactly in $\mathbb H$.
Theorem \ref{teo33} implies that the product (\ref{p2}) converges compactly in $\mathbb{H}$ if the series
$$\sum_{n=0}^{\infty}\left|1-\mathcal{P}_n(q)\right|$$ converges compactly in $\mathbb{H}$.
Let $K\subseteq \mathbb{H}$ be a
 compact set. Let $R>0$ be such that $K\subseteq B(0,R) $ and let $N\in\mathbb{N}$
be the integer such that  $|a_n|>R$ for all $n\geq N$. For any $n\in\mathbb N$ let us denote by $I_n$ the element in $\mathbb S$ such that
$a_n\in\mathbb C_{I_n}$ and let $\mathnormal{p}_n(z)$ be the restriction of
$\mathcal{P}_n$ to the complex plane $\mathbb C_{I_n}$. Hence, denoting by $z$ the variable in $\mathbb C_{I_n}$, we have
$$\wp_n(z)=(1-za_n^{-1})e^{za_n^{-1}+\frac 12 z^2a_n^{-2}\ldots +\frac{1}{n}z^{p_n}a_n^{-p_n}}$$
and we can estimate the coefficients $c_k$ of the Taylor expansion
$$
\wp_{n}(z)=1-\sum_{k=0}^\infty c_k (za_n)^{k+p_n+1}
$$
of $\wp_{n}$ at the point $0$ as in the complex case, see \cite{rudin}, thus obtaining
\begin{equation}\label{ck}
0\leq c_k \leq \frac{1}{p_n+1}.
\end{equation}
Since the slice regular extension $\mathcal{P}_n$ of $\wp_n$ is unique by the Identity Principle, the coefficients of the power series expansion of  $\mathcal{P}_n$ are the same of $\wp_n$ and so we have:
$$\mathcal{P}_n(q)=1-\sum_{k=0}^{\infty}c_k q^{k+p_n+1}a_n^{-(k+p_n+1)}$$
for every $q\in\mathbb{H}$.
Then
\[
\begin{split}
\left|1-\mathcal{P}_n(q)\right|&=\left|\sum_{k=0}^{\infty}c_k q^{k+n+1}a_n^{-(k+n+1)}\right|\\
&\leq\sum_{k=0}^{\infty}c_k \left(\frac{|q|}{|a_n|}\right)^{(k+p_n+1)}
\end{split}
\]
and using the estimate \eqref{ck} we obtain
\[
\begin{split}\left|1-\mathcal{P}_n(q)\right|&\leq\sum_{k=0}^{\infty}\frac{1}{p_n+1} \left(\frac{|q|}{|a_n|}\right)^{k+p_n+1}\\
&\leq \frac{1}{p_n+1}\left(\frac{|q|}{|a_n|}\right)^{p_n+1}\left(\sum_{k=0}^{\infty} \left(\frac{|q|}{|a_n|}\right)^{k}\right).
\end{split}
\]
The series
$$\sum_{k=0}^{\infty} \left(\frac{|q|}{|a_n|}\right)^{k}$$ is a geometric series whose ratio is strictly smaller than $1$
when $n\geq N$,
since $q\in K \subseteq B(0, R)$ and $|a_{n}|> R$ for $n\geq N$. Hence the power series is convergent to the value
$$\left(1-\frac{|q|}{|a_n|}\right)^{-1}$$
and we obtain
$$\left|1-\mathcal{P}_n(q)\right|\leq\frac{1}{p_n+1}\left(\frac{|q|}{|a_n|}\right)^{p_n+1}\left(1-\frac{|q|}{|a_n|}\right)^{-1}.$$
Since $\lim_{n\rightarrow \infty}\left(1-\frac{|q|}{|a_n|}\right)^{-1}=1$ and the series
 \begin{equation}\label{3107}\sum_{n=0}^{\infty}\frac{1}{p_n+1}\left(\frac{|q|}{|a_n|}\right)^{p_n+1}\end{equation} converges on $K$, also the series $$\sum_{n=0}^{\infty}\left|1-\mathcal{P}_n(q)\right|$$ converges
     on $K$. The function defined by the $\star$-product in \eqref{p} is entire by Proposition \ref{prop2:23}.
The function $e_{p_n}(q,a_n^{-1})$ has no zeros since it is the
(unique) slice regular extension of a  function holomorphic on the complex plane $\mathbb C_{I_n}$ and without zeros on that plane.
\end{proof}
\begin{remark}\label{remark22}{\rm Condition \eqref{condit1} in the previous result is satisfied when $p_n=n-1$ for all $n\in\mathbb N$. A fortiori, we can choose $p_n=n$.}
\end{remark}
\begin{remark}\label{rmknotation}{\rm
Let $f$ be an entire slice regular function and let $\alpha_1, \alpha_2, \ldots$ be its non spherical zeros
and $[\beta_1], [\beta_2], \ldots$, be its spherical zeros. As we discussed in Chapter 2, a spherical zero $[\beta]$ is characterized by the fact that $f$ contains the factor $q^2-2{\rm Re}(\beta)q+|\beta|^2$ to a suitable power. Since this factor splits as $(q-\beta)\star(q-\overline{\beta})$, the spherical zeros can also be listed as pairs of conjugate elements $\beta$, $\overline\beta$, where $\beta$ is any element belonging to the given sphere.
It is important to remark that, when forming the list of the zeros of a given function, two elements $\beta, \overline\beta$ defining a sphere should appear one after the other. We will say that $\beta$ (or any other element in $[\beta]$) is a generator of a spherical zero.\index{spherical zero!generator of a} If the sphere has multiplicity $m$ the pair $\beta, \overline\beta$ will be repeated $m$ times in the list. Also the isolated zeros appear with their multiplicities.
With this notation we can list the zeros as $\gamma_1,\gamma_2,\ldots$ (where $\gamma_i$ stands  for either one of the elements $\alpha_j$ or $\beta_j$ or $\overline \beta_j$) according to increasing values of their moduli. When some elements have the same modulus, we can list them in any order (but keeping together the pairs defining a sphere).
}
\end{remark}

\begin{remark}\label{remarksfere}{\rm
  When the list of zeros contains a spherical zero, namely a pair $\beta_n$, $\overline{\beta_n}$, we can choose $\gamma_n=\beta_n$ and $\gamma_{n+1}=\overline{\beta_n}$ and, for $z$ belonging to the complex plane of $\beta_n$, $\overline{\beta_n}$, we have
\[
\begin{split}
&(1-q\beta_n^{-1})\star e_\star^{q\beta_n^{-1}+\ldots +\frac{1}{p_n}q^{p_n}\beta_n^{-p_n}}\star (1-q\bar{\beta}_n^{-1})\star e_\star^{q\bar{\beta}_n^{-1}+\ldots +\frac{1}{p_n}q^{p_{n}}\bar{\beta}_n^{-p_{n}}}\\
&={\rm ext}\left(
(1-z\beta_n^{-1})\exp\left({z\beta_n^{-1}+\ldots +\frac{1}{p_n}z^{p_n}\beta_n^{-p_n}}\right)\right. \\
&\times(1-z\bar{\beta}_n^{-1}) \exp\left(z\bar{\beta}_n^{-1}+\ldots \left. +\frac{1}{p_n}z^{p_n}\bar{\beta}_n^{-p_n}\right)\right)\\
&={\rm ext}\left(
(1-z\beta_n^{-1}) (1-z\bar{\beta}_n^{-1}) \exp\left({z\beta_n^{-1}+\ldots + \frac{1}{p_n}z^{p_n}\beta_n^{-p_n}+ z\bar{\beta}_n^{-1}+\ldots }\right)\right)\\
&={\rm ext}\left((1- z (\beta_n^{-1}+\bar{\beta}_n^{-1}) + z^2 |{\beta}_n^{-1}|^2)\exp\left({2z\frac{{\rm Re}(\beta_n)}{|\beta_n|^2}+\ldots +2\frac{1}{p_n}z^{p_n}\frac{{\rm Re}(\beta_n^{p_n})}{|\beta_n|^{2p_n}}}\right)
\right)\\
&=\left(1- 2q \frac{{\rm Re}(\beta_n)}{|\beta_n|^2} + q^2 \frac{1}{|{\beta}_n|^2}\right)\exp({2q\frac{{\rm Re}(\beta_n)}{|\beta_n|^2}+\ldots +2\frac{1}{p_n}q^{p_n}\frac{{\rm Re}(\beta_n^{p_n})}{|\beta_n|^{2p_n}}}).
\end{split}
\]
}
\end{remark}
To prove the Weierstrass factorization theorem we need additional information in the case in which a function has isolated zeros only:
\begin{lemma}\label{isolatedW}
Let $f$ be an entire slice regular function whose sequence of zeros $\{\gamma_n\}$ consists of isolated, nonreal elements, repeated according to their multiplicity. Then there exist $\delta_n\in[\gamma_n]$ such that
$$
f(q)=h(q)\star \overset{\star\infty}{\underset{n=1}\prod} e_{p_n}(q, \bar{\delta}_n^{-1}) \star (1-q\bar{\delta}_n^{-1} ),
$$
where $h$ is a nowhere vanishing entire slice regular function.
\end{lemma}
\begin{proof}
By our hypothesis if $f$ vanishes at $\gamma_n$, $f$ cannot have any other zero on $[\gamma_n]$ and, in particular, $f(\overline{\gamma_n})\not=0$. Starting from $f$ we will construct a function  having zeros $\overline{\gamma_n}$ so that this new function will have spherical zeros. Let us start by considering the function $f_1$ defined by
$$
f_1(q)= f(q) \star (1- q \delta_1^{-1})\star e_{p_1}(q,\delta_1^{-1})
$$
where $\delta_1^{-1}=f(\overline{\gamma}_1)^{-1} \overline{\gamma}_1^{-1} f(\overline{\gamma}_1)$ and so $\delta_1\in[\gamma_1]$. Formula \ref{pointwise}
implies that $f_1(\overline{\gamma}_1)=0$ and so $f_1(q)$ has a spherical zero at $[\gamma_1]$ and
$$
f_1(q)=\left(1-2\frac{{\rm Re}(\gamma_1)}{|\gamma_1|^2} + q^2\frac{1}{|\gamma_1|^2}\right)\tilde f_1(q)
$$
where $\tilde f_1(q)$ has zeros belonging to the sequence $\{\gamma_n\}_{n\geq 2}$. Thus $\tilde f_1$ is not vanishing at $\overline{\gamma}_2$ and we repeat the procedure to add to $\tilde f_1$ the zero $\overline{\gamma}_2$ by constructing the function
\[
\begin{split}
f_2(q) &= \left(1-2\frac{{\rm Re}(\gamma_1)}{|\gamma_1|^2} + q^2\frac{1}{|\gamma_1|^2}\right)\tilde f_1(q)
\star (1- q \delta_2^{-1})\star e_{p_2}(q,\delta_2^{-1})\\
&= f_1(q)
\star (1- q \delta_2^{-1})\star e_{p_2}(q,\delta_2^{-1})
\end{split}
\]
where $\delta_2^{-1}=\tilde f(\overline{\gamma}_2)^{-1} \overline{\gamma}_2^{-1} \tilde f(\overline{\gamma}_2)$ and $\delta_2\in[\gamma_2]$. The function $f_2$ has spherical zeros at $[\gamma_1]$, $[\gamma_2]$ and isolated zeros belonging to the sequence $\{\gamma_n\}_{n\geq 3}$.
Note that
$$
f_2(q)= f(q) \star (1- q \delta_1^{-1})\star e_{p_1}(q,\delta_1^{-1}) \star (1- q \delta_2^{-1})\star e_{p_2}(q,\delta_2^{-1}).
$$
Iterating the reasoning, we obtain a function $k(q)$ where
\begin{equation}\label{kappa}
k(q)=f(q)\star \overset{\star\infty}{\underset{n=1}\prod} (1-q\delta_n^{-1})\star e_{p_n}(q,\delta_n^{-1}).
\end{equation}
By construction, $k(q)$ is an entire slice regular function with zeros at the spheres $[\gamma_n]$ and no other zeros. Since the factors corresponding to the spheres have real coefficients, see Remark \ref{remarksfere}, we can pull them on the left and write
\[
\begin{split}
k(q)&=\prod_{n=1}^\infty \left(1- 2q \frac{{\rm Re}(\gamma_n)}{|\gamma_n|^2} + q^2 \frac{1}{|{\gamma}_n|^2}\right)\exp\left({2q\frac{{\rm Re}(\gamma_n)}{|\gamma_n|^2}+\ldots +2\frac{1}{p_n}q^{p_n}\frac{{\rm Re}(\gamma_n^{p_n})}{|\gamma_n|^{2p_n}}}\right) h(q)\\
&=\left(\prod_{n=1}^\infty  (1- q \gamma_n^{-1})\star e_{p_n}(q,\gamma_n^{-1})\star (1- q \bar\gamma_n^{-1})\star e_{p_n}(q,\bar\gamma_n^{-1})\right) h(q).
\end{split}
\]
Let us set
$$
S(q)=\prod_{n=1}^\infty \left(1- 2q \frac{{\rm Re}(\gamma_n)}{|\gamma_n|^2} + q^2 \frac{1}{|{\gamma}_n|^2}\right)\exp\left({2q\frac{{\rm Re}(\gamma_n)}{|\gamma_n|^2}+\ldots +2\frac{1}{p_n}q^{p_n}\frac{{\rm Re}(\gamma_n^{p_n})}{|\gamma_n|^{2p_n}}}\right)
$$
so that we can rewrite
$$
k(q)=S(q)h(q).
$$
We now multiply \eqref{kappa} on the right by $$\underset{\infty}{\overset{\star 1}\prod} e_{p_n}(q,\bar \delta_n{-1})\star (1-q\bar \delta_n^{-1})$$ and we obtain
$$
k(q)\star \underset{\infty}{\overset{\star 1}\prod} e_{p_n}(q,\bar \delta_n{-1})\star (1-q\bar \delta_n^{-1})=
f(q)\star S(q),
$$
from which we deduce
$$
S(q)h(q)\star \underset{\infty}{\overset{\star 1}\prod} e_{p_n}(q,\bar \delta_n{-1})\star (1-q\bar \delta_n^{-1})=
f(q)\star S(q)=S(q) f(q).
$$
By multiplying on the left by $S(q)^{-1}$ we finally have
$$
h(q)\star \underset{\infty}{\overset{\star 1}\prod} e_{p_n}(q,\bar \delta_n{-1})\star (1-q\bar \delta_n^{-1})=
f(q)
$$
so the statement follows.
\end{proof}
Using the notation in Remark \ref{rmknotation} to denote the list of zeros of a function, we can prove the following:
\begin{theorem}[Weierstrass Factorization Theorem]\label{thmWeier}
Let $f$ be an entire slice regular function and let $f(0)\not=0$. Suppose $\{\gamma_n\}$ are the zeros of $f$ repeated according
to their multiplicities. Then there exist a sequence $\{p_n\}$ of nonnegative integers, a sequence $\{\delta_n\}$ of quaternions, and a never vanishing entire slice regular function $h$ such that
$$f(q)= g(q) \star h(q)$$
where
\begin{equation}\label{Gqp}
g(q)=\overset{\star\infty}{\underset{n=1}\prod}(1-q\delta_n^{-1} )\star e_{p_n}(q, \delta_n^{-1}),
\end{equation}
$\delta_n\in[\gamma_n]$. In particular, $\delta_n=\gamma_n$ if $\gamma_n\in\mathbb R$, $\delta_n=\beta_n$, and $\delta_{n+1}=\overline{\beta}_n$ when $[{\beta}_n]$ is a spherical zero and, in this case, $p_{n+1}$ may be chosen equal to $p_{n}$.
Moreover, for all $n\in \mathbb N$, the function $e_{p_n}(q, \delta_n^{-1})$ is given by
$$
e_{p_n}(q,\delta_n^{-1})=e_\star^{q\delta_n^{-1}+\ldots +\frac{1}{p_n}q^{p_n}\delta_n^{-{p_n}}}.
$$
Under the same hypotheses, if $f$ has a zero of multiplicity $m$ at $0$ then
 $$
f(q)= q^m g(q) \star h(q),
 $$
 where $g$ and $h$ are as above.
\end{theorem}

\begin{proof} Let us consider first the  real and spherical zeros of $f$. Then, by Theorem \ref{teo8}
 $$g(q)=\left( \underset{n=0}{\overset{\star\infty}\prod}(1-q\delta_n^{-1} )\star e_{p_n}(q,\delta_n^{-1})\right)$$
is a slice regular function which has the listed zeros if the $\delta_n$ are chosen to be the real zeros of $f$ or the pairs $\beta_n$, $\overline{\beta_n}$.
Since all the corresponding factors
  $$(1-q\delta_n^{-1} )\star e_{p_n}(q,\delta_n^{-1})$$ or
   $$
   (1-q\beta_n^{-1} )\star e_{p_n}(q,\beta_n^{-1})\star (1-q\bar\beta_n^{-1} )\star e_{p_n}(q,\bar\beta_n^{-1})
   $$
   have real coefficients, they commute with each other. Let us set $f_1(q)= g(q)^{-\star}\star f(q)$.
 The slice regular function $f_1(q)$ cannot have real of spherical zeros as the zeros of $f$ cancel with the factors in $g(q)^{-\star}$. Thus $f(q)= g(q) \star f_1(q)$ where $f_1$ can have isolated zeros, if $f$ possesses isolated zeros.
  We now consider the function $f_1^c(q)$ and we recall that its zeros are in one-to-one correspondence with the zeros of $f_1(q)$.
  By Lemma \ref{isolatedW}, we have
  $$
  f_1^c(q)=h(q)\star  \overset{\star\infty}{\underset{n=1}\prod} e_{p_n}(q, \bar{\delta}_n^{-1}) \star (1-q\bar{\delta}_n^{-1} )
  $$
  where $h$ is a suitable entire slice regular function. Consequently, we have
  $$
  f_1(q) =  \overset{\star\infty}{\underset{n=1}\prod} (1-q{\delta}_n^{-1} )\star e_{p_n}(q, {\delta}_n^{-1})  \star h^c(q)
  $$
  and the statement follows.
 If $f$ has a zero of multiplicity $m$ at $0$ it suffices to apply the previous reasoning to the function $q^{-m}f(q)$.
\end{proof}
\begin{remark}{\rm
We note that, in the case of an isolated zero $\alpha_n$, $\delta_n$ must be chosen is a suitable way in order to obtain the desired zero.  When the list of zeros contains a spherical zero, namely a pair $\beta_n$, $\overline{\beta_n}$ we can choose $\delta_n=\beta_n$ and $\delta_{n+1}=\overline{\beta_n}$ so that the product contains factors of the form illustrated in Remark \ref{remarksfere}.}
\end{remark}
As a corollary of Theorem \ref{thmWeier}, if we factor first the real zeros, then all the spherical zeros using the previous remark, we obtain the theorem as written in \cite{GSS,MR2836832}.

\begin{theorem}[Weierstrass Factorization Theorem]
Let $f$ be an entire slice regular function. Suppose that: $m\in \mathbb{N}$ is the multiplicity of $0$ as a zero of $f$,
$\{b_n\}_{n\in\mathbb{N}}\subseteq \rr\setminus \{0\}$ is the sequence of the
other real zeros of $f$,  $\{[s_n]\}_{n\in\mathbb{N}}$
is the sequence of the spherical zeros of $f$,  and $\{a_n \}_{n\in\mathbb{N}}\subseteq \hh\setminus{\mathbb{R}}$
is the sequence of the non real zeros of $f$ with isolated multiplicity greater then zero.
If all the zeros listed above are repeated according
to their multiplicities, then there exists a never vanishing, entire slice regular function $h$ and, for all $n\in\mathbb{N}$,  there exist $c_n\in [s_n]$
 and $\delta_n\in [Re(a_n)+I|Im(a_n)|]$ such that
$$f(q)=q^m\; \mathcal{R}(q) \; \mathcal{S}(q)\; \mathcal{A}(q) \star h(q)$$
where
$$\mathcal{R}(q)=\prod_{n=0}^{\infty}(1-qb_n^{-1})e^{qb_n^{-1}+\ldots+\frac{1}{n}q^{n}b_n^{-n}},$$

$$\mathcal{S}(q)=\prod_{n=0}^{\infty}\left(\frac{q^2}{|c_n|^2} - \frac{2q Re(c_n)}{|c_n|^2} + 1\right)e^{q\frac{Re(c_n)}{|c_n|^2}+
\ldots +\frac{1}{n}q^{n}\frac{Re(c_n^{n})}{|c_n|^{2n}}},$$

$$\mathcal{A}(q)=\overset{\star\infty}{\underset{n=0}\prod}(1-q\delta_n^{-1} )\star e_n(q)$$
and where, for all $n\in \mathbb N$,  $e_n(q, \delta_n^{-1})=e_\star^{q\delta_n^{-1}+\ldots +\frac{1}{n}q^{n}\delta_n^{-{n}}}$.
\end{theorem}
\begin{corollary}\label{corW}
Let $f$ be an entire slice regular function. Then $f$ can be written as $f=gh$ where $g$ and $h$ are entire slice regular, moreover $g$ is such that $g=1$ or $g$ has spherical or real zeros only and $h$ is such that $h=1$ or $h$ has isolated zeros only.
\end{corollary}

In the Weierstrass factorization theorem we used the factors
$$
G(q\delta_n^{-1}, p_n)=(1-q\delta_n^{-1})\star e_{p_n}(q\delta^{-1})
$$
see \eqref{Gqp}. Note that the decomposition in factors obtained in this way is far from being unique, in fact the sequence $\{p_n\}$ is not unique. As we have already pointed out, sometimes it is possible to choose all the $p_n$ equal to a given integer, see Remark \ref{remark22}. It is then interesting to ask whether there exists the smallest of such numbers.  More precisely, assume that the series
\begin{equation}\label{serieslambda}
\sum_{n=1}^\infty |\delta_n^{-1}|^\lambda
\end{equation}
converges for some $\lambda\in\mathbb R^+$. Let $p$ be the smallest natural number such that
$$
\sum_{n=1}^\infty |\delta_n^{-1}|^{p+1}
$$
converges. Then, for $|q|\leq R$, $R>0$, also the series
$$
\sum_{n=1}^\infty |q\delta_n^{-1}|^p
$$
converges uniformly and so does the product
\begin{equation}\label{canonical}
\prod_{n=0}^{\star\infty}G(q\delta_n^{-1}, p).
\end{equation}
This discussion leads to the following definition:
\begin{definition}
We will say that \eqref{canonical} is a {\em canonical product} \index{canonical product} and that $p$ is the {\em genus}
\index{genus} of the canonical product.
\end{definition}
In the Weierstrass representation of a function it is convenient to choose the canonical product, instead of another representation. Let $q$ be the sum of the degrees of the exponents in the functions $e_{p}$, if it is finite. Then we have:
\begin{definition}
The genus of an entire function $f$ is defined as $\max (p,q)$ where $p,q$ are as above and are finite.
If $q$ is not finite or the series \eqref{serieslambda} diverges for all $\lambda$ we say that the genus of $f$ is infinite.
\end{definition}

\section{Blaschke products}
Another useful example of infinite product is the one obtained using the so-called Blaschke factors. In the quaternionic setting it is convenient to write a Blaschke factor distinguishing the case of a isolated zero or of a spherical zero. We start by describing the Blaschke factors related to isolated zeros.
\begin{definition}
Let $a\in\mathbb{H}$, $|a|<1$. The function
\begin{equation}
\label{eqBlaschke1}
B_a(q)=(1-q\bar a)^{-\star}\star(a-q)\frac{\bar a}{|a|}
\end{equation}
is called a Blaschke factor at $a$.\index{Blaschke factor!quaternionic, unit ball}
 \end{definition}
 \begin{remark}\label{Bpointwise}{\rm
 Using
 Theorem \ref{pointwise}, $B_a(q)$ can be rewritten  in terms of the pointwise multiplication.
In fact, by setting
 $\lambda(q)=1-q\bar a$ we can write
$$
(1-q\bar
a)^{-\star}=(\lambda^c(q)\star \lambda(q))^{-1}\lambda^c(q).
$$
Applying formula (\ref{productfg}) to the products $\lambda^c(q)\star\lambda(q)$ and $\lambda^c(q)\star (a-q)$,
the Blaschke factor (\ref{eqBlaschke1}) may be written as
\begin{equation}\label{lineartransf}
\begin{split}
B_a(q)&=(\lambda^c(q)\star \lambda(q))^{-1}\lambda^c(q)\star (a-q)\frac{\bar a}{|a|}\\
&=(\lambda^c(q)\lambda (\tilde q))^{-1}\lambda^c(q)(a-\tilde q)\frac{\bar a}{|a|}\\
&=\lambda (\tilde q)^{-1}(a-\tilde q)\frac{\bar a}{|a|}=(1-\tilde{q}\bar a)^{-1}(a-\tilde q)\frac{\bar a}{|a|},\\
\end{split}
\end{equation}
where $\tilde q=\lambda^c(q)^{-1} q \lambda^c(q)$.
 Summarizing, we have
$$
 B_a(q)= (1-\tilde{q}\bar a)^{-1}(a-\tilde q)\frac{\bar a}{|a|}
 $$
 where  $\tilde q=(1-q a)^{-1} q (1-q a)$.
 }
 \end{remark}
  The following result immediately follows from the definition:
 \begin{proposition}
 Let $a\in\mathbb{H}$, $|a|<1$. The Blaschke factor $B_a$ is a slice hyperholomorphic function in $\mathbb B$.
 \end{proposition}

\begin{theorem}\label{elrhfn}
Let $a\in\mathbb{H}$, $|a|<1$. The Blaschke factor $B_a$ has
the following properties:
\begin{enumerate}
\item[(1)]
it takes the unit ball $\mathbb{B}$ to itself;
\item[(2)] it takes the boundary of the unit ball to itself;
\item[(3)] it has a unique zero for $q=a$.
\end{enumerate}
\end{theorem}
\begin{proof}
Using Remark \ref{Bpointwise}, we rewrite $B_a(q)$ as
$B_a(q)=(1-\tilde q\bar a)^{-1}(a-\tilde q)\dfrac{\bar a}{|a|}$.
 Let us show
that $|q|=|\tilde q|<1$ implies  $|B_a(q)|^2<1$. The latter inequality is
equivalent to
$$
|a-\tilde  q|^2<|1-\tilde q\bar a|^2
$$
which is also equivalent to
\begin{equation}\label{dffiuygah}
|a|^2+|q|^2<1+|a|^2|q|^2.
\end{equation}
Then (\ref{dffiuygah}) becomes $(|q|^2-1)(1-|a|^2)<0$ and it holds when $|q|<1$.
When $|q|=1$ we set $q=e^{I \theta}$, so that $\tilde q=e^{I '\theta}$  and we have
$$
|B_a(e^{I \theta})|=|1-e^{I '\theta}\bar
a|^{-1}|a-e^{I '\theta}|\frac{|\bar a|}{|a|}=|e^{-I '\theta}-\bar
a|^{-1}|a-e^{I '\theta}|=1.
$$
Finally, from (\ref{lineartransf}) it follows that  $B_a(q)$ has only one zero that comes from the factor $a-\tilde q$. Moreover
$$B_a(a)=(1-\tilde{a}\bar a)^{-1}(a-\tilde a)\frac{\bar a}{|a|}$$
where $$\tilde a= (1-a^2)^{-1}a (1-a^2)=a$$ and thus $B_a(a)=0$.
\end{proof}

 As we have just proved, $B_a(q)$ has only one zero at $q=a$ and
 analogously to what happens in the case of the zeros of a function, the product of two Blaschke factors of the form $B_a(q)\star B_{\bar a}(q)$ gives the Blaschke factor with zeros at the sphere $[a]$. Thus we give the following definition:
\begin{definition}
Let $a\in\mathbb{H}$, $|a|<1$. The function
\begin{equation}
\label{blas_sph}
B_{[a]}(q)=(1-2{\rm Re}(a)q+q^2|a|^2)^{-1}(|a|^2-2{\rm Re}(a)q+q^2)
\end{equation}
is called  Blaschke factor at the sphere $[a]$.
 \end{definition}

\begin{theorem}\label{converge}
Let $\{a_j\}\subset \mathbb{B}$, $j=1,2,\ldots$ be a sequence
of nonzero quaternions and assume that  $\sum_{j\geq 1} (1-|a_j|)<
\infty$. Then the function
\begin{equation}\label{B_product}
B(q):=\prod^\star_{j\geq 1}(1-q\bar
a_j)^{-\star}\star(a_j-q)\frac{\bar a_j}{|a_j|},
\end{equation}
 converges
uniformly on the compact subsets of $\mathbb{B}$ and defines a slice hyperholomorphic function.
\end{theorem}
\begin{proof}
Let $b_j(q):=
B_{a_j}(q)-1$.
We rewrite $b_j(q)$ using
 Remark \ref{Bpointwise}  and we have:
\[
\begin{split}
b_j(q)=& B_{a_j}(q)-1\\
&=(1-\tilde q \bar a_j)^{-1}(a_j-\tilde q)\frac{\bar a_j}{|a_j|}-1\\
=& (1-\tilde q \bar a_j)^{-1}\left[(a_j-\tilde q)\frac{\bar a_j}{|a_j|}-(1-\tilde q \bar a_j)\right]\\
=& (1-\tilde q \bar a_j)^{-1}\left[(|a_j|-1)\left(1+\tilde q\frac{\bar a_j}{|a_j|}\right)\right].
\end{split}
\]
Thus,  recalling that $|\tilde q|=|q|$ and $|q|<1$, we have
$$
|b_j(q)|\leq 2 (1-|q|)^{-1}(1-|a_j|)
$$
and since $\sum_{j= 1}^\infty (1-|a_j|)<
\infty$ then $\sum_{j=1}^\infty |b_j(q)|=\sum_{j=1}^\infty |B_j(q)-1|$ converges uniformly on the compact subsets of $\mathbb B$. The statement follows from Proposition \ref{prop2:23}.
\end{proof}

To assign a Blaschke product having zeros at a given set of points $a_j$ with multiplicities $n_j$, $j\geq 1$ and at spheres $[c_i]$ with multiplicities $m_i$, $i\geq 1$, we may think to take the various factors a number of times corresponding to the multiplicity of the zero. However, we know that the polynomial $(p-a_j)^{\star n_j}$ is not the unique polynomial having a zero at $a_j$ with the given multiplicity $n_j$, thus the Blaschke product $\prod^{\star n_j}_{j=1} B_{a_j}(p)$ is not the unique Blaschke product having zero at $a_j$ with multiplicity $n_j$.\\
  Thus we have the following result which takes into account the multiplicity of the zeros in the most general way: \begin{theorem}\label{blaschke zeros}\index{Blaschke product!quaternionic, unit ball}
A Blaschke product having zeros at the set
 $$
 Z=\{(a_1,n_1),  \ldots, ([c_1],m_1), \ldots \}
 $$
 where $a_j\in \mathbb{B}$, $a_j$ have respective multiplicities $n_j\geq 1$, $a_j\not=0$ for $j=1,2,\ldots $, $[a_i]\not=[a_j]$ if $i\not=j$, $c_i\in \mathbb{B}$, the spheres $[c_j]$ have respective multiplicities $m_j\geq 1$,
 $j=1,2,\ldots$, $[c_i]\not=[c_j]$ if $i\not=j$
and
\begin{equation}\label{conditionconvergence}
\sum_{i,j\geq 1} \Big(n_i (1-|a_i|)+
2 m_j(1-|c_j|)\Big)<\infty
\end{equation}
is of the form
\begin{equation}
\label{blaschke123}
\prod_{i\geq 1} (B_{[c_i]}(q))^{m_i}\prod_{i\geq 1}^\star \prod_{\ j=1}^{\star  n_i} (B_{\alpha_{ij}}(q)),
\end{equation}
where $n_j\geq 1$, $\alpha_{11} = a_1$ and $\alpha_{ij}$ are suitable elements in $[a_i]$ for $j=2,3,\ldots$.
\end{theorem}
\begin{proof}
 The hypothesis (\ref{conditionconvergence}) and Theorem \ref{converge} guarantee that the infinite product converges. The zeros of the pointwise product $\prod_{i\geq 1} (B_{[c_i]}(q))^{m_i}$ correspond to the given spheres with their multiplicities. Consider now the product:
\[
\prod_{i=1}^{\star  n_1} (B_{\alpha_{i1}}(q))=B_{\alpha_{11}}(q)\star B_{\alpha_{12}}(q) \star \cdots \star B_{\alpha_{1 n_1}}(q).
\]
 From the definition of multiplicity, this product admits a zero at the point $\alpha_{11}=a_1$. This zero has multiplicity $1$ if $n_1=1$; if $n_1\geq 2$, the other zeros are $\tilde\alpha_{12}, \ldots, \tilde\alpha_{1 n_1}$ where $\tilde\alpha_{1j}$ belong to the sphere $[\alpha_{1j}]=[a_1]$. Since there cannot be other zeros on $[a_1]$ different from $a_1$ (otherwise the whole sphere $[a_1]$ would be a sphere of zeros), we conclude that $a_1$ has multiplicity $n_1$.
This fact can be seen directly using formula  (\ref{productfg}).
Let us now consider $r\geq 2$ and
\begin{equation}\label{rfactor}
\prod_{j=1}^{\star  n_r} (B_{\alpha_{rj}}(q))=B_{\alpha_{r1}}(q) \star \cdots \star B_{\alpha_{r n_r}}(q),
\end{equation}
and set
\[
B_{r-1}(q):= \prod_{i\geq 1}^{\star (r-1)} \prod_{j=1}^{\star  n_i} (B_{\alpha_{ij}}(q)).
\]
Then
\[
B_{r-1}(q)\star B_{\alpha_{r1}}(q)= B_{r-1}(q) B_{\alpha_{r1}}(B_{r-1}(q)^{-1}qB_{r-1}(q))
\]
has a zero  at
 $a_r$ if and only if $$B_{\alpha_{r1}}(B_{r-1}(a_r)^{-1}a_rB_{r-1}(a_r))=0,$$ i.e. if and only if $$\alpha_{r1}=B_{r-1}(a_r)^{-1}a_rB_{r-1}(a_r).$$ If $n_r=1$ then $a_r$ is a zero of multiplicity $1$ while if $n_r\geq 2$, all the other zeros of the product (\ref{rfactor}) belongs to the sphere $[a_r]$ thus the zero $a_r$ has multiplicity $n_r$.
 This finishes the proof.
\end{proof}

\vskip 1 truecm
\noindent{\bf Comments to Chapter 4}. Infinite products of quaternions can be treated by adapting the arguments from the complex setting, see for example \cite{rudin}, to the quaternionic setting. Some results on the quaternionic logarithm come from \cite{MR2836832}. The Weierstrass theorem was originally proved in \cite{MR2836832}. The version of the theorem proved in this chapter is equivalent but obtained in a slightly different way. The Blaschke products (in the ball case) has been treated in several articles but appeared for the first time in \cite{MR3127378}. Blaschke products in the half space have been introduced in \cite{acls_milan}.

%
%
%


%
%
%


\chapter{Growth of entire slice regular functions}

\section{Growth scale}
The simplest example of entire slice regular functions is given by polynomials (with coefficients on one side).
In the classical complex case the growth of a polynomial is related to its degree and thus to the number of its zeros.
In the quaternionic case this fact is true up to a suitable notion of ''number of zeros''. In fact, as we have seen through the book, there can be spheres of zeros, and thus an infinite number of zeros, even when we consider polynomials.
However, if we count the number of spheres of zeros and the number of isolated zeros, we still have a relation with the degree of the polynomial.
In fact, each sphere is characterized by a degree two polynomial (see \eqref{companion}), and thus each sphere with multiplicity $m$ counts as a degree $2m$ factor, while each isolated zero of multiplicity $r$ counts as a degree $r$ factor.
\\
In this section we introduce the notion of order and type of an entire slice regular function, discussing its growth in relation with the coefficients of its power series expansion and with the density of its zeros.
\\
Let $f$ be an entire slice regular function and let $$M_{f_I}(r)=\max_{|z|=r,\, z\in\mathbb C_I} |f(z)|,$$
and
$$M_f(r)=\max_{|q|=r} |f(q)|.$$
As in the complex case, we have
\begin{proposition} Let $f$ be an entire slice regular function.
Then function $M_f(r)$ is continuous.
\end{proposition}
Moreover, in the case of intrinsic functions, we have
\begin{proposition} \label{Mfintr} Let $f$ be an entire slice regular function which is quaternionic intrinsic. Then
$$
M_f(r)=M_{f_I}(r),
$$
for all $I\in\mathbb S$.
\end{proposition}
\begin{proof}
Let us write $f(x+Iy)=\alpha(x,y)+I\beta(x,y)$ where $\alpha$, $\beta$ are real-valued functions, as $f$ is intrinsic. We have, for $q=x+Iy$,
\[
\begin{split}
M_f(r)&=\sup_{|q|=r}|f(x+Iy)|\\
&=\sup_{I\in\mathbb S}\sup_{ x^2+y^2=r^2} |f(x+Iy)|\\
&=\sup_{I\in\mathbb S}\sup_{ x^2+y^2=r^2} (\alpha(x,y)^2+\beta(x,y)^2)^{1/2}\\
&=\sup_{ x^2+y^2=r^2} (\alpha(x,y)^2+\beta(x,y)^2)^{1/2}\\
&=\sup_{ x^2+y^2=r^2} |f_I(x+Iy)|= M_{f_I}(r)\\
\end{split}\]
and the statement follows.
\end{proof}
\begin{definition}
Let $f$ be a quaternionic entire function.
Then $f$ is said to be of finite order if there exists $k>0$ such that
$$
M_f(r)<e^{r^k},
$$
for sufficiently large values of $r$ ($r>r_0(k)$). The greatest lower bound $\rho$ of such numbers $k$ is called the {\em order} \index{order} of $f$.
\end{definition}
From the definition of order, it immediately follow the inequalities below which will be useful in the sequel:
\begin{equation}\label{ineqMfr}
e^{r^{\rho-\varepsilon}} < M_f(r) < e^{r^{\rho+\varepsilon}}.
\end{equation}
\\
Let us recall the notations
$$
{\underset{r\to\infty}{\underline{\lim}}} \phi(r)=\lim_{r\to\infty}\inf_{t\geq r}\phi (t)
$$
and
$$
{\underset{r\to\infty}{\overline{\lim}}} \phi(r) =\lim_{r\to\infty}\sup_{t\geq r}\phi (t).
$$
Then, the inequalities \eqref{ineqMfr} are equivalent to
$$
\rho={\underset{r\to\infty}{\overline{\lim}}}\frac{\log(\log M_f(r))}{\log r}.
$$
This latter condition can be considered as an equivalent definition of order of an entire function $f$.\\
We now show that an entire slice regular function which grows slower than a positive power of $r$ is a polynomial:
\begin{proposition}
If there exists $N\in\mathbb N$ such that
$$
{\underset{r\to\infty}{\underline{\lim}}} \frac{M_f(r)}{r^N}<\infty
$$
then $f(q)$ is a polynomial of degree at most $N$.
\end{proposition}
\begin{proof}
Let $f(q)=\sum_{n=0}^\infty q^n a_n$ and set
$$
p_N(q)=\sum_{n=0}^N q^n a_n, \qquad g_N(q)=q^{-N-1}(f(q)-p_N(q)).
$$
Then, by our hypothesis, $g_N(q)=\sum_{j=0}^\infty q^j a_{N+1+j}$ is an entire regular function which tends to $0$ on a sequence of balls $|q|=r_j$ when $r_j\to\infty$. Thus $g_N$ vanishes on a sequence of balls intersected with a complex plane $\mathbb C_I$ and so, by the Identity Principle, $g_N\equiv 0$. We conclude that $f(q)$ coincides with $p_N(q)$.
\end{proof}
From this proposition, it follows that the growth of an entire regular function is larger than any power of the radius $r$. In the next definition we compare the growth of a function with functions of the form $e^{r^k}$:
\begin{definition}
Let $f$ be an entire regular function of order $\rho$ and let $A>0$ be such that for sufficiently large values of $r$
$$
M_f(r)<e^{Ar^\rho}.
$$
We say that $f$ of order $\rho$ is of type $\sigma$ if $\sigma$ is the greatest lower bound of such numbers $A$.\\
When $\sigma=0$  we say that $f$ is of minimal \index{type} type.\\
When $\sigma = \infty$ we say that $f$ is of maximal type. \\
When $0<\sigma<\infty$ we say that $f$ is of normal type. \index{type!normal} \index{type!maximal}
\index{type!minimal} \index{normal type} \index{maximal type} \index{minimal type}
\end{definition}
As in the case of the order, one can verify that the type of a function $f$ of order $\rho$ is given by
$$
\sigma={\underset{r\to\infty}{\overline{\lim}}}\frac{\log(M_f(r))}{r^\rho}.
$$
\begin{example}{\rm The function $\exp(q^n \sigma)$, where $n\in\mathbb N$, has type $\sigma$ and order $n$. }
\end{example}
\begin{definition}
We will say that the function $f(q)$ is of growth larger than $g(q)$ if the order of $f$ is larger than the order of $g$ or, if $f$ and $g$ have the same order and the type of $f$ is larger than the type of $g$.
\end{definition}
\begin{remark} From the definition of order and type, it follows that the order of the sum of two functions is not greater than the largest of the order of the summands. If one summand has order larger than the order of the other summand, then the sum has same order and type of the function of larger growth. If two functions have the same order and this is also the order of their sum, then the type of the sum is not greater than the largest of the type of the summands.
\end{remark}
We now relate the order and type of a function with the decrease of its Taylor coefficients. Recall that if $f(q)=\sum_{n=0}^\infty q^na_n$ is an entire slice regular function, then
$$
{\underset{n\to\infty}{\lim}} \sqrt[n]{|a_n|}=0.
$$
\begin{theorem}\label{thmorder}  The order and the type of the entire regular function $f(q)=\sum_{n=0}^{\infty} q^na_n$ can be expressed by
\begin{equation}
\rho={\underset{n\to\infty}{\overline{\lim}}} \frac{n\log(n)}{\log (\frac{1}{|a_n|})}
\end{equation}
\begin{equation}
(\sigma e\rho)^{1/\rho}= {\underset{n\to\infty}{\overline{\lim}}}\left(n^{\frac{1}{\rho}}\, \sqrt[n]{|a_n|}\right).
\end{equation}
\end{theorem}
\begin{proof}
The proof closely follows the proof of the corresponding results in complex analysis since it depends only on the modulus of the coefficients $a_n$, see e.g. \cite{levin}. We insert it for the sake of completeness.
\\
First of all we observe that if $f(q)$ is of finite order then, asymptotically (namely for sufficiently large $r$), we have
\begin{equation}\label{L1.08}
M_f(r)< \exp(Ar^k).
\end{equation}
The Cauchy estimates, see \eqref{Cauchyestimates}, give
$$
|a_n|\leq \frac{M_f(r)}{r^n},
$$
and so
$$
|a_n|\leq \frac{\exp(Ar^k)}{r^n}.
$$
An immediate computation gives that the maximum of the function $\dfrac{\exp(Ar^k)}{r^n}$ is given by
$(eAk/n)^{\frac nk}$ so we conclude that, asymptotically:
\begin{equation}\label{L1.09}
|a_n|\leq \left(\frac{eAk}{n}\right)^{\frac nk}.
\end{equation}
Assume now that \eqref{L1.09} is valid for $n$ large enough. Then
$$
|q^n a_n|< r^n \left(\frac{eAk}{n}\right)^{\frac nk}
$$
and if we take $n> \lfloor 2^k eAk r^k\rfloor=N_r$ (where $\lfloor x \rfloor$ denotes the integer part of $x$), we obtain $|q^n a_n|< 2^{-n}$ which yields
$$
|f(q)|< \sum_{j=0}^{N_r} r^j |a_j| +2^{-N_r}.
$$
Let $\mu(r)=\max_j r^j |a_j|$, then
\begin{equation}\label{L1.10}
M_f(r)\leq (1+2^k eAk r^k)\mu(r)+2^{-N_r}.
\end{equation}
As $f$ is not a polynomial, then $M_f(r)$ grows faster than any power of $r$. From \eqref{L1.09} it follows that,
asymptotically:
$$
\mu(r)\leq r^n \max_n \left(\frac{eAk}{n}\right)^{\frac nk}=e^{Ar^k},
$$
since the maximum is attained for $n=Akr^k$.
Using \eqref{L1.10} we have
\begin{equation}\label{L1.11}
M_f(r)< (2+2^k aAk^r)e^{Ar^k}.
\end{equation}
So we have that if $f$ is of finite order, then \eqref{L1.08} holds, but this implies \eqref{L1.09} which, in turns, implies \eqref{L1.11}. Thus the order $\rho$ equals the greatest lower bound of the numbers $k$ for which \eqref{L1.09} holds, while the type equals the greatest lower bound of numbers $A$ for which \eqref{L1.09} is valid for $k=\rho$.
\end{proof}
Since the conjugate $f^c$ of a slice regular functions expanded in power series has coefficients with the same modulus of the coefficients of $f$, we have:
\begin{corollary}\label{orderfc}
Let $f$ be an entire slice regular function. Its order and type coincide with order and type of its conjugate $f^c$.
\end{corollary}
Using Theorem \ref{thmorder} one can construct entire regular functions of arbitrary order and type: a function of order $n$ and type $\sigma$ is $e^{\sigma q^n}$ (compare with the classical complex case).
\\
Our next goal is to establish a dependence between the growth of a function and the density of distribution of its zeros and, in particular, we need to define how to count the zeros of a function. As we already observed, if a function has spherical zeros it automatically possess an infinite number of zeros, and so to introduce a notion of counting function, we need to treat in the appropriate way the spherical zeros. To this end, we introduce the following:

\noindent{\bf Assumption}. Assume that a function has zeros
$$
\alpha_1,\ldots, \alpha_n, \ldots, [\beta_1],\ldots, [\beta_m], \ldots
$$
with $\lim_{n\to\infty} |\alpha_n|=\infty$, $\lim_{n\to\infty} |\beta_n|=\infty$. \\
It is convenient to write a spherical zero as a pair of conjugate numbers, so we have:
$$
\alpha_1,\ldots, \alpha_n, \ldots, \beta_1, \overline{\beta}_1\ldots, \beta_m, \overline{\beta_m} \ldots.
$$
We will also assume that
 we arrange the zeros according to increasing values of their moduli, keeping  together the pairs $\beta_\ell, \overline{\beta}_\ell$ giving rise to spherical zeros. Then we rename $\gamma_s$, $s=1,2,\ldots$, the elements in the list so obtained:
\begin{equation}\label{sequencezeros}
\gamma_1, \gamma_2\ldots, \gamma_n, \ldots ,
\end{equation}
where $\gamma_s$ denotes one of the elements $\alpha_m$ or $\beta_\ell$ or $\overline{\beta_\ell}$.

As we shall see, the results obtained using this sequence will not depend on the chosen representative $\beta_\ell$ of a given sphere: what will matter is only the modulus $|\beta_\ell|$ which is independent of the representative chosen.
\begin{definition}
We define  the convergence exponent \index{convergence exponent} of the sequence \eqref{sequencezeros} to be the greatest lower bound of $\lambda$  for which the series
\begin{equation}\label{convexp}
\sum_{n=1}^\infty \frac{1}{|\gamma_n|^\lambda}
\end{equation}
converges. If the series \eqref{convexp} diverges for every $\lambda >0$ we will say that the convergence exponent is infinite.
\end{definition}
Note that, by its definition, the larger $|\gamma_n|$ become, the smaller becomes $\lambda$.\\
\begin{definition}
Consider a sequence of quaternions as in \eqref{sequencezeros}.
Let $n(r)$ be the number of elements in the sequence  belonging to the ball $|q|<r$. We say that $n(r)$ is the counting number or counting function of the sequence. \index{counting number} \index{counting function} The number $$
\rho_1={\underset{r\to\infty}{\overline{\lim}}}\frac{\log (n(r))}{\log(r)}
$$ is called order of the function $n(r)$.\\ \index{order!of the counting number}
The number
$$
\Delta={\underset{r\to\infty}{\overline{\lim}}}\frac{n(r)}{r^{\rho_1}}
$$
is called upper density \index{upper density} \index{density!upper} of the sequence \eqref{sequencezeros}, and if the limit exists, $\Delta$ is simply called density.\index{density}
\end{definition}
\begin{proposition}
The convergence exponent of the sequence $\{\gamma_n\}$ equals the order of the corresponding counting function $n(r)$.
\end{proposition}
\begin{proof}
The proof of this result is not related to quaternions, see \cite{levin}. We repeat it for the reader's convenience.
The series \eqref{convexp} is a series of real numbers which can be expressed by a Stieltjes integral in the form
$$
\int_0^\infty \frac{dn(t)}{t^\lambda},
$$
which is equal to
$$
\int_0^r \frac{dn(t)}{t^\lambda}=\frac{n(t)}{t^\lambda}+\lambda \int_0^r \frac{n(t)}{t^{\lambda +1}} dt.
$$
 If the series \eqref{convexp} is convergent, the two positive terms at the right hand side are bounded. Moreover
 the integral
 \begin{equation}\label{i1.20}
 \int_0^\infty \frac{n(t)}{t^{\lambda +1}} dt
 \end{equation}
 is convergent since it is increasing and bounded. Consequently, for any $\varepsilon >0$ and $r>r_0=r_0(\varepsilon)$ we have
 $$
 \frac{n(r)}{r^\lambda}=\lambda n(r)\int_r^\infty \frac{dt}{t^{\lambda+1}}\leq \lambda \int_r^\infty \frac{n(t) dt}{t^{\lambda+1}}<\varepsilon.
 $$
 So
 $$
 \lim_{r\to\infty} \frac{n(r)}{r^\lambda}=0
 $$
 and the order of $n(r)$ is not greater than $\lambda$. Conversely, the preceding reasoning shows that the convergence of the above integral \eqref{i1.20} implies the convergence of the series \eqref{convexp}. Let $\rho_1$ be the order of $n(r)$. Then for $t$ large enough we have
 $$
 n(t)< t^{\rho_1+\varepsilon/2}.
 $$
 By setting $\lambda=\rho_1+\varepsilon$ we have that \eqref{i1.20} converges and therefore \eqref{convexp} converges. So $\lambda$ is not greater than the order of $n(r)$ and the statement follows.
\end{proof}
\section{Jensen theorem}

In complex analysis, Jensen theorem states that if a function $f(z)$ is analytic in $|z|<R$ and such that $f(0)\not=0$, then
$$
\int_0^R \frac{n_f(t)}{t} dt =\frac{1}{2\pi}\int_0^{2\pi}\log|f(Re^{i\theta})| d\theta -\log|f(0)|,
$$
where $n_f(t)$ \index{$n_f(t)$} is the number of zeros  of $f$ in the disc $|z|<t$.
\\
In this section we prove
 an analog of this theorem in the slice regular case, where $n_f$ is the counting number of the sequence constructed with the zeros of $f$ in the ball $|q|<t$.
\\
To state Jensen theorem we need some preliminary lemmas and the following definition:

\begin{definition}
Let $f$ be a function slice regular in a ball centered at the origin and with radius $R>0$. Let $\gamma=\{\gamma_n\}$, with $|\gamma_n|\to\infty$ be the sequence of its zeros written according to the assumption in the previous section. We denote by $n_{f,I}(t)$ the number of elements in $\gamma$ belonging to the disc $\{q\ : \ |q|<t \}\cap\mathbb C_I$, for $t\leq R$.
We denote by $n_{f}(t)$ the number of elements in $\gamma$ belonging to the ball $\{q\ : \ |q|<t \}$, for $t\leq R$.
\end{definition}
We note that since the choice of the elements $\beta_\ell$, $\overline{\beta}_\ell$ representing the sphere $[\beta_\ell]$ is arbitrary, on each complex plane $\mathbb C_I$ we can find a pair of representatives of $[\beta_\ell]$. Thus, $n_{f,I}(t)$ is the sum of the number of isolated zeros $\alpha_\ell$ which belong to the disc $|z|<t$ in the complex plane $\mathbb C_I$ and of twice the number of spheres $[\beta_\ell]$ in that same disc.
\begin{lemma}\label{lemmanumezeri}
If $g$ is quaternionic intrinsic in $B(0;t)$, $t>0$, then $n_{g}(t)=n_{g,I}(t)$.
\end{lemma}
\begin{proof}
If the function $g$ has a nonreal zero at $q_0=x_0+Jy_0$, then $0=\overline{g(q_0)}=g(\bar q_0)$, thus $g$ vanishes at $\bar q_0$ and so it has a spherical zero at $[q_0]$. Assume that the spherical zero $[q_0]$ has multiplicity $m$, namely $$g(q)=((q-x_0)^2+y_0^2)^m \tilde g (q)$$ where $\tilde g(x_0+Iy_0)\not=0$ for all $I\in\mathbb S$.  Then for every $I\in\mathbb S$, on the complex plane $\mathbb C_I$ the function $g$ has zeros at $x_0\pm I y_0$, each of which with multiplicity $m$.\\
 The real zeros of $g$ belong to every $\mathbb C_I$. Thus, the number of points $\alpha_ell$, being real, is constant for all $I\in\mathbb S$. Repeating the reasoning for all the zeros of $g$, we deduce that $n_{g}(t)=n_{g,I}(t)$, for all $I\in\mathbb S$.
\end{proof}

\begin{lemma}\label{lemmanumezeri11}
If $h$ is slice regular in $B(0;t)$, $t>0$,  and it has isolated zeros only, then $$n_h(t)=\frac 12 n_{h^s}(t).$$
\end{lemma}
\begin{proof}
To prove the statement assume that in the ball $\{q\ : \ |q|<t\}$ there are $N(t)$ distinct isolated zeros and that
$$
h(q)=\prod^{\star N(t)}_{r=1} (q- \alpha_{r 1}) \star \ldots \star (q- \alpha_{r j_r}) \star \tilde h(q),
$$
where, for all $r=1,\ldots ,N(t)$, $j_r\geq 1$ denotes the multiplicity of the zero $\tilde \alpha_{r 1}\in [\alpha_{r 1}]$ (note that only the zero $\alpha_{11}$ can be immediately read from the factorization of $h$). Thus $n_h(t)=
\sum_{r=1}^{N(t)} j_r$.
Then we have
$$
h^c(q)=\tilde h^c(q) \star  \prod^{\star 1}_{r= N(t)} (q- \overline{\alpha_{r j_r}}) \star \ldots \star (q- \overline{\alpha_{r 1}}) ,
$$
(where $\prod^{\star 1}_{r= N(t)}$ indicates that we are taking the products starting with the index $N(t)$ and ending with $1$)
and, by setting $\alpha_{r 1} =x_r+ I_r y_r$, $r=1,\ldots, N(t)$, we deduce
\[
\begin{split}
h^s(q) & = \prod^{\star N(t)}_{r=1} (q- \alpha_{r 1}) \star \ldots \star (q- \alpha_{r j_r}) \star \tilde h(q)
\star \tilde h^c(q) \star  \prod^{\star 1}_{r=N(t)} (q- \overline{\alpha_{r j_r}}) \star \ldots \star (q- \overline{\alpha_{r 1}})\\
& = \prod_{r=1}^{N(t)} ((q-x_r)^2 +y_r^2)^{j_r} \tilde h^s(q).
\end{split}
\]
Since $\tilde h^s$ does not vanish on $|q|<r$ we have that $n_h(t)=
\sum_{r=1}^{N(t)} 2j_r$ and the statement follows.
\end{proof}
\begin{lemma}
Let $f$ be a function, not identically zero and slice regular in a domain containing $\overline{B(0;R)}$. Then there exist two  functions $g,h$ slice regular
in $B(0;R)$ and continuous in $\overline{B(0;R)}$
such that  $f=gh$, where $g$ has at most spherical only,  and $h$ has at most isolated zeros only.
\end{lemma}
\begin{proof}
We observe that $f$ can have only a finite number of spherical zeros in $B(0;R)$. Otherwise, if these zeros were an infinite number, then on a complex plane $\mathbb C_I$ the function $f$ would admit an infinite number of zeros with an accumulation point in $\overline{B(0;R)}\cap\mathbb C_I$ and so, by the Identity Principle, the function $f$ would be identically zero, contradicting the assumption. Thus this finite number of zeros can be pulled out on the left using Theorem \ref{R-fattorizzazslice}. The product of this finite number of factors corresponding to these zeros (and thus all with real coefficients) gives the function $g$. If $f$ does not have any spherical zero, we can set $g=1$ on $B(0;R)$. The factor $h$ contains at most isolated zeros. Both the functions $g$ and $h$ are slice regular and continuous where needed, by construction.
\end{proof}

\begin{theorem}[Jensen theorem]
Let $f$ be a function slice regular in a domain containing $\overline{B(0;R)}$ and assume that $f(0)\not=0$.
Let $f=gh$ where $g$ has at most spherical zeros only,  and $h$ has at most isolated zeros only.
Then
\begin{equation}
\begin{split}
\int_0^R \frac{n_f(t)}{t} dt &= \frac{1}{2\pi}\left[\int_0^{2\pi} \log |g(Re^{I\theta})|\, d\theta -\log|g(0)|\right. \\
&+
\left. \frac 12 \left(\int_0^{2\pi} \log |h^s(Re^{I\theta})|\, d\theta -\log|h^s(0)|\right)\right].
\end{split}
\end{equation}
\end{theorem}

\begin{proof}
First of all we observe that
$$n_f(t)=n_g(t)+n_h(t)=n_g(t)+\frac 12 n_{h^s}(t)$$ where the last equality follows from Lemma \ref{lemmanumezeri11}.
Then we have
\begin{equation}
\int_0^R \frac{n_f(t)}{t} dt  = \int_0^R \frac{n_{g}(t)}{t} dt +\frac 12 \int_0^R \frac{n_{h^s}(t)}{t} dt.
\end{equation}
 Since both $g$ and $h^s$ are quaternionic intrinsic functions we have
 $$
 g_I, h^s_I: \ B(0,t)\cap\mathbb C_I \to \mathbb C_I
 $$
and so the restrictions $g_I, h^s_I$ to any complex plane $\mathbb C_I$ are holomorphic functions to which we can apply the complex Jensen theorem, see Theorem 5, p. 14 in \cite{levin}, namely
 $$
\int_0^R \frac{n_{g_I}(t)}{t} dt = \frac{1}{2\pi}\int_0^{2\pi} \log |g(Re^{I\theta})|\, d\theta -\log|g(0)|,
 $$
 and similarly for $h^s_I$.
 Then, the statement follows from Lemma \ref{lemmanumezeri}.
\end{proof}
 We now prove the following consequence of Jensen theorem, characterizing the number of zeros of $f$ in a ball.
\begin{proposition} [Jensen inequality]\label{prop43}\index{Jensen inequality}
Let $f$ be slice regular in the ball centered at the origin and with radius $er$, where $r>0$ and $e$ is the Napier number.
 Let $f=gh$ where $g$ has at most spherical zeros only, and $h$ has at most isolated zeros only and assume that $|g(0)|=|h(0)|=1$.
Then
\begin{equation}\label{ineqJ}
n_f(r)\leq \log (M_g(er)M_h(er)).
\end{equation}
The bound $\log (M_g(er)M_h(er))$ is optimal, as equality may occur.
\end{proposition}
\begin{proof}
Since $|g(0)|=|h(0)|=1$ and $f(0)=g(0)h(0)$ we have $|f(0)|=1$. We can apply Jensen theorem, and the fact that $n_f(r)$ is evidently a monotone function in $r$ yields
\begin{equation}\label{chainineq}
\begin{split}
n_f(r)&\leq \int_r^{R}\frac{n_f(t)}{t}\, dt
\\
& \leq \frac{1}{2\pi}\left[\int_0^{2\pi} \log |g(R e^{I\theta})|\, d\theta+
\frac 12 \int_0^{2\pi} \log (|h^s(R e^{I\theta})|\, d\theta \right]\\
&\leq \log M_g(R)+\frac 12 \log M_{h^s}(R),
\end{split}
\end{equation}
where we have set $R=er$.
Recalling \eqref{fcab}, we have
\[
\begin{split}
\sup_{|q|=R} |h^c(q)|&= \sup_{I\in \mathbb S}\sup_{\theta\in[0,2\pi)}|h^c(Re^{I\theta})|\\
&= \sup_{I\in \mathbb S}
\sup_{\theta\in[0,2\pi)} |\overline{\alpha(R,\theta)} + I\overline{\beta (R,\theta)}|\\
&=\sup_{I\in \mathbb S}
\sup_{\theta\in[0,2\pi)} |\alpha (R,\theta) +I\beta (R,\theta) |\\
&=\sup_{I\in \mathbb S}\sup_{\theta\in[0,2\pi)}|h(Re^{I\theta})|=\sup_{|q|=R} |h(q)|
\end{split}
\]
and so, using \eqref{pointwise} and \eqref{fcab}, we deduce
\begin{equation}\label{supstar}
\begin{split}
\sup_{|q|=R} |h^s(q)|&=\sup_{|q|=R} |h(q)\star h^c(q)|\\
&=\sup_{|q|=R} |h(q)||h^c(\tilde q)| \\
&\leq \sup_{|q|=R} |h(q)|^2\\
&\leq (\sup_{|q|=R} |h(q)|)^2.
\end{split}
\end{equation}
From \eqref{supstar} we obtain
\begin{equation}\label{chainineq1}
\begin{split}
\log M_g(R)+\frac 12 \log M_{h^s}(R)&\leq \log M_g(R)+ \log (M_{h}(R))\\
&=\log (M_g(R)M_h(R)),
\end{split}
\end{equation}
and so
$$
n_f(r)\leq \log (M_g(R)M_h(R)),
$$
which proves the first part of the statement.
We have equalities, instead of inequalities, in \eqref{chainineq}
if
$$
n_f(r) = \int_r^{R}\frac{n_f(t)}{t}\, dt
$$
which means that $f$ cannot have zeros in $\{q\in\mathbb H\ : \ r<|q|<R\}$.
Moreover, we must have the equality
$$
\int_r^{R}\frac{n_f(t)}{t}\, dt =
\frac{1}{2\pi}\left[\int_0^{2\pi} \log |g(R e^{I\theta})|\, d\theta+
\frac 12 \int_0^{2\pi} \log |h^s(R e^{I\theta})|\, d\theta \right]
$$
which, by Jensen formula, means that $f$ has no zeros in the ball $B(0;r)$.

Finally, if we have an equality instead of the last inequality in \eqref{chainineq} then   $|f(er e^{I\theta})|=M_g(er)M_h(er)$ on the sphere $|q|=er$.
Consider now a function of the form
$$
\varphi (q)= e^{n+i\alpha} \prod^{\star n}_{k=1} (R^2 - q \bar a_k )^{-\star}\star (q-a_k) er,
$$
where $|a_k|=r$ and the factor $e^{n+i\alpha}$ is needed in order to obtain $|\varphi(0)|=1$.
Our purpose is to show that such a function $\varphi$ satisfies equalities in \eqref{chainineq}. First of all, we note that
$$
(R^2 - q \bar a_k )^{-\star}\star (q-a_k) R=(R^2 - p \bar a_k)^{-1} (p-a_k) R
$$
where $p=s(q)^{-1}q s(q)$, $s(q)= R^2 - q a_k$ and $s(q)\not=0$ for $q\in B(0;R)$. Let us study when the inequality
$$
|R^2 -\bar a_k p|^{-1} |p-a_k| R  \leq  1
$$
is satisfied.
The inequality is equivalent to
$$|p-a_k|^2 R^2\leq |R^2 -\bar a_k p|^2$$ but also to
$$(p-a_k)(\bar p-\bar a_k)R^2 \leq (R^2 -\bar a_k p)(R^2 -\bar p  a_k )$$
that is, after some calculations
$$
(R^2 -|a_k|^2)(|p|^2 - R^2)\leq 0.
$$
Since $(R^2 -|a_k|^2)>0$, the above inequality is satisfied if and only if $|p|^2 - R^2\leq 0$ that is if and only if $|p|=|q|\leq er$. We conclude that each factor takes the ball $B(0;R)$ to $B(0;1)$ and, in particular, it has modulus $1$ on the sphere $|q|=R$. As a consequence, $|\varphi(z)|=e^n$ for all $q$ such that $|q|=R$. All the roots $a_k$ of $\varphi$ have modulus $r$ and $\varphi(0)=1$. If there are pairs $a_k, a_{k+1}$ such that $a_{k+1}=\bar a_k$ then the product of the two corresponding factor gives a spherical zero. If there are $r$ such pairs, we will have $r$ spheres of zeros and the product of the corresponding factors, which commute among themselves and with the other factors, form the function $g$ such that $\varphi(q)=g(q)h(q)$. The remaining factors are $n-2r$ and will give $2(n-2r)$ spherical zeros of $h^s$. Since the factor $e^{n+i\alpha}$ can be suitably split in $e^{2r+\frac i2\alpha}$ in front of $g$ and  $e^{(n-2r)+\frac i2\alpha}$ in front of $h$, it follows by its construction, that the function $\varphi$ satisfies the equalities in \eqref{chainineq}.
\end{proof}

\section{Carath\'eodory theorem}
The real part of a slice regular function does not play the important role that real parts play for holomorphic functions. For example, what plays the role of the real part in the Schwarz formula in the complex case is the function $\alpha(x,y)$ if $f(x+Iy)=\alpha(x,y)+I\beta(x,y)$
 is a slice regular function in $B(0;R)$, $R>0$. Thus, also the analog of Carath\'eodory theorem is stated in terms of the function $\alpha(x,y)$. Let us set \index{$A_f$}
 $$A_f(r)=\max_{|x+Iy|=r}|\alpha(x,y)|,\qquad {\rm for}\qquad r<R.
 $$
\begin{lemma}
Let $f\in\mathcal{R}(B(0;R))$, $R>0$.
Then $A_f(r)\leq M_f(r)$.
\end{lemma}
\begin{proof}
Since $\alpha(x,y)=\frac 12 (f(x+Iy)+f(x-Iy))$ for any $x+Iy\in B(0;R)$ we have
\[
\begin{split}
A_f(r)&=\max_{|x+Iy|=r}|\alpha(x,y)|\\
&=\frac 12 \max_{|x+Iy|=r}|f(x+Iy)+f(x-Iy)|\\
&\leq \frac 12 (\max_{|x+Iy|=r}|f(x+Iy)|+ \max_{|x+Iy|=r} |f(x-Iy)|)\\
& = M_f(r)
\end{split}
\]
for $r<R$.
\end{proof}
We can now prove a Carath\'eodory type inequality, which is a sort of converse of the previous inequality:
\begin{theorem}[Carath\'eodory inequality]\index{Carath\'eodory!inequality}
Let $f\in\mathcal R(B(0;R))$, $f(q)=f(x+Iy)=\alpha(x,y)+I\beta(x,y)$ and let $0<r<R$. Suppose, for simplicity,  that $f(0)\in\mathbb{R}$. Then
$$
M_f(r)\leq \frac{2r}{R-r} (A_f(R)- \alpha (0)) +|\alpha (0)|.
$$
\end{theorem}
\begin{proof}
Let us write the function $f(q)$ using the Schwarz formula in Theorem \ref{Schwarz formula}:
\begin{equation}\label{schw2}
f(q)
=\frac{1}{2\pi}\int_{-\pi}^\pi
(Re^{It}-q)^{-\star }\star (Re^{It}+q)\,
\alpha(Re^{I\, t})dt.
\end{equation}
If we fix two units $I,J\in\mathbb S$ with $I$ orthogonal to $J$, we can decompose the  function $\alpha$ as $\alpha=\alpha_0+\alpha_1 I+\alpha_2 J+\alpha_3 IJ$ where $\alpha_\ell$ are harmonic functions, and since
$$
\alpha_\ell(0)=\frac{1}{2\pi}\int_{-\pi}^\pi \alpha_\ell (R e^{It})\, dt, \quad \ell=0,\ldots, 3
$$
we have that
$$
\alpha (0)-\frac{1}{2\pi}\int_{-\pi}^\pi \alpha (R e^{It})\, dt=0.
$$
We can add this quantity to the right hand side of \eqref{schw2}, obtaining
\begin{equation}\label{schw3}
\begin{split}
f(q)
&=\frac{1}{2\pi}\int_{-\pi}^\pi
(Re^{It}-q)^{-\star }\star (Re^{It}+q)\,
\alpha(Re^{I\, t})\, dt -\frac{1}{2\pi}\int_{-\pi}^\pi \alpha (R e^{It})\, dt +\alpha (0)\\
&=\frac{1}{2\pi}\int_{-\pi}^\pi
(Re^{It}-q)^{-\star }\star (Re^{It}+q-Re^{It}+q)\,
\alpha(Re^{I\, t})\, dt  +\alpha (0)\\
&=\frac{1}{\pi}\int_{-\pi}^\pi
(Re^{It}-q)^{-\star }\star  q\,
\alpha(Re^{I\, t})\, dt  +\alpha (0).\\
\end{split}
\end{equation}
Assume $f\equiv 1$. Since $f(x+Iy)=1$ for all $x+Iy\in B(0;R)$ we have that $\alpha(x,y)+I \beta(x,y)\equiv 1$ for all $x+Iy\in\ B(0;R)$. By assigning to $I$, e. g., the values $i,j,k, (i+j)/\sqrt 2, (i+k)/\sqrt 2, (j+k)/\sqrt 2$ we deduce that $\beta(x,y)\equiv 0$ and $\alpha(x,y)\equiv 1$.
Thus we have
\begin{equation}\label{schw4}
\frac{1}{\pi}\int_{-\pi}^\pi
(Re^{It}-q)^{-\star }\star  q\, dt =0.
\end{equation}
From \eqref{schw3} and \eqref{schw4} we obtain
$$
-f(q)=\frac{1}{\pi}\int_{-\pi}^\pi
(Re^{It}-q)^{-\star }\star  q\,
(A_f(r)-\alpha(Re^{I\, t}))\, dt  -\alpha (0)
$$
from which we deduce, since $A_f(r)-\alpha(Re^{I\, t})$ is nonnegative,
$$
|f(q)| \leq \frac{2r}{R-r}(A_f(R)- \alpha(0)) + |\alpha (0)|
$$
Taking the maximum on $|q|=r$ of both sides we obtain
$$
M_f(r) \leq \frac{2r}{R-r}(A_f(R)-\alpha(0)) + |\alpha (0)|,
$$
and this concludes the proof.
\end{proof}
In the complex case, Carath\'eodory theorem allows to estimate from below the modulus of a holomorphic function without zeros in a disc centered at the origin.
The proof makes use of the composition of the logarithm with the holomorphic function, so this technique cannot be immediately used in the quaternionic case. However, we can still prove a bound from below on the modulus of a slice regular function.
\\
Let us recall the result in the complex case (see \cite{levin} Theorem 9, p. 19):
\begin{theorem} Let $f$ be a function holomorphic in $|z|\leq R$ which has no zeros in that disk and such that $f(0)=1$. Then for any $z$ such that $|z|\leq r<R$, the modulus of $f$ satisfies
\begin{equation}
\log |f(z)| \geq - \frac{2r}{R-r}\log M_f(R).
\end{equation}
\end{theorem}
The result in the quaternionic case is similar, but the value of the constant is different:
\begin{theorem}\label{logslice} Let $f$ be a function slice regular in $|q|\leq R$ which has no zeros in that ball and such that $f(0)=1$. Then for any $q$ such that $|q|\leq r<R$, the modulus of $f$ satisfies
\begin{equation}
\log |f(q)| \geq - \frac{3r+R}{R-r}\log M_f(R).
\end{equation}
\end{theorem}
\begin{proof}
Let us consider the symmetrization $f^s$ of $f$. This is a slice regular function which is intrinsic so we can use the result in the complex case, see Proposition \ref{Mfintr} and its proof, and for any $I\in\mathbb S$ we have
\begin{equation}\label{car1}
\log |f^s(q)| =\log |f^s(x+Iy)| \geq - \frac{2r}{R-r} \log M_{f_I^s}(R)=- \frac{2r}{R-r} \log M_{f^s}(R).
\end{equation}
From \eqref{supstar} we have $M_{f^s}(R)\leq M_f(R)^2$, so we deduce $\log M_{f^s}(R)\leq 2 \log M_f(R)$ or, equivalently
\begin{equation}\label{car2}
- \log M_{f^s}(R)\geq - 2 \log M_f(R).
\end{equation}
It follows that
\begin{equation}\label{car3}
\begin{split}
\log | f^s(q)| &=\log (|f(q)|\, |f^c(f(q)^{-1}q f(q))| ) \\
&= \log |f(q)| + \log |f^c(f(q)^{-1}q f(q))| \\
& \leq \log |f(q)| + \max_{|q|\leq R} \log |f^c(f(q)^{-1}q f(q))|\\
& \leq \log |f(q)| + \log (\max_{|q|\leq R}  |f^c(f(q)^{-1}q f(q))|)\\
& \leq \log |f(q)| + \log (M_f(R)).
\end{split}
\end{equation}
From \eqref{car1}, \eqref{car2}, \eqref{car3} we have
$$
\log |f(q)| + \log (M_f(R)) \geq \log | f^s(q)| \geq - \frac{2r}{R-r}\log M_{f^s}(R)
$$
which gives
$$
\log |f(q)| \geq - \frac{4r}{R-r}\log M_{f}(R) -\log (M_f(R)) = - \frac{3r+R}{R-r}\log M_{f}(R).
$$
\end{proof}
In order to provide a lower bound for the modulus of a function slice regular in a ball centered at the origin, we need the following result:
\begin{proposition}\label{f-star}
Let $a_1,\ldots, a_t\in\mathbb B$.
Given the function
$$
f(q)=\prod^{\star t}_{k=1}(1-q \overline{a_k})^{-\star}\star (q-a_k)
$$
it is possible to write $f^{-\star}(q)$, where it is defined, in the form
\begin{equation}\label{starch5}
f^{-\star}(q)=\prod^{\star 1}_{k=t}(\hat{a_k}-q )\overline{a_k}\star \prod^{\star 1}_{k=t} (q-\tilde{a}_k)^{-\star}
\end{equation}
where $\hat{a_k}\in [a_k^{-1}]$, $\tilde{a}_k\in [a_k]$, and in particular $\tilde{a}_1=a_1$, $\hat{a}_t=\bar{a}_t^{-1}$.
\end{proposition}
\begin{proof}
First of all, note that
$$
(q-a_k)^{-\star}\star (1-q \overline{a_k})= (1-q \overline{a_k})\star (q-a_k)^{-\star}
$$
and so
$$
f^{-\star}(q)=\prod^{\star 1}_{k=t} (1-q \overline{a_k})\star (q-a_k)^{-\star}.
$$
We prove \eqref{starch5} by induction. Assume that $t=2$ and consider
$$
f^{-\star}(q)=(1-q \overline{a_2})\star (q-a_2)^{-\star}\star (1-q \overline{a_1})\star (q-a_1)^{-\star}.
$$
Since the spheres $[a_2]$ and $[\overline{a_1}^{\,-1}]=[{a_1}^{-1}]$ are different, by Theorem \ref{dfactor} we have, for suitable $\hat a_1\in [a_1^{-1}]$, $a_2'\in [a_2]$:
\begin{equation}\label{commuta}
\begin{split}
(q-a_2)^{-\star}\star (1-q \overline{a_1}) &= (q-a_2)^{-\star} \star (\overline{a_1}^{-1}-q )\overline{a_1}\\
 &= (q^2- 2 {\rm Re}(a_2) q +|a_2|^2)^{-1}(q-a_2)\star (\overline{a_1}^{-1}-q )\overline{a_1}\\
&= (q^2- 2 {\rm Re}(a_2) q +|a_2|^2)^{-1} (\hat{a}_1-q )\star (q-a'_2)\overline{a_1}\\
&= (q^2- 2 {\rm Re}(a_2) q +|a_2|^2)^{-1} (\hat{a}_1-q )\overline{a_1}\star (q-\overline{a_1}^{-1}a'_2\overline{a_1})\\
&= (q^2- 2 {\rm Re}(a_2) q +|a_2|^2)^{-1} (\hat{a}_1-q )\overline{a_1}\star (q-\hat{a}_2)\\
&=  (\hat{a}_1-q )\overline{a_1}\star (q^2- 2 {\rm Re}(a_2) q +|a_2|^2)^{-1} (q-\hat{a}_2)\\
&=  (\hat{a}_1-q )\overline{a_1}\star  (q-\hat{a}_2)^{-\star},
\end{split}
\end{equation}
where $\hat a_2=\overline{a}_1^{-1} a_2' \overline{a}_1\in [a_2]$,
so $f^{-\star}(q)$ rewrites as
$$
f^{-\star}(q)=(\overline{a_2}^{-1}-q )\overline{a_2}\star (\hat{a}_1-q )\overline{a_1}\star  (q-\hat{a}_2)^{-\star} \star (q-a_1)^{-\star}
$$
and we have the assertion.\\
Let us assume that the statement is valid for product with $n$ factors and let us show that that it holds for $n+1$ factors.
\begin{equation}
\begin{split}
f^{-\star}(q)&=\prod^{\star 1}_{k=n+1} (1-q \overline{a_k})\star (q-a_k)^{-\star}\\
&=\prod^{\star 1}_{k=n+1} (1-q \overline{a_k})\star (q-a_k)^{-\star}\\
 &=(1-q \overline{a_{n+1}})\star (q-a_{n+1})^{-\star}\star\prod^{\star 1}_{k=n} (1-q \overline{a_k})\star (q-a_k)^{-\star}\\
&=(1-q \overline{a_{n+1}})\star (q-a_{n+1})^{-\star}\star \prod^{\star 1}_{k=n}(\hat{a_k}-q )\overline{a_k}\star \prod^{\star 1}_{k=n} (q-\tilde{a}_k).
\end{split}
\end{equation}
Now we use iteratively \eqref{commuta} to rewrite first the product $(q-a_{n+1})^{-\star}(\hat{a_n}-q )\overline{a_n}$ as $(\hat{a'_n}-q )\overline{a_n}\star (q-a'_{n+1})^{-\star}$ where $\hat{a'_n}\in[{a_n}^{-1}]$ and $a'_{n+1}\in[a_{n+1}]$. Then, with the same technique, we rewrite $(q-a'_{n+1})^{-\star}\star (\hat{a}_{n-1} -q )\overline{a_{n-1}}$ and so on.
Since the leftmost factor is $(1-q \overline{a_{n+1}})=(\overline{a_{n+1}}^{-1}-q )\overline{a_{n+1}}$ and the rightmost factor is $(q-\tilde a_1)$, we have the statement for $n+1$ factors and the thesis follows.
\end{proof}
\begin{corollary}\label{corf-star}
Let $a_1,\ldots, a_t\in B(0,2R)$.
Given the function
$$
f(q)=(2R)^t\prod^{\star t}_{k=1}((2R)^2-q \overline{a_k})^{-\star}\star (q-a_k)
$$
it is possible to write $f^{-\star}(q)$, where it is defined, in the form
$$
f^{-\star}(q)=(2R)^{-t}\prod^{\star 1}_{k=t}((2R)^2\hat{a_k}-q )\overline{a_k}\star \prod^{\star 1}_{k=t} (q-\tilde{a}_k)^{-\star},
$$
where $\hat{a_k}\in [a_k^{-1}]$, $\tilde{a}_k\in [a_k]$, and in particular $\tilde{a}_1=a_1$, $\hat{a}_t=\bar{a}_t^{-1}$ and
$$
f(q)=(2R)^{t}\prod^{\star 1}_{k=t}((2R)^2\hat{a_k}-q )\overline{a_k}\star \prod^{\star 1}_{k=t} (q-\tilde{a}_k)^{-\star}.
$$
\end{corollary}
\begin{proof}
The proof follows
from the proof of Proposition \ref{f-star} with minor modifications.
\end{proof}
\begin{theorem}\label{cartansr}
Let $f(q)$ be a function slice regular in the ball with center at the origin and radius $2eR$, where $R>0$. Let $f$ be  such that $f(0)=1$ and let $f=gh$ where $g$ has at most spherical zeros only and $h$ has at most isolated zeros only. Let $\eta$ be an arbitrary real number belonging to $(0,3e/2]$. For any $q\in B(0,R)$ but outside a family of balls whose radii have sum not exceeding $4\eta R$ we have
$$
\log|f(q)| > - \left(5+\log\left(\frac{3e}{2\eta}\right)\right) \log (M_g(2eR)M_h(2eR)).
$$
\end{theorem}
\begin{proof}
Let us assume that $f$ has, in the given ball $B(0,2eR)$, $t$ isolated zeros $\alpha_1,\ldots, \alpha_t$ and $p$ spherical zeros $[\beta_1],\ldots, [\beta_p]$ (note that some elements may be repeated). Set $n=t+2p$ and consider the function
$$
\phi(q)=(-1)^n(2R)^{2n} \prod^{\star t}_{k=1}((2R)^2-q \overline{a_k})^{-\star}\star (q-a_k)
$$
$$ \cdot\prod_{\ell=1}^p ((2R)^4-2(2R)^2 {\rm Re}(\beta_\ell) q +q^2|\beta_\ell |^2)^{-1}
(q^2 -2 {\rm Re}(\beta_\ell) q +|\beta_\ell|^2)
$$
$$
\cdot (a_1\cdots a_t|\beta_1|^2\cdots |\beta_p|^2)^{-1}.
$$
In the sequel, it will be useful to set
$$
\phi_1(q)= \prod^{\star t}_{k=1}((2R)^2-q \overline{a_k})^{-\star}\star (q-a_k)
$$
and
$$
\phi_2(q)= \prod_{\ell=1}^p ((2R)^4-2(2R)^2 {\rm Re}(\beta_\ell) q +q^2|\beta_\ell |^2)^{-1}
(q^2 -2 {\rm Re}(\beta_\ell) q +|\beta_\ell|^2).
$$
In the $\star$-product above, the elements $a_k\in [\alpha_k]$, $k=1,\ldots, t$, have to be suitably chosen.
More precisely, let us apply Theorem \ref{R-fattorizzazslice} in order to write $f(q)$ in the form
\begin{equation}\label{fsplit}
f(q)=\prod_{k=1}^{\star t} (q - a'_k) \prod_{\ell=1}^p (q^2 -2 {\rm Re}(\beta_\ell) q +|\beta_\ell|^2) \star g(q)
\end{equation}
where $g(q)\not =0$ and some elements $a'_k$, $\beta_\ell$ may be repeated. Note that $a'_1=a_1$ while $a'_k\in [\alpha_k]$ are chosen in order to obtain the assigned isolated zeros. Then in order to construct the function
$\phi(q)$ we choose $a_k$ such that the elements $\tilde a_k$ constructed in the proof of Theorem \ref{f-star} are such that $\tilde a_k =a_k$.\\
We now evaluate $\phi(0)$: by Theorem \ref{pointwise}, since $0$ is real, it suffices to take the product of the factors evaluated at the origin and so
\[
\begin{split}
\phi(0)&=(-1)^n(2R)^{2n} \prod^{t}_{k=1}(2R)^{-2}(-a_k)
\prod_{\ell=1}^p (2R)^{-4}|\beta_\ell|^2 (a_1\cdots a_t|\beta_1|^2\cdots |\beta_p|^2)^{-1}\\
&= (a_1\cdots a_t|\beta_1|^2\cdots |\beta_p|^2)(a_1\cdots a_t|\beta_1|^2\cdots |\beta_p|^2)^{-1}\\
&=1.
\end{split}
\]
Let us now evaluate $|\phi(2Re^{I\theta})|$ for $I\in\mathbb S$. By Theorem \ref{pointwise}, the evaluation of the $\star$-product
$$
\prod^{\star t}_{k=1}((2R)^2-q \overline{a_k})^{-\star}\star (q-a_k)
$$
at the point $2Re^{I\theta}$ gives
$$
\prod^{t}_{k=1}((2R)^2- 2Re^{I_k \theta} \overline{a_k})^{-1} (2R e^{I_k \theta}-a_k)
$$
where $I_1=I$ while $I_2, \ldots, I_t$ have to be chosen according to Theorem \ref{pointwise}.
Then we have
\[
\begin{split}
&|\prod^{t}_{k=1}((2R)^2- 2Re^{I_k \theta} \overline{a_k})^{-1} (2R e^{I_k \theta}-a_k)|\\
&=(2R)^{-t}\prod^{t}_{k=1}|((2R)- e^{I_k \theta} \overline{a_k})^{-1}|\, | 2R e^{I_k \theta}-a_k |\\
&=(2R)^{-t}\prod^{t}_{k=1}|(2R- e^{I_k \theta} \overline{a_k})^{-1}|\, | e^{I_k \theta}-a_k|\\
&=(2R)^{-t}\prod^{t}_{k=1}|(2Re^{-I_k \theta} -  \overline{a_k})^{-1}|\, | (2R e^{I_k \theta}-a_k)|\\
&=(2R)^{-t}.
\end{split}
\]
Similarly,
\[
\begin{split}
&|\prod_{\ell=1}^p ((2R)^4-2(2R)^3 {\rm Re}(\beta_\ell) e^{I_t \theta} +(2R e^{I_t \theta})^2|\beta_\ell |^2)^{-1}
\\
&\times((2R e^{I_t \theta})^2 -2 {\rm Re}(\beta_\ell) (2R e^{I_t \theta}) +|\beta_\ell|^2)|
\\
&\times (2R)^{-2p}\prod_{\ell=1}^p |(2R)^2- 2(2R) {\rm Re}(\beta_\ell) e^{I_k \theta} + e^{2I_k \theta}|\beta_\ell |^2|^{-1}\\
&\times
|((2R)^2 e^{2I_k \theta} -2 (2R) {\rm Re}(\beta_\ell)  e^{I_k \theta} +|\beta_\ell|^2))|
\\
&=(2R)^{-2p}.
\end{split}
\]
Thus we have
$$
|\phi(2R e^{I_t \theta})|=(2R)^n(|a_1|\cdots|a_t|\,|\beta_1|^2\cdots |\beta_p|^2|)^{-1}.
$$
We now consider the function
$$
\psi(q)=\phi^{-\star}(q)\star f(q)
$$
which, by definition, is slice regular. Note also that, by Corollary \ref{corf-star} and formula \eqref{fsplit} we have
\[
\begin{split}
\psi(q)&=\phi^{-\star}(q)\star f(q)\\
&= (2R)^{-t}\prod^{\star 1}_{k=t}((2R)^2\hat{a_k}-q )\overline{a_k}\star \prod^{\star 1}_{k=t} (q-\tilde{a}_k)^{-\star}\\
&\star \prod_{k=1}^{\star t} (q - a'_k) \prod_{\ell=1}^p (q^2 -2 {\rm Re}(\beta_\ell) q +|\beta_\ell|^2) \star g(q)\\
&= (2R)^{-t}\prod^{\star 1}_{k=t}((2R)^2\hat{a_k}-q )\overline{a_k} \prod_{\ell=1}^p (q^2 -2 {\rm Re}(\beta_\ell) q +|\beta_\ell|^2) \star g(q)\\
\end{split}
\]
and so $\phi$ does not have zeros in the ball $|q|\leq 2R$, in fact  the factors
$$
(2R)^2-q \overline{a_t},
\quad {\rm and}\quad
(2R)^4-2(2R)^2 {\rm Re}(\beta_\ell) q +q^2|\beta_\ell |^2,
$$
have roots which are clearly outside that ball and $g(q)$ does not vanish.
Applying Theorem \ref{logslice}, we have that for any $q$ such that $|q|\leq R<2R$, the modulus of $f$ satisfies
\begin{equation}
\begin{split}
\log |\psi (q)| &\geq - 5 \log M_{\psi}(2R)\\
&=- 5 \log M_{f}(2R) + 5 \log M_{\phi}(2R)\\
&= - 5 \log M_{f}(2R) + 5 \log |\phi(2R e^{I_t \theta})|\\
&= - 5 \log M_{f}(2R)\\
&\geq - 5 \log M_{f}(2eR)\\
&\geq - 5 \log (M_{g}(2eR)M_{h}(2eR))
\end{split}
\end{equation}
where we used the fact that $|\phi(2R e^{I_t \theta})|>1$.
We now need a lower bound on
$$\phi(q)=(-1)^n(2R)^{2n}\phi_1(q) \phi_2(q)(a_1\cdots a_t|\beta_1|^2\cdot |\beta_p|^2)^{-1}.$$
 To this end, it is useful to write $\phi_1(q)$, using Corollary
\ref{corf-star}, in the form
$$
\phi(q)=(2R)^{-t} \prod^{\star t}_{k=1} (q-\tilde{a}_k) \star \prod^{\star 1}_{k=t}(((2R)^2\hat{a_k}-q )\overline{a_k})^{-\star}.
$$
Then we have
$$
|\prod^{\star 1}_{k=t}(((2R)^2\hat{a_k}-q )\overline{a_k})^{-\star}|=
\prod^{1}_{k=t}|((2R)^2\hat{a_k}-q_k )\overline{a_k})^{-1}|,
$$
where $q_k$ are suitable elements in $[q]$ (computed applying Theorem \ref{pointwise}), and it is immediate that
$$
\prod^{1}_{k=t}|((2R)^2\hat{a_k}-q_k )\overline{a_k})|\leq (6R^2)^t,
$$
so we deduce
$$
|\prod^{\star 1}_{k=t}(((2R)^2\hat{a_k}-q )\overline{a_k})^{-\star}|\geq \frac{1}{(6R^2)^t}.
$$
Applying Cartan Theorem \ref{cartan} we have that outside some exceptional balls
$$
\prod^{\star t}_{k=1} (q-\tilde{a}_k) \geq \left(\frac{2\eta R}{n}\right)^t.
$$
A lower bound on $\phi_2(q)$ can be provided by observing that $\phi_2(q)$ can be written as
$$
\phi_2(q)= \prod_{\ell=1}^p ((2R)^4-2(2R)^2 {\rm Re}(\beta_\ell) q +q^2|\beta_\ell |^2)^{-1}
\prod_{\ell=1}^p (q^2 -2 {\rm Re}(\beta_\ell) q +|\beta_\ell|^2).
$$
Computations similar to those done in the case of $\phi_1(q)$ show that
$$
|\prod_{\ell=1}^p ((2R)^4-2(2R)^2 {\rm Re}(\beta_\ell) q +q^2|\beta_\ell |^2)^{-1}|\geq (6R^2)^{2p}
$$
and
\[
\begin{split}
|\prod_{\ell=1}^p (q^2 -2 {\rm Re}(\beta_\ell) q +|\beta_\ell|^2)&= \prod_{\ell=1}^p (q -\beta_\ell) \star (q -\overline{\beta_\ell})|\\
&\geq \left(\frac{2\eta R}{n}\right)^{2p}.
\end{split}
\]
Therefore outside the exceptional balls we have
\[
\begin{split}
 |\phi(q)|&\geq \frac{(2R)^n}{|a_1\cdots a_t|\, |\beta_1|^2\cdots |\beta_p|^2}  \left(\frac{2\eta R}{n}\right)^n \frac{1}{(6R^2)^n}\\
 & \geq\left(\frac{2\eta}{3e}\right)^n.
\end{split}
\]
Moreover, Proposition \ref{prop43} gives
 $$
 n= n_f(2R)\leq \log ( M_g(2eR) M_h(2eR)),
 $$
 and so, outside the exceptional balls
 $$
 \log |\phi(q)| \geq \log\left(\frac{2\eta}{3e}\right) \log ( M_g(2eR) M_h(2eR)).
 $$
 Since
 \[
 \begin{split}
 \log |f(q)| &= \log |\psi(q)| +\log |\phi(q)| \\
 &\geq - \left(5+\log\left(\frac{3e}{2\eta}\right)\right) \log M_{f}(2eR)
 \end{split}
 \]
the statement follows.
\end{proof}

\section{Growth of the $\star$-product of entire slice regular functions}

Given two entire slice regular functions, it is a natural question to ask if we can characterize order and type of their $\star$-product if we know order and type of the factors. The answer is positive and is contained in the next result:
\begin{theorem}\label{prodgrowth}
Let $f$ and $g$ be two entire functions of order and type $\rho_f$, $\sigma_f$ and $\rho_g$, $\sigma_g$, respectively, and let $\rho_{f\star g}$, $\sigma_{f\star g}$ be order and type of the product $f\star g$. Then:
\begin{enumerate}
\item[(1)] If $\rho_f\not=\rho_g$ then $\rho_{f\star g}=\max(\rho_f, \rho_g)$ and $\sigma_{f\star g}$ equals the type of the function with larger order.
\item[(2)] If $\rho_f=\rho_g$ , one function has normal type $\sigma$ and the other has minimal type, then $\rho_{f\star g}=\rho_f =\rho_g$ and $\sigma_{f\star g}=\sigma$.
\item[(3)] If $\rho_f=\rho_g$ , one function has maximal type and the other has at most normal type, then $\rho_{f\star g}=\rho_f =\rho_g$ and $\sigma_{f\star g}$  has maximal type.
\end{enumerate}
\end{theorem}
\begin{proof}
We show that $(2)$ holds, the other two statements follow with similar arguments.\\
Recalling \eqref{ineqMfr}, denoting by $\rho$ the order of $f$ and $g$, and assuming that $f$ has normal type $\sigma_f$, we have
$$
M_f(r)< e^{(\sigma_f +\varepsilon/2)r^{\rho}}, $$
and
$$
M_g(r)< e^{(\varepsilon/2)r^{\rho}},
$$
and so
\[
\begin{split}
M_{f\star g}(r)&=\max_{|q|=r}(|(f\star g)(q)|\\
&\leq \max_{|q|=r}|f(q)| \max_{|q|=r}|g(q)|
\\
&=M_f(r)M_g(r)\\
&<
e^{(\sigma_f +\varepsilon)r^{\rho}}.
\end{split}
\]
We now need a lower bound for $M_{f\star g}(r)$.
First of all, we can find a positive number $R_1$, large enough, such that for two given positive numbers $\varepsilon$ and $\delta$ the inequality
\begin{equation}\label{Mfine}
M_f(R_1)>  e^{(\sigma_f - \varepsilon/2)R_1^{\rho}}
\end{equation}
holds,
and for all $R\geq R_1$
$$
M_g(R)<  e^{\delta R^{\rho}}.
$$
We are now in need to use Theorem \ref{cartansr}.
To this end we assume, with no loss of generality, that $g(0)=1$.
In fact, if this is not the case, we can $\star$-multiply $g(q)$ on the left by the factor $q^{-m}c$ for suitable $c\in\mathbb H$ and $m\in\mathbb N$, without changing its order and type.
We now assume $0< \delta <1$, $R=R_1(1-\delta)^{-1}$ and take $\eta=\delta/8$. Inside the ball centered at the origin and with radius $R$, we exclude the balls described in Theorem \ref{cartansr}. Note that the sum of their diameters is less than $\delta R$.
So there exists $r_1\in (R_1,R)$ such that the ball centered at the origin and with radius $r_1$ does not meet any of these excluded balls.
By Theorem \ref{cartansr} on this ball of radius $r_1$ we have
\begin{equation}\label{inecartan}
\log|g(q)| > - \left(5+\log\left(\frac{12 e}{\delta}\right)\right) \log M_g(2eR),
\end{equation}
moreover, since $R_1 <r_1 < R_1(1-\delta)^{-1}$ and \eqref{Mfine} is in force, we have
$$
M_f(r_1) > M_f(R_1)>  e^{(\sigma - \varepsilon/2)R_1^{\rho}}>  e^{(\sigma - \varepsilon/2)(1-\delta)^{\rho}r_1^{\rho}}.
$$
This latter inequality and \eqref{inecartan} yield
$$
\log M_{f\star g}(r_1) > (\sigma - \varepsilon/2)(1-\delta)^{\rho}r_1^{\rho} - \left(5+\log\left(\frac{12 e}{\delta}\right)\right) \log M_g(2eR).
$$
Since
$$
\log M_g(2eR) < \delta (2eR)^{\rho}< \delta (1-\delta)^{-\rho}(2e)^{\rho} r_1^{\rho},
$$
we finally  have
$$
\log M_{f\star g}(r_1) > \left[(\sigma - \varepsilon/2)(1-\delta)^{\rho} - \left(5+\log\left(\frac{12 e}{\delta}\right)\right)\delta (1-\delta)^{-\rho}(2e)^{\rho} \right] r_1^{\rho} .
$$
For any given $\varepsilon$ we can choose $\delta$ such that
$$
\left[(\sigma - \varepsilon/2)(1-\delta)^{\rho} - \left(5+\log\left(\frac{12 e}{\delta}\right)\right)\delta (1-\delta)^{-\rho}(2e)^{\rho} \right]\geq \sigma-\varepsilon
$$
and so
$$
M_{f\star g}(r_1) > e^{(\sigma-\varepsilon)r_1^{\rho}}.
$$
This ends the proof of point $(2)$.
\end{proof}
Recalling that $f^s=f\star f^c$ and Corollary \ref{orderfc}, we immediately have:
\begin{corollary} \label{orderfs}
Given an entire slice regular function $f$, the order and type of $f^s$ coincide with the order and type of $f$.
\end{corollary}
As a consequence of Theorem \ref{prodgrowth} we can now provide a relation between the convergence exponent of the sequence of zeros of an entire regular function $f$  and its order. The bound differs from the one in the complex case, unless we consider intrinsic functions. In fact we have:
\begin{theorem}
 The convergence exponent of the sequence of zeros of an entire regular function does not exceed twice its order.
\end{theorem}
\begin{proof}
Suppose that $f(0)=1$ and let us use Proposition \ref{prop43} and the notation therein. Then
\[
\begin{split}
\rho_1 & = {\underset{r\to\infty}{\overline{\lim}}}\frac{\log n_f(r)}{\log r}\\
&\leq {\underset{r\to\infty}{\overline{\lim}}}\frac{\log(\log M_g(er))}{\log er}+\frac{\log(\log M_h(er))}{\log er} =\rho_g+\rho_h,
\end{split}
\]
where $\rho_g$ and $\rho_f$ denote the order of $g$ and $h$, respectively.
 By Theorem \ref{prodgrowth} the order $\rho$ of $f$ equals $\max(\rho_g,\rho_h)$ so
 $$
 \rho_1 \leq 2\rho.
 $$
  If $f(0)\not= 1$, it is sufficient to normalize it or, if $f$ has a zero of order $k$ at $0$, it is sufficient to consider the function $\tilde f(q)=k! q^{-k}f(q)(\partial_x f(0))^{-1}$. This function $\tilde f(q)$ has the same order as $f$ and the assertion is true for this function.
\end{proof}
In two particular cases, one of which is the case of intrinsic functions, we have the bound which holds in the complex case:
\begin{theorem} \ \
 \begin{enumerate}
  \item[(1)] The convergence exponent of the sequence of zeros of an entire, intrinsic, slice regular function does not exceed its order.
  \item[(2)] The convergence exponent of the sequence of zeros of an entire slice regular function having only isolated zeros does not exceed its order.
\end{enumerate}
\end{theorem}
\begin{proof}
If $f$ is intrinsic then it does not have nonreal isolated zeros (see Lemma \ref{zeriseriereale}) thus, using the notation of the previous theorem,
$f=g$ and $h$ can be set equal $1$, so that
\[
\begin{split}
\rho_1 &= {\underset{r\to\infty}{\overline{\lim}}}\frac{\log n_f(r)}{\log r}\\
&\leq {\underset{r\to\infty}{\overline{\lim}}}\frac{\log(\log M_g(er))}{\log er} =\rho_g=\rho.
\end{split}
\]
If $f$ has only isolated zeros, then $f=h$ and $g$ can be set equal $1$, so the proof above shows the statement, changing $g$ with $h$.
\end{proof}

\section{Almost universal entire functions}

In this section we consider a kind of universal property for quaternionic entire functions. This property, in the quaternionic case, has been originally treated in \cite{GalSa1}.
The result in the complex plane was proved by Birkhoff \cite{birkhoff}. Later, Seidel and Walsh generalized the result to simply connected sets as in the following theorem,  see \cite{seidel walsh}:
\begin{theorem}\label{TmSW}
There exists an entire function $F(z)$ such that given an arbitrary function $f(z)$ analytic in a simply connected region $R\subseteq\mathbb C$, for suitably chosen $a_1, a_2,\ldots \in\mathbb C$ the relation
$$
\lim_{n\to\infty} F(z+a_n)=f(z)
$$
holds for $z\in R$ uniformly on any compact set in $R$.
\end{theorem}
 Recall that if $g$ is a polynomial with quaternionic coefficients (written on the right), then the standard composition $f\circ g$ is not, in general, slice regular. However, if $g$ is quaternionic intrinsic, and so when $g$ is a polynomial with real coefficients, then $f\circ g$ is slice regular.
 Thus, for our purposes we will select real numbers $a_n$. \\
 We also need the following
\begin{definition}
Let $\Omega\subseteq\mathbb H$ be an open set and let $\mathfrak{F}(\Omega)$ \index{$\mathfrak{F}(\Omega)$} be the set of all axially symmetric compact sets $K\subset \Omega$, such that  ${\mathbb{C}_{I}}\setminus (K\cap \mathbb{C}_{I})$ is connected  for some (and hence for all) $I\in \mathbb{S}$. \end{definition}

To prove our next result, we need
the Runge-type approximation result in  \cite[Theorem 4.11]{runge} stated below:
\begin{theorem}\label{Run}
 Let $K\in\mathfrak F(\mathbb H)$ and let $U$ be an axially symmetric open subset of $\mathbb H$ containing $K$. For any $f\in {\cal{R}}(U )$, there exists a sequence of  polynomials
$P_{n}(q)$, $n\in \mathbb{N}$, such that $P_{n}\to f$ uniformly on $K$, as $n\to \infty$.
\end{theorem}
We can now prove the following:
\begin{theorem}
There exists an entire slice regular function $F(q)$ such that given an arbitrary  function $f(q)$ slice regular in a region $U \in \mathfrak{R}(\mathbb H)$,  for suitably chosen $a_1, a_2,\ldots \in\mathbb R$ the relation
\begin{equation}\label{entire}
\lim_{n\to\infty} F(q+a_n)=f(q)
\end{equation}
holds for $q\in U $ uniformly on any compact set in $\mathfrak{F}(\mathbb H)$ contained in $U $.
\end{theorem}
\begin{proof}
We follow the proof in \cite{seidel walsh}, and we define the spheres
$$
S(4^n,2^n)=\{q\in\mathbb H\ ; \ | q-4^n|=2^n\}\quad {\rm for}\ n=1, 2, \ldots
$$
and
the spheres
$$
\Sigma(0, 4^n+2^n+1)=\{q\in\mathbb H\ ; \ | q|=4^n+2^n+1\}\quad {\rm for}\ n=1, 2, \ldots .
$$
It turns out that $S(4^n,2^n)$ and $S(4^m,2^m)$ are mutually exterior if $n\not=m$, moreover $\Sigma (0, 4^n+2^n+1)$ contains in its interior
$S(4^j,2^j)$ for $j=1,2,\ldots, n$ but it does not contain $S(4^j,2^j)$ nor its interior points if $j> n$.
The function $F$ can be constructed as limit of  polynomials  with rational quaternionic coefficients, i.e. with coefficients
of the form $a_0+ia_1+ja_2+ka_3$ and $a_i$ rational. Note that this assumption on the coefficients is not restrictive since it does not change properties like (uniform) convergence to a given function. Thus we consider all the polynomials  with rational quaternionic coefficients and we arrange them in a sequence $\{P_n\}$.

Let $\pi_1(q)$ be a polynomial such that $$|P_1(q-4)-\pi_1(q)|<\frac 12$$ for $q$ on $S(4,2)$ or inside $S(4,2)$. The polynomial exists since, for example, one may choose $\pi_1(q)=P_1(q-4)$. Let us define the function
$$
\tilde P_2(q)=\left\{ \begin{array}{c} P_2(q-4^2) \quad {\rm for}\ q\ {\rm on\ or\ inside\ \ } S(4^2,2^2)\\
  \pi_{1}(q) \quad {\rm for}\ q\ {\rm on\ or\ inside\ \ } \Sigma(0, 4+2+1).
  \end{array} \right.
$$
 Theorem \ref{Run} applied to $\tilde P_2$ shows that there exists a polynomial $\pi_2(q)$ such that
$$
| P_2(q-4^2)- \pi_2(q) |<1/2^2\qquad {\rm on\ or\ inside\ \ } S(4^2,2^2),
$$
$$
| \pi_1(q)- \pi_2(q) |<1/2^2,\qquad {\rm on\ or\ inside\ \ } \Sigma(0, 4+2+1).
$$
Thus we can define inductively
$$
\tilde P_n(q)=\left\{ \begin{array}{c} P_n(q-4^n) \quad {\rm for}\ q\ {\rm on\ or\ inside\ \ } S(4^n,2^n)\\
  \pi_{n-1}(q) \quad {\rm for}\ q\ {\rm on\ or\ inside\ \ } \Sigma(0, 4^{n-1}+2^{n-1}+1)
  \end{array} \right.
$$
and using Theorem \ref{Run}, we can find a polynomial $\pi_n(q)$ such that
$$
| P_n(q-4^n)- \pi_n(q) |<1/2^n,\qquad {\rm on\ or\ inside\ \ } S(4^n,2^n),\\
$$
$$
| \pi_{n-1}(q)- \pi_n(q) |<1/2^n \quad {\rm for}\ q\ {\rm on\ or\ inside\ \ } \Sigma(0, 4^{n-1}+2^{n-1}+1).
$$
The sequence $\{\pi_n(q)\}$ converges uniformly inside the balls $\Sigma$, and hence uniformly on every bounded set.
 So the function $F(q)=\lim_{n\to\infty} \pi_n(q)$ converges at any point in $\mathbb H$ and it is entire. We now show that it satisfies (\ref{entire}), see Proposition 4.1 in \cite{duality}.
Let $f$ be a slice regular function in a region $U \in {\mathfrak R}(\mathbb H)$. By Theorem \ref{Run}
there exist a subsequence  $\{P_{n_k}(q)\}$ of the polynomials considered above, such that
\begin{equation}\label{eqrunge}
\lim_{k\to\infty} P_{n_k}(q)=f(q)
 \end{equation}
in $U$, uniformly on every  $K\in\mathfrak{F}(\mathbb{H})$ contained in $U$.
We apply the above construction
 and consider $q\in S(4^n,2^n)$. We have
$$
F(q)=\pi_n(q) +(\pi_{n+1}(q)-\pi_n(q))+ (\pi_{n+2}(q)-\pi_{n+1}(q))+\ldots
$$
and
\[
\begin{split}
|F(q)-P_n(q-4^n)|&\leq | P_n(q-4^n)-\pi_n(q)| +|\pi_{n+1}(q)-\pi_n(q)|\\
&+ |\pi_{n+2}(q)-\pi_{n+1}(q)|+\ldots
\\
&<\frac{1}{2^n}+\frac{1}{2^{n+1}}+\ldots =\frac{1}{2^{n-1}}
\end{split}
\]
and so $$\lim_{n\to\infty} [F(q+4^n)-P_n(q)]=0$$ for $q$ in any bounded set. Thus, by using (\ref{eqrunge}), we deduce
$$
\lim_{k\to\infty}|F(q+4^{n_k})-f(q)|\leq \lim_{k\to\infty}|F(q+4^{n_k})- P_{n_k}(q)| + \lim_{k\to\infty} | P_{n_k}(q)-f(q)|=0
$$
which concludes the proof.
\end{proof}
\begin{remark}
The proof shows that the sequence $(a_{n})_{n\in \mathbb{N}}$ depends on the approximated function $f$.
\end{remark}
We now present another interesting result,  proved in the complex case by MacLane \cite{maclane}, showing that there exists a slice regular  function whose set of derivatives is dense in $\mathcal R (\mathbb H)$.
The proof below follows the proof given in \cite{am}.
\begin{theorem}
There exists an entire slice regular function $F$ such that the set $\{F^{(n)}\}_{n\in\mathbb N}$ is dense in $\mathcal R(\mathbb H)$.
\end{theorem}
\begin{proof}
We begin by defining a linear operator $I$ acting on the monomials $q^n$ as
$$I(q^n)=\frac{q^{n+1}}{n+1}$$ and then we extend by linearity to  the set of all   polynomials. We note that
$$
I^k(q^n)=\frac{q^{n+k}}{(n+k)\ldots (n+1)},
$$
and that, for $|q|\leq r$, we have the inequality
$$
|I^k(q^n)|\leq \frac{r^{n+k}}{(n+k)\ldots (n+1)}\leq \frac{r^{n+k}}{k!}.
$$
As a consequence, we have
 $$\max_{| q|\leq r} | I^k (q^n) | \to 0\qquad {\rm for}\ k\to\infty.
  $$
 For any $\delta >0$ and any $r>0$ there exists $k_0\in\mathbb N$ such that
$$\max_{| q|\leq r} | I^k (q^n) |  <\delta, \ \ \ {\rm for}\ k\geq k_0.
$$
Now observe that, given $f\in\mathcal R(\mathbb H)$,  $\varepsilon >0$ and $m\in\mathbb N$,  if $| f(q)|\leq\delta$ for $|q|\leq r$, then for $|q|\leq r/2$ the Cauchy estimates give
$$
| f^{(j)}(q)| \leq \frac{j! \max_{|q|\leq r} | f(q)|}{(r/2)^{j}}\leq j! 2^j/r^j \delta <\varepsilon ,
$$
for any $j=0,\ldots m$, if $\delta$ is sufficiently small. Thus, if  we consider a polynomial $P(q)$ and any $r>0$, $\varepsilon >0$ and $m\in\mathbb N$, there exists $k_0\in\mathbb N$ such that for $k \geq k_0$ we have
\begin{equation}\label{maximum}
\max_{| q| \leq r} | (I^k(P))^{(j)}| < \varepsilon ,
\end{equation}
for $j=0,\ldots ,m$.
\\
Let now $\{P_n\}_{n\in\mathbb N}$ be a sequence of polynomials  dense in $\mathcal R (\mathbb H)$ and let
$$
\sum_{j=1}^\infty I^{k_j}(P_j)(q),
$$
where the integers $k_j$ are constructed below. Note also that we denote $I^{k_j}(P_j)(q)$ by $Q_j(q)$. We set $k_1=0$, so $Q_1=P_1$; then we choose  $k_2>k_1 +\deg(P_1)$ and, since we set $Q_2= I^{k_2}(P_2)$, by (\ref{maximum}) we can select $k_2$ to be such that $$| Q_2(q)|< 1/ 2^2, \ \ \ \ {\rm for} \ \ | q| \leq 2.$$
Then, inductively, for $n\geq 3$, we choose $k_n> k_{n-1}+\deg(P_{n-1}),$ we set $Q_n=I^{k_n}(P_n)$ and we select $k_n$  to be such that
\begin{equation}\label{defQn}
\begin{split}
&| Q_n(q)|\leq \frac{1}{2^n}, \\
&| Q'_n(q)|\leq \frac{1}{2^n}, \\
&\ldots \\
&| Q_n^{(k_{n-1})}(q)|\leq \frac{1}{2^n}
\end{split}
\end{equation}
for $| q| \leq n$. The series $$\sum_{j=1}^\infty I^{k_j}(P_j)(q)=\sum_{j=1}^\infty Q_j(q)$$ converges uniformly on bounded subsets in $\mathbb H$ to a slice regular function $F(q)$,
in fact (\ref{defQn}) implies that:
$$
\max_{|q|\leq j} | Q_j(q)| \leq \frac{1}{2^j}.
$$
 We now show that the function $F$ is such that $\{F^{(n)}\}_{n\in\mathbb N}$ is dense in $\mathcal R(\mathbb H)$. To this end, let $f\in\mathcal R(\mathbb H)$ and let $r>0$, $\varepsilon >0$ be arbitrary. Let $n_0\in\mathbb N$ be such that $n_0>r$ and $1/2^{n_0-1} <\varepsilon$. Recalling that the sequence $\{P_n\}$ is dense
 in $\mathcal R(\mathbb H)$, there exists $n\in\mathbb N$, $n>n_0$ such that $\max_{| q| \leq n_0} | f(q)- P_n(q) |<\varepsilon$. The assumptions on the number $k_n$ imply that  $Q_j^{(k_n)}(q)=0$ for $j=1,\ldots , n-1$ and  $Q_n^{(k_n)}(q)=P_n(q)$, thus
 \[
 \begin{split}
 | f(q)-F^{(k_n)}(q)|  &=|  f(q)- P_n(q)  -  \sum_{j=n+1}^\infty  Q_j^{(k_n)}(q) |
 \\&
 \leq | f(q)-P_n(q)| +  \sum_{j=n+1}^\infty | Q_j^{(k_n)}(q) | \\
 &< 2\varepsilon ,
 \end{split}
 \]
 and so $$\max_{|q|\le n_0} | f(q)-F^{(k_n)}(q)|<2\varepsilon$$ and the statement follows.
\end{proof}

\section{Entire slice regular functions of exponential type}

As in the complex setting, it makes sense to introduce also in the quaternionic setting the notion of entire functions of exponential type. \index{exponential type} These are functions at most of first order and normal type and their exponential type is defined as
$$
\sigma=\overline{\lim}_{r\to \infty} \frac{\log M_f(r)}{r}.
$$
 In other words, we have the following:
\begin{definition}
An entire slice regular function $f$ is said to be of exponential type if there exist constants $A,B$ such that
$$
|f(q)|\leq B e^{A|q|}
$$
for all $q\in\mathbb H$.
\end{definition}
Functions of order less than $1$ or of order $1$ and minimal type, are said to be of exponential type zero.\index{exponential type zero}
As a nice application of our definition of the composition of the exponential function with a slice regular function, in particular a polynomial, we show how to associate to each entire slice regular function $f$ of exponential type its Borel transform.
\begin{definition}
Let $f$ be an entire slice regular function of exponential type, and let $f(q)=\sum_{k=0}^\infty q^k\dfrac{a_k}{k!}$, $a_k\in\mathbb H$. The function
$$
\phi(q)=\mathcal{F}(f)(q)=\sum_{k=0}^\infty q^{-(k+1)}a_k
$$
is called the Borel transform of $f(q)$.\index{Borel transform}
\end{definition}
\begin{remark}{\rm
If $\sigma$ is the exponential type of a function $f$, then
$$
\sigma={\underset{k\to\infty}{\overline{\lim}}}|a_k|^{1/k}.
$$
Thus the Borel transform $\phi$ of $f$ is slice regular for $|q|>\sigma$.
}
\end{remark}
Let $q,w\in\mathbb H$ and consider the function $$e_\star^{qw}=\sum_{n=0}^\infty \frac{1}{n!}q^nw^n.$$ This function can be obtained through the Representation Formula: take $z=x+I_w y$ on the same complex plane as $w$. Then $e_\star^{zw}=e^{zw}$ and we can extend this function for any $q=x+Iy\in\mathbb H$
$$
e_\star^{qw}={\rm ext}(e^{zw})=\frac 12(e^{zw}+e^{\bar{z}w})+\frac 12 I I_w (e^{\bar{z}w} - e^{zw}).
$$
This is useful to obtain an integral formula for an entire function $f$ in terms of its the Fourier-Borel transform.
\begin{theorem}
An entire slice regular function $f$ can be written in terms of its Fourier-Borel transform $\phi$ as
$$
f(q)=\frac{1}{2\pi} \int_{\Sigma\cap\mathbb C_I} e_\star^{qw} \, dw_I\, \phi (w),\qquad dw_I=-Idw,
$$
where $\Sigma$ surrounds the singularities of $\phi$.
\end{theorem}
\begin{proof}
First of all, we select a basis $1, I, J, IJ$ of $\mathbb H$, so that we can  write
\[
\begin{split}
f(q)&=\sum_{k=0}^\infty \frac{q^k}{k!}(a_{k0}+Ia_{k1}+Ja_{k2}+IJa_{k3})\\
&=\sum_{k=0}^\infty \frac{q^k}{k!}(a_{k0}+Ia_{k1})+\sum_{k=0}^\infty \frac{q^k}{k!}(a_{k2}+Ia_{k3})J\\
&=F(q)+G(q)J.
\end{split}
\]
The Fourier-Borel transform of $F$ and $G$ are
$$
 \mathcal{F} F(q)=\sum_{k=0}^\infty q^{-(k+1)}(a_{k0}+Ia_{k1}), \qquad \mathcal{F} G(q)= \sum_{k=0}^\infty q^{-(k+1)} ( a_{k2}+Ia_{k3}),
$$
and it is immediate that
$$
\phi(q)=\mathcal{F} f(q)= \mathcal{F} F(q)+\mathcal{F} G(q)J.
$$
When $z=x+Iy$ then $F$ and $G$ are entire holomorphic functions and so we have, for $w\in\mathbb C_I$
$$
F(z)=\frac{1}{2\pi}\int_{\gamma} e^{zw}\, \mathcal{F}(F)(w)\, dw I^{-1}=\frac{1}{2\pi}\int_{\gamma} e^{zw}\, dw_I \mathcal{F}(F)(w)
$$
$$
G(z)=\frac{1}{2\pi}\int_{\gamma} e^{zw}\, \mathcal{F}(G)(w)\, dw I^{-1}=\frac{1}{2\pi}\int_{\gamma} e^{zw}\, dw_I \mathcal{G}(F)(w)
$$
where $\gamma$ can be chosen to surround both the singularities of $\mathcal{F}(F)$ and $\mathcal{F}(G)$.
By the Splitting Lemma we have
\begin{equation}\label{expstar}
\begin{split}
f(z)&=F(z)+G(z)J\\
&=
\frac{1}{2\pi}\int_{\gamma} e^{zw}\, dw_I (\mathcal{F}(F)(w)+\mathcal{G}(F)(w)J)\\
&=\frac{1}{2\pi}\int_{\gamma} e^{zw}\, dw_I \phi (w).
\end{split}
\end{equation}
Now we can reconstruct $f(q)$ using the Representation Formula and the previous formula \eqref{expstar}. This leads to
\[
\begin{split}
f(q)&= f(x+I_q y)\\
&=\frac{1}{2\pi}\int_{\gamma} \left[\frac 12 (e^{zw}+e^{\bar zw})+\frac 12 I_qI(e^{\bar zw}-e^{zw})\right]\, dw_I \phi (w)\\
&=\frac{1}{2\pi}\int_{\gamma} e_\star^{qw} \, dw_I \phi (w),
\end{split}
\]
and since $w$ is arbitrary and $\gamma$ can be chosen to be $\Sigma\cap\mathbb C_I$, the statement follows.
\end{proof}
As we said at the beginning of the section, we introduced the Borel transform as an application of the slice regular composition of the exponential function. This is obviously the beginning of a theory that we plan to further develop.
\vskip 1 truecm
\noindent{\bf Comments to Chapter 5}. The material in this Chapter is new, except for the section on the universal property of entire functions which is taken from \cite{GalSa1}. The complex version of the results in this chapter may be found in the book \cite{levin}.






\begin{thebibliography}{99}

\bibitem{AACKS}
 K. Abu-Ghanem, D. Alpay, F. Colombo, D. P. Kimsey, I. Sabadini,
 {\em Boundary interpolation for slice hyperholomorphic Schur functions},
 Integral Equations Operator Theory, {\bf 82} (2015),  223--248.


\bibitem{AGALCOSA}
   K. Abu-Ghanem, D. Alpay, F. Colombo, I. Sabadini,
    {\em Gleason's problem and Schur multipliers in the multivariable quaternionic setting},
     J. Math. Anal. Appl., {\bf 425} (2015), 1083--1096.

\bibitem{adler}
 S. Adler, {\em Quaternionic Quantum Field Theory}, Oxford University Press,
1995.

\bibitem{ahlfors} L. V. Ahlfors, {\em Complex Analysis. An Introduction to the Theory of Analytic Functions of One Complex Variable}, Third
edition, McGraw-Hill, Inc., New York, 1979.

 \bibitem{MR2002b:47144}
D. Alpay,
{\em The {S}chur algorithm, reproducing kernel spaces and system
  theory},
 American Mathematical Society, Providence, RI, 2001,
Translated from the 1998 French original by Stephen S. Wilson,
  Panoramas et Synth\`eses.



\bibitem{abcs}
  D. Alpay, V. Bolotnikov, F. Colombo, I. Sabadini,
   {\em Self-mappings of the quaternionic unit ball: multiplier properties, the Schwarz-Pick inequality, and the Nevanlinna-Pick interpolation problem},
   Indiana Univ. Math. J., {\bf 64} (2015),  151--180.

   \bibitem{ALLJECK4}
D. Alpay, V. Bolotnikov, F. Colombo, I. Sabadini,
{\em Interpolation problems for certain classes of slice hyperholomorphic functios},
submitted.


 \bibitem{FUNCGEN}
D.~{Alpay}, F.~{Colombo}, J. Gantner, D. P. Kimsey,
{\em Functions of the infinitesimal generator of a strongly continuous  quaternionic group},
to appear in Anal. Appl. (Singap.).


\bibitem{acgs}
D.~{Alpay}, F.~{Colombo}, J. Gantner, I.~{Sabadini},
{\em A new resolvent equation for the $S$-functional calculus},
 J. Geom. Anal., {\bf 25} (2015), 1939--1968.

\bibitem{ack}
D.~{Alpay}, F.~{Colombo},  D. P. Kimsey,
{\em The spectral theorem for for quaternionic unbounded normal operators based on the $S$-spectrum},
 {Preprint 2014,} avaliable on arXiv:1409.7010.



\bibitem{acks2}
D. {Alpay}, F. {Colombo},  D. P. Kimsey, I. {Sabadini}.
{\em The spectral theorem for unitary operators based on the $S$-spectrum},
to appear in Milan J. Math. (2016).

\bibitem{ALLJECK3}
D. Alpay, F. Colombo, D. P. Kimsey, I. Sabadini,
{\em Wiener algebra for quaternions},
  to appear in Mediterranean Journal of Mathematics, (2016).

\bibitem{ACKS1}
   D. Alpay, F. Colombo, D. P. Kimsey, I. Sabadini,
   {\em An extension of Herglotz's theorem to the quaternions},
    J. Math. Anal. Appl., {\bf 421} (2015), 754--778.


\bibitem{acls_milan}
 D. Alpay,  F. Colombo,  I. Lewkowicz,  I. Sabadini,  {\em Realizations of
 slice hyperholomorphic generalized contractive and positive functions},
 Milan J. Math., {\bf 83} (2015),  91–-144.


\bibitem{acs1}
D. {Alpay}, F. {Colombo},  I. {Sabadini},
 {\em Schur functions and their realizations in the slice hyperholomorphic
  setting},
{Integral Equations Operator Theory}, {\bf 72}  (2012), 253--289.

\bibitem{MR3127378}
D. {Alpay}, F. {Colombo},  I. {Sabadini},
 {\em Pontryagin  De Branges Rovnyak spaces of slice hyperholomorphic functions},
J. Anal. Math., {\bf 121} (2013), 87-125.

\bibitem{acs3}
D. {Alpay}, F. {Colombo}, I. {Sabadini},
{\em Krein-Langer factorization and related topics in the slice
  hyperholomorphic setting},  J. Geom. Anal., {\bf 24} (2014),  843--872.


\bibitem{ACSMATMET}
 D. Alpay, F. Colombo, I. Sabadini,
  {\em On some notions of convergence for n-tuples of operators},
   Math. Methods Appl. Sci., {\bf 37} (2014),  2363--2371.



\bibitem{ALLJECK1}
D.~{Alpay}, F.~{Colombo},  I.~{Sabadini},
 {\em  Generalized quaternionic Schur functions in the ball and half-space and Krein-Langer factorization},
in Hypercomplex Analysis: New Perspectives and Applications, (2014), 19--41.



\bibitem{ALLJECK2}
D.~{Alpay}, F.~{Colombo},  I.~{Sabadini},
{\em Quaternionic Hardy spaces in the open unit ball and in the half space and Blaschke products},
in Proceedings 30th International Colloquium on Group Theoretical Methods (F. Brackx, H. De Schepper and J. Van der Jeugt, editors), Journal of Physics: Conference Series, {\bf 597} (2015), 012009.



\bibitem{ALLJECK5}
D. Alpay, F. Colombo, I. Sabadini,
{\em Inner product spaces and Krein spaces in the quaternionic setting}
 in Recent Advances in Inverse Scattering, Schur Analysis and Stochastic Processes
Operator Theory: Advances and Applications Volume 244 (2015), 33--65.


\bibitem{perturbation}
D.~{Alpay}, F.~{Colombo}, I.~{Sabadini},
 {\em Perturbation of the generator of a quaternionic evolution operator},
  Anal. Appl. (Singap.), {\bf 13} (2015), 347--370.







\bibitem{ACSBOOK} D. {Alpay}, F. {Colombo},  I. {Sabadini},
{\em Slice Hyperholomorphic Schur Analysis},
Quaderni Dipartimento di Matematica del Politecnico di Milano, QDD209, 2015 (Book preprint).

\bibitem{HINFTY}
D. Alpay, F. Colombo, I. Sabadini, T. Qian,
{\em The $H^\infty$ functional calculus based on the $S$-spectrum for quaternionic operators and for $n$-tuples of noncommuting operators},
submitted.
\bibitem{Fock}
D. {Alpay}, F. {Colombo},  I. {Sabadini}, G. Salomon,
{\em Fock space in the slice hyperholomorphic setting},
in Hypercomplex Analysis: New Perspectives and Applications, (2014), 43--59.



\bibitem{alt}
A. Altavilla,
 {\em Some properties for quaternionic slice regular functions on domains without real points},
  Complex Var. Elliptic Equ., {\bf 60} (2015), 59--77.


\bibitem{am} R. Aron, D. Markose, {\em On universal functions}, J. Korean Math. Soc., {\bf 41} (2004), 65-76.


\bibitem{arc}
 N. Arcozzi, G. Sarfatti,
 {\em Invariant metrics for the quaternionic Hardy space},
  J. Geom. Anal., {\bf 25} (2015), 2028--2059.

\bibitem{arc2}  N. Arcozzi, G. Sarfatti,
{\em From Hankel operators to Carleson measures in a quaternionic variable}, to appear in Proc. Edinburgh Math. Soc.


\bibitem{arc3}  N. Arcozzi, G. Sarfatti, {\em The orthogonal projection on slice functions on the quaternionic sphere}, in Proceedings 30th International Colloquium on Group Theoretical Methods (F. Brackx, H. De Schepper and J. Van der Jeugt, editors), Journal of Physics: Conference Series, Vol. 597.

\bibitem{birkhoff} G. D. Birkhoff, {\it D\'emonstration d'un th\'eor\`eme \'el\'ementaire sur les fonctions enti\`eres}, Comptes Rendus Acad. Sci., Paris, {\bf 189} (1929), 473--475.


\bibitem{BvN}
G. Birkhoff, J. von Neumann, {\em The logic of quantum mechanics},  Ann. of Math.,
 {\bf 37} (1936), 823-843.


\bibitem{BISISTOPP}
C. Bisi, C. Stoppato,
 {\em The Schwarz-Pick lemma for slice regular functions},
  Indiana Univ. Math. J., {\bf 61} (2012), 297--317.

\bibitem{boas} R. P. Boas, {\em Entire functions}, Academic Press Inc., New York, 1954.

\bibitem{bds} F. Brackx, R. Delanghe, F. Sommen, {\em Clifford Analysis},
Pitman Res. Notes in Math., 76, 1982.

\bibitem{bls} J. Bure\v{s}, R. L\'avi\v cka,  V. Sou\v{c}ek, {\em Elements of quaternionic analysis and Radon transform}, Textos de Matematica 42, Departamento de Matematica, Universidade de Coimbra, 2009.

\bibitem{cartan} H. Cartan, {\em Elementary Theory of Analytic Functions of One or Several Complex Variables}, Dover Publ. Inc., New York, 1995.

\bibitem{MR3311947}
C.M.P. Castillo Villalba, F. Colombo, J. Gantner, J.O. Gonz\'{a}lez-Cervantes,
{\em Bloch, Besov and Dirichlet Spaces of Slice Hyperholomorphic Functions},
Complex Anal. Oper. theory, {\bf 9} (2015), 479-517.


\bibitem{Chui-Parnes}
C. Chui, M. N. Parnes, {\it Approximation by overconvergence of a power series}, J. Math. Anal. Appl., {\bf 36}(1971), 693-696.

\bibitem{FRCPOW}
F. Colombo, J. Gantner,
{\em Fractional  powers of quaternionic operators and Kato's formula using slice hyperholomorphicity},
 arXiv:1506.01266.

\bibitem{Fconsequences}
F. Colombo, J. Gantner,
{\em  Formulations of the  F-functional calculus and some consequences},
to appear in  Proc. Roy. Soc. Edinburgh Sect. A.

\bibitem{Taylor}
F. Colombo, J. Gantner,
{\em  On power series expansions of the S-resolvent operator and the Taylor formula},
  arXiv:1501.07055.



\bibitem{CGJ1}
F. Colombo, J. Gantner, T. Janssens,
{\em The Schatten class  and the Berezin transform for quaternionic operators},
submitted.

\bibitem{CGJ2}
F. Colombo, J. Gantner, T. Janssens,
{\em Berezin transform of slice hyperholomorphic functions},
submitted.

\bibitem{cgssann}
F. Colombo, G. Gentili, I. Sabadini, D.C. Struppa
{\em A functional calculus in a  non commutative  setting},
 Electron. Res. Announc. Math. Sci., {\bf 14} (2007), 60-68.

\bibitem{MR2555912}
     F. Colombo, G. Gentili, I. Sabadini, D.C. Struppa,
      {\em Extension results for slice regular functions of a quaternionic variable},
       Adv. Math.,
        {\bf 222} (2009), 1793--1808.



\bibitem{CGeSA}
    F. Colombo, G. Gentili, I. Sabadini,
    {\em A Cauchy kernel for slice regular functions},
     Ann. Global Anal. Geom., {\bf 37} (2010), 361--378.

\bibitem{cgss}
F. Colombo, G. Gentili, I. Sabadini, D.C. Struppa,
{\em Non commutative functional calculus:  unbounded operators},
 J. Geom. Phys, {\bf 60} (2010), 251--259.

\bibitem{cgss1}
F. Colombo, G. Gentili, I. Sabadini, D.C. Struppa,
{\em Non commutative functional calculus:  bounded operators},
   Complex Anal. Oper. Theory, {\bf 4} (2010), 821--843.


\bibitem{cglss}
  F. Colombo, J. O. Gonz\'alez-Cervantes,  M. E. Luna-Elizarraras, I. Sabadini, M. Shapiro,
  \emph{  On two approaches to the Bergman theory for slice regular functions},
   Advances in hypercomplex analysis, 39–54, Springer INdAM Ser., 1, Springer, Milan, 2013.


\bibitem{CGONZSAB}
 F. Colombo, J. O. Gonz\'alez-Cervantes , I. Sabadini,
 {\em The Bergman-Sce transform for slice monogenic functions},
 Math. Methods Appl. Sci., {\bf 34} (2011),  1896--1909.



   \bibitem{CGS}  F. Colombo, J. O. Gonz\'alez-Cervantes, I. Sabadini,
\emph{ On slice biregular functions and isomorphisms of Bergman spaces},
 Complex Var. Elliptic Equ., {\bf 57} (2012),  825--839.

\bibitem{CGOZSAB}
 F. Colombo, J. O. Gonz\'alez-Cervantes,  I. Sabadini,
 {\em A nonconstant coefficients differential operator associated to slice monogenic functions},
  Trans. Amer. Math. Soc., {\bf 365} (2013), 303--318.


 \bibitem{CGS3} F. Colombo, J. O. Gonz\'alez-Cervantes, I. Sabadini,
\emph{The C-property for slice regular functions and applications to the Bergman space},
Compl. Var. Ell. Equa., {\bf 58}  (2013), 1355--1372.

\bibitem{moscow} F. Colombo, J. O. Gonzalez-Cervantes, I. Sabadini, {\em Some integral representations of slice hyperholomorphic functions}. Mosc. Math. J. {\bf 14} (2014), 473--489.


\bibitem{CLSSo}
 F. Colombo, R. Lavicka, I. Sabadini, V.  Soucek,
  {\em The Radon transform between monogenic and generalized slice monogenic functions},
   Math. Ann., {\bf 363} (2015),  733--752.

\bibitem{WCALCOLO}
F. Colombo, R. Lavicka, I. Sabadini, V. Soucek,
{\em  Monogenic plane waves and the W-functional calculus},
  to appear in  Math. Methods Appl. Sci..

\bibitem{MR2742644}
     F. Colombo,  I. Sabadini,
      {\em A structure formula for slice monogenic functions and some of its consequences},
       Hypercomplex analysis, 101--114, Trends Math., Birkh\"auser Verlag, Basel, 2009.

\bibitem{JGA} F. Colombo, I. Sabadini,
{\em On some properties of the quaternionic functional calculus},
J. Geom. Anal., {\bf 19}  (2009), 601-627.

\bibitem{CLOSED} F. Colombo, I. Sabadini,
{\em  On the  formulations of the quaternionic functional calculus},
 J. Geom. Phys., {\bf 60} (2010), 1490--1508.


\bibitem{Cauchy}
F. Colombo,  I. Sabadini,
{\em The Cauchy formula with $s$-monogenic kernel and a functional calculus for noncommuting operators},
  J. Math. Anal. Appl, {\bf 373} (2011), 655--679.

\bibitem{MR2803786}
F. Colombo, I. Sabadini,
{\em  The quaternionic evolution operator},
 Adv. Math., {\bf 227} (2011), 1772--1805.


\bibitem{Boundedpert}
F. Colombo, I. Sabadini,
{\em Bounded perturbations of the resolvent operators associated to the  $\mathcal{F}$-spectrum},
Hypercomplex Analysis and Applications, Trends in Mathematics, Birkhauser, (2011), 13-28.



\bibitem{COSASO}
 F. Colombo, I. Sabadini, F. Sommen,
 {\em The inverse Fueter mapping theorem},
  Commun. Pure Appl. Anal., {\bf 10} (2011), 1165--1181.




\bibitem{SCFCALC}
F. Colombo, I. Sabadini,
{\em The  $\mathcal{F}$-spectrum and the $\mathcal{SC}$-functional calculus},
  Proc. Roy. Soc. Edinburgh Sect. A, {\bf 142} (2012), 479--500.

\bibitem{FABIRE}
 F. Colombo, I. Sabadini,
 {\em  On some notions of spectra for quaternionic operators and for n-tuples of operators},
  C. R. Math. Acad. Sci. Paris, {\bf 350} (2012),  399--402.

\bibitem{FUNBOUNDED}
F. Colombo, I. Sabadini,
{\em The F-functional calculus for unbounded operators},
 J. Geom. Phys., {\bf 86} (2014), 392--407.


\bibitem{FFRANK}
F. Colombo, I. Sabadini, F. Sommen,
{\em The Fueter mapping theorem in integral form and  the $\mathcal{F}$-functional calculus},
 Math. Methods Appl. Sci., {\bf 33} (2010), 2050--2066.


\bibitem{csss} F. Colombo, I. Sabadini, F. Sommen, D.C.
Struppa, {\em Analysis of Dirac Systems and Computational Algebra},
Progress in Mathematical Physics, Vol. 39, {Birkh\"auser}, Boston,
2004.

\bibitem{CoSo}
F. Colombo, F. Sommen,
{\em  Distributions and the global operator of slice monogenic functions},
  Complex Anal. Oper. Theory,
{\bf 8} (2014),  1257--1268.


\bibitem{JFA}
F. Colombo,  I. Sabadini, D.C. Struppa,
{\em A new functional calculus for non commuting operators},
  J. Funct. Anal., {\bf 254} (2008), 2255--2274.

   \bibitem{MR2520116}
        F. Colombo,  I. Sabadini, D.C. Struppa,
       {\em Slice monogenic functions},
        Israel J. Math., {\bf 171} (2009), 385--403.

\bibitem{CSSd}
  F. Colombo,  I. Sabadini, D.C. Struppa,
 {\em An extension theorem for slice monogenic functions and some of its consequences},
  Israel J. Math., {\bf 177} (2010), 369--389.

\bibitem{duality}
   F. Colombo,  I. Sabadini, D.C. Struppa,
   {\em Duality theorems for slice hyperholomorphic functions},
   J. Reine Angew. Math., {\bf 645} (2010), 85--105.



\bibitem{MR2785869}
F. Colombo,  I. Sabadini, D.C. Struppa,
 {\em The Pompeiu formula for slice hyperholomorphic functions},
 Michigan Math. J., {\em 60} (2011), 163--170.


\bibitem{CSS}
F. Colombo, I. Sabadini, D.~C. Struppa,
 {\em Noncommutative functional calculus. Theory and applications of slice regular functions},
 volume 289 of {\em Progress
  in Mathematics}.
 Birkh\"auser/Springer Basel AG, Basel, 2011.


\bibitem{runge}
 F. Colombo,  I. Sabadini, D.C. Struppa,
 {\em The Runge theorem for slice hyperholomorphic functions},
 Proc. Amer. Math. Soc., {\bf 139} (2011),  1787--1803.

\bibitem{sheaves}
  F. Colombo,  I. Sabadini, D.C. Struppa,
 {\em  Sheaves of slice regular functions},
  Math. Nachr., {\bf 285} (2012), 949--958.


\bibitem{CSSsev}
  F. Colombo,  I. Sabadini, D.C. Struppa,
  {\em Algebraic properties of the module of slice regular functions in several quaternionic variables},
   Indiana Univ. Math. J., {\bf 61} (2012), 1581--1602.



\bibitem{CDEBR}
L. Cnudde, H. De Bie, G. Ren,
{\em Algebraic approach to slice monogenic functions},
 Complex Anal. Oper. Theory, {\bf 9} (2015), 1065--1087.

\bibitem{CDB} L. Cnudde, H. De Bie,
    {\em Slice Fourier transform and convolutions},
    arXiv:1511.05014.

\bibitem{debranges} L. de Branges, {\em New and Old Problems for Entire Functions}, Bull. A.M.S. 1964.

\bibitem{deleo} S. De Leo, P. P. Rotelli, {\em Quaternionic analyticity}, Appl. Math. Lett. {\bf 16} (2003),  1077--1081.

\bibitem{DeSoS}
R. Delanghe, F. Sommen, V. Soucek,
 {\em Clifford algebra and spinor-valued functions. A function theory for the Dirac operator},
  Related REDUCE software by F. Brackx and D. Constales. With 1 IBM-PC floppy disk (3.5 inch).
   Mathematics and its Applications, 53. Kluwer Academic Publishers Group, Dordrecht, 1992. xviii+485 pp.


\bibitem{DRGS}
 C. Della Rocchetta, G. Gentili, G. Sarfatti,
  {\em The Bohr theorem for slice regular functions},
   Math. Nachr., {\bf 285} (2012),  2093--2105.

\bibitem{ds}
N. Dunford, J. Schwartz. {\it Linear Operators, part I: General
Theory }, J. Wiley and Sons (1988).


\bibitem{duren} P. L. Duren, {\em Univalent Functions},  Grundlehren der Mathematischen Wissenschaften [Fundamental Principles of Mathematical Sciences], 259, Springer-Verlag, New York, 1983.


\bibitem{12}
 G. Emch, {\em M\'ecanique quantique quaternionienne et relativit\'e restreinte},  I, Helv. Phys.
Acta,  {\bf 36} (1963), 739--769.


\bibitem{fp}
D. R. Farenick, B. A. F. Pidkowich, {\em The spectral theorem in quaternions},
Linear Algebra Appl.,
{\bf 371} (2003), 75--102.



\bibitem{14}
 D. Finkelstein, J. M. Jauch, S. Schiminovich,  D. Speiser, {\em Foundations of
quaternion quantum mechanics}, J. Mathematical Phys., {\bf 3} (1962), 207--220.


 \bibitem{fliess}
M.~Fliess, {\em Matrices de {H}ankel},
\newblock {\em J. Math. Pures Appl. (9)}, {\bf 53} (1974), 197--222.


\bibitem{ggs}
S.G. Gal, O.J. González-Cervantes, I. Sabadini,
 {\em On some geometric properties of slice regular functions of a quaternion variable},
  Complex Var. Elliptic Equ., {\bf 60} (2015), 1431--1455.

\bibitem{GGONS2}
  S.G. Gal, O.J. González-Cervantes, I. Sabadini,
   {\em Univalence results for slice regular functions of a quaternion variable},
    Complex Var. Elliptic Equ., {\bf 60} (2015),  1346--1365.

\bibitem{GALSAB1}
S.G. Gal,  I. Sabadini,
 {\em Approximation in compact balls by convolution operators of quaternion and paravector variable},
 Bull. Belg. Math. Soc. Simon Stevin, {\bf 20} (2013),   481--501.

\bibitem{GALSAB2}
S.G. Gal,  I. Sabadini,
{\em Carleman type approximation theorem in the quaternionic setting and applications},
Bull. Belg. Math. Soc. Simon Stevin, {\bf 21} (2014), 231--240.


\bibitem{GALSAB3}
S.G. Gal,  I. Sabadini,
{\em Walsh equiconvergence theorems in the quaternionic setting},
 Complex Var. Elliptic Equ., {\bf 59} (2014),  1589--1607.


\bibitem{GALSAB4}
S.G. Gal,  I. Sabadini,
{\em On Bernstein and Erd\H os-Lax's inequalities for quaternionic polynomials},
 C. R. Math. Acad. Sci. Paris, {\bf 353} (2015), 5--9.

\bibitem{GALSAB5}
S.G. Gal,  I. Sabadini,
{\em Arakelian's approximation theorem of Runge type in the hypercomplex setting},
 Indag. Math. (N.S.), {\bf 26} (2015),  337--345.

\bibitem{GalSa1}
S.G. Gal,  I. Sabadini,
{\em Universality properties of the quaternionic power series and entire functions},
 Math. Nachr., {\bf 288} (2015),  917--924.

\bibitem{galsa3}
S.G. Gal,  I. Sabadini,
 {\em Approximation by polynomials on quaternionic compact sets},
  Math. Methods Appl. Sci., {\bf 38} (2015), 3063--3074.



\bibitem{GSalS}
 G. Gentili, S. Salamon, C. Stoppato,
  {\em Twistor transforms of quaternionic functions and orthogonal complex structures},
   J. Eur. Math. Soc. (JEMS), {\bf 16} (2014), 2323--2353.

\bibitem{GSARF}
 G. Gentili, G. Sarfatti,
 {\em Landau-Toeplitz theorems for slice regular functions over quaternions},
 Pacific J. Math., {\bf 265} (2013),  381--404.



\bibitem{GSTO0}
 G. Gentili, C. Stoppato,
 {\em Zeros of regular functions and polynomials of a quaternionic variable}, Michigan Math. J., {\bf 56} (2008),  655--667.

\bibitem{GSTO1}
 G. Gentili, C. Stoppato,
 {\em  The open mapping theorem for regular quaternionic functions},
  Ann. Sc. Norm. Super. Pisa Cl. Sci., {\bf 8} (2009),  805--815.


\bibitem{zeri} G. Gentili, C. Stoppato, {\em The zero sets of slice regular functions and the open mapping theorem}, Hypercomplex analysis and applications, 95–107, Trends Math., Birkh\"user/Springer Basel AG, Basel, 2011.

\bibitem{GSTO3}
 G. Gentili, C. Stoppato,
 {\em Power series and analyticity over the quaternions},
  Math. Ann., {\bf 352} (2012), 113--131.



\bibitem{GStruPRAG}
 G. Gentili, C. Stoppato, D. C. Struppa,
  {\em A Phragmen-Lindelof principle for slice regular functions},
   Bull. Belg. Math. Soc. Simon Stevin, {\bf 18} (2011), 749--759.

\bibitem{GSS}
 G. Gentili, C. Stoppato, D. C.  Struppa, {\em Regular functions of a quaternionic variable}.
  Springer Monographs in Mathematics. Springer, Heidelberg, 2013.



\bibitem{MR2227751}
 G. Gentili, D. C. Struppa,
  {\em A new approach to Cullen-regular functions of a quaternionic variable},
   C. R. Math. Acad. Sci. Paris, {\bf 342} (2006), 741--744.

\bibitem{GS}
  G. Gentili, D. C. Struppa,
  {\em A new theory of regular functions of a quaternionic variable},
   Adv. Math., {\bf 216} (2007), 279--301.


 \bibitem{mjm} G. Gentili, D.C. Struppa, {\em The multiplicity of the zeros of polynomials with quaternionic coefficients}, Milan J. Math., {\bf 216} (2008), 15--25.



\bibitem{MR2869150}
G.~Gentili, D.C. Struppa.
\newblock Lower bounds for polynomials of a quaternionic variable.
\newblock {\em Proc. Amer. Math. Soc.}, 140(5):1659--1668, 2012.



\bibitem{MR2836832}
G.~Gentili, I.~Vignozzi.
\newblock The {W}eierstrass factorization theorem for slice regular functions
  over the quaternions.
\newblock {\em Ann. Global Anal. Geom.}, 40(4):435--466, 2011.



\bibitem{GMP} R. Ghiloni, V. Moretti, A. Perotti,
{\em Continuous slice functional calculus in quaternionic Hilbert spaces},
Rev. Math. Phys.,  {\bf 25} (2013), 1350006, 83 pp.



\bibitem{spectcomp} R. Ghiloni, V. Moretti, A. Perotti,
{\em Spectral properties of compact normal quaternionic operators},
 in Hypercomplex Analysis: New Perspectives and Applications
Trends in Mathematics, 133--143,  (2014).



\bibitem{GP1}
 R. Ghiloni, A. Perotti,
 {\em A new approach to slice regularity on real algebras},
In Hypercomplex analysis and applications, Trends Math., pages 109--123.
Birkh¨auser/Springer Basel AG, Basel, 2011.



\bibitem{gp} R. Ghiloni, A. Perotti, {\em Slice regular functions on real alternative
algebras}, Adv. Math., {\bf 226} (2011), 1662-1691.

\bibitem{GP3}
 R. Ghiloni, A. Perotti,
 {\em  Zeros of regular functions of quaternionic and octonionic variable: a division lemma and the camshaft effect},
Ann. Mat. Pura Appl.,
{\bf 190} (2011), 539--551.

\bibitem{GP4}
 R. Ghiloni, A. Perotti,
 {\em  Volume Cauchy formulas for slice functions on real
associative $*$-algebras},
 Complex Var. Elliptic Equ., {\bf 58} (2013), 1701--1714.

\bibitem{GP5}
 R. Ghiloni, A. Perotti,
 {\em  Global differential equations for slice regular functions},
Math. Nachr., {\bf 287} (2014), 561–573.


\bibitem{gp1} R. Ghiloni, A. Perotti, {\em Power and spherical series over real alternative $*$-algebras}
 Indiana Univ. Math. Journal, {\bf 63} (2014), 495--532.

\bibitem{GP7}
R. Ghiloni, A. Perotti,
{\em   Lagrange polynomials over Clifford numbers},
 J. Algebra Appl. {\bf 14} (2015), no. 5, 1550069, 11 pp.
\bibitem{GPS1} R. Ghiloni, A. Perotti, C. Stoppato
    {\em The algebra of slice functions}, to appear in Trans. of Amer. Math. Soc., DOI:10.1090/tran/6816.


\bibitem{GR}
R. Ghiloni, V. Recupero,
{\em Semigroups over real alternative *-algebras: generation
theorems and spherical sectorial operators},
 Trans. Amer. Math. Soc. (In Press).

\bibitem{gil} M. Gil, {\em Localization and perturbation of zeros of entire functions}, Lectures Notes in Pure and Applied Mathematics, 258, CRC Press Boca Raton, 2010.


\bibitem{ghs} K. G\"urlebeck, K. Habetha, W. Spr\"o\ss ig, { Holomorphic Functions in the Plane and $n$-Dimensional Space}, Birkh\"auser, Basel, 2008.




\bibitem{21}
L. P. Horwitz, L. C. Biedenharn, {\em Quaternion quantum mechanics: Second quantization
and gauge fields}, Annals of Physics, {\bf 157} (1984), 432Â„1â‚¬7488.



\bibitem{jefferies} B. Jefferies, {\em Spectral properties of noncommuting operators},
Lecture Notes in Mathematics, 1843, Springer-Verlag, Berlin, 2004.


\bibitem{jmc} B. Jefferies, A. McIntosh, {\em The Weyl calculus and Clifford analysis},
Bull. Austral. Math. Soc., {\bf 57} (1998), 329--341.

\bibitem{jmcpw} B. Jefferies, A. McIntosh, J. Picton-Warlow, {\em The monogenic functional
calculus}, Studia Math., {\bf 136} (1999), 99--119.


\bibitem{kaneko} A. Kaneko, {\em Introduction to Hyperfunctions}, Kluwer Academic Publishers, 1988.


\bibitem{lam} T. Y. Lam, {\em A first course in noncommutative rings}, Springer-Verlag, New York,
1991.

\bibitem{levin} B. Ja. Levin, {\em Distribution of zeros of entire functions}, Revised edition. Translations of Mathematical Monographs, 5. American Mathematical Society, Providence, R.I., 1980.


\bibitem{LIATAO}
C. Li, A. McIntosh, T. Qian, {\em Clifford algebras, Fourier transforms and singular convolution operators on Lipschitz surfaces}, Rev. Mat. Iberoamericana, {\bf 10} (1994),  665--721.


\bibitem{Luh}
W. Luh, {\it Approximation analytischer Functionen durch \"uberkonvergente Potenzreihen
und deren Matrix-Transformierten}, Mitt. Math. Sem. Giessen, {\bf 88} (1970), 1-56.



\bibitem{maclane}
G. MacLane, {\em Sequences of derivatives and normal families}, J. Anal. Math., {\bf 2} (1952), 72--87.

\bibitem{McI1}
A. McIntosh,
{\em Operators which have an $H^\infty$ functional calculus}.
 Miniconference on operator theory and partial differential equations (North Ryde, 1986), 210--231,
Proc. Centre Math. Anal. Austral. Nat. Univ., 14, Austral. Nat. Univ., Canberra, (1986).


\bibitem{mcp} A. McIntosh, A. Pryde, {\em A functional calculus for several commuting
operators}, Indiana U. Math. J., {\bf 36} (1987), 421--439.


\bibitem{Fek}
J. P\'al, {\it Zwei kleine Bemerkungen}, Tohoku Math. J., {\bf 6} (1914-1915), 42-43.


\bibitem{RW}
G. Ren, X. Wang,
{\em  Carath\'eodory theorems for slice regular functions},
Complex Anal. Oper. Theory, {\bf 9} (2015), 1229--1243.


\bibitem{RZ}
G. Ren, Z. Xu, {\em  Schwarz's Lemma for Slice Clifford Analysis},
 Adv. Appl. Clifford Algebr., {\bf 25} (2015), 965--976.

\bibitem{rinehart} R.F. Rinehart, {\em Elements of a theory of
intrinsic functions on algebras}, Duke Math. J. {\bf 27} (1960),
1-19.

\bibitem{rudin} W. Rudin, Real and Complex Analysis, McGraw-Hill, Singapore, 1986.


\bibitem{sarfatti} G. Sarfatti, \emph{Elements of function theory in the unit ball of quaternions},
 PhD thesis, Universit\'a di Firenze, 2013.

\bibitem{sarfatti1} G. Sarfatti, {\em Quaternionic Hankel operators and approximation by slice regular functions},to appear in Indiana University Mathematics Journal.

\bibitem{seidel walsh}
 W. Seidel, J.L. Walsh,  {\em On approximation by euclidean and non-euclidean translations of an analytic function}, Bull. Amer. Math. Soc. {\bf 47} (1941), 916--920.

\bibitem{Sel}
A. I. Seleznev, {\it On universal power series} (in Russian), Mat. Sb.(N.S.), {\bf 28}(1951), 453-460.

\bibitem{So1} F. Sommen, {\em An extension of the Radon transform to Clifford analysis,} Complex variables
8 (1987) 243-266.

\bibitem{sommen} F. Sommen, {\em On a generalization of Fueter's theorem} Z. Anal. Anwendungen, {\bf 19} (2000),  899--902.


  \bibitem{STOPP3}
  C. Stoppato,
  {\em Regular Moebius transformations of the space of quaternions},
  Ann. Global Anal. Geom., {\bf 39} (2011), 387--401.



 \bibitem{STOPP1}
 C. Stoppato
 {\em A new series expansion for slice regular functions},
  Adv. Math., {\bf 231} (2012), 1401--1416.

  \bibitem{STOPP2}
  C. Stoppato,
  {\em Singularities of slice regular functions},
   Math. Nachr., {\bf 285} (2012), 1274--1293.

\bibitem{taylor} B.A. Taylor, {\em Some locally convex spaces of entire functions},
 1968 Entire Functions and Related Parts of Analysis
 (Proc. Sympos. Pure Math., La Jolla, Calif., 1966) pp. 431-467 Amer. Math. Soc., Providence, R.I.

 \bibitem{vlac}
F. Vlacci,
 {\em  The argument principle for quaternionic slice regular functions},
  Michigan Math. J., {\bf 60} (2011),  67--77.

 \bibitem{YQ}
Y. Yang, T. Qian,
{\em
 Zeroes of slice monogenic functions},
  Math. Methods Appl. Sci. {\bf 34} (2011), 1398--1405.


\end{thebibliography}
\end{document}